\documentclass[12pt]{amsart}
\usepackage{amsmath, amssymb, amsthm, mathrsfs, textcomp, colonequals, comment,soul}
\usepackage{epic}
\usepackage[shortlabels]{enumitem}
\usepackage{xypic}
\usepackage[backref]{hyperref}
\usepackage[alphabetic,backrefs,lite]{amsrefs}
\usepackage[usenames,dvipsnames]{color}
\usepackage{fullpage}
\usepackage[normalem]{ulem}
\usepackage{microtype}

\hypersetup{
  colorlinks   = true, %Colours links instead of ugly boxes
  urlcolor     = blue, %Colour for external hyperlinks
  linkcolor    = blue, %Colour of internal links
  citecolor   = blue %Colour of citations
}

\numberwithin{equation}{section}
\allowdisplaybreaks[1]
\newtheorem{theorem}[equation]{Theorem}

\newtheorem{corollary}[equation]{Corollary}
\newtheorem{prop}[equation]{Proposition}
\newtheorem{propdef}[equation]{Proposition/Definition}
\newtheorem{lemma}[equation]{Lemma}

\theoremstyle{definition}

\newtheorem{remark}[equation]{Remark}

\newtheorem{defn}[equation]{Definition}

\newtheorem{algo}[equation]{Algorithm}
\newtheorem{example}[equation]{Example}

\setenumerate{itemsep=2pt}

\newcommand{\padma}[1]{{\color{Magenta} \sf $\clubsuit\clubsuit\clubsuit$ Padma: [#1]}}
\newcommand{\andrew}[1]{{\color{Green} \sf $\diamondsuit\diamondsuit\diamondsuit$ Andrew: [#1]}}
\newcommand{\ol}[1]{\overline{#1}}
\newcommand{\mc}[1]{\mathcal{#1}}
\newcommand{\mf}[1]{\mathfrak{#1}}

\newcommand{\Q}{\mathbb Q}
\newcommand{\rats}{\mathbb Q}
\newcommand{\Z}{\mathbb Z}
\newcommand{\ints}{\mathbb Z}

\newcommand{\nats}{\mathbb N}

\renewcommand{\O}{\mathcal{O}}

\newcommand{\aff}{\mathbb A}

\renewcommand{\P}{\mathbb P}
\newcommand{\proj}{\mathbb P}

\renewcommand{\phi}{\varphi}

\newcommand{\Spec}{\mathrm{Spec} \ }

\DeclareMathOperator{\Gal}{Gal}

\DeclareMathOperator{\GL}{GL}
\DeclareMathOperator{\divi}{div}

\DeclareMathOperator{\reg}{reg}

\DeclareMathOperator{\cha}{char}

\DeclareMathOperator{\Aut}{Aut}

\DeclareMathOperator{\lcm}{lcm}

\DeclareMathOperator{\Proj}{Proj}

\DeclareMathOperator{\red}{red}
\DeclareMathOperator{\rad}{rad}
\DeclareMathOperator{\Div}{Div}
\DeclareMathOperator{\Frac}{Frac}
\DeclareMathOperator{\depth}{depth}
\DeclareMathOperator{\height}{ht}
\DeclareMathOperator{\chara}{char}

\title{Minimal regular normal crossings models of superelliptic curves}

\author{Andrew Obus}
\address{Baruch College / CUNY Graduate Center}
\curraddr{1 Bernard Baruch Way. New York, NY 10010, USA}
\email{andrewobus@gmail.com}
\author{Padmavathi Srinivasan}
\address{Boston University}
\curraddr{665 Commonwealth Avenue, Boston, MA , USA}
\email{padmask@bu.edu}

\date{\today}

\begin{document}

\maketitle

\begin{abstract}
Let $K$ be a complete discretely valued field with perfect residue
field $k$. If $X \to \proj^1_K$ is a $\ints/d$-cover with $\chara k
\nmid d$, we compute the minimal regular normal crossings model
$\mc{X}$ of $X$ as the normalization of an explicit normal model $\mc{Y}$ of
$\mathbb{P}^1_K$ in $K(X)$.  The model $\mc{Y}$ is given using Mac Lane's description of discrete valuations on the rational function field $K(\mathbb{P}^1)$.
\end{abstract}

%\tableofcontents

\section{Introduction}\label{Sintro}
Let $K$ be a complete discretely valued field with perfect residue
field $k$ and valuation ring $\mc{O}_K$. Let $X$ be a smooth
projective geometrically integral curve over $K$. A model for $X$ is a proper flat $\mc{O}_K$-scheme with generic fiber $X$. For many arithmetic applications, one needs explicit descriptions of models of $X$ that are ``as close to smooth'' as possible. For example, bounds on the number of rational points for a curve $X$ over a number field via the effective Chabauty--Kim method require good bounds on the number of components in the special fiber of a minimal regular model at every place of the number field, see for example \cite[Theorem~B]{Betts}. Our main theorem is the following:
%In this paper, we explicitly describe the minimal regular normal crossings model for $\Z/d$ covers $X$ of $\mathbb{P}^1_K$ when $d$ is coprime to $\chara k$. }
%(that is, a regular model for which the reduced induced subscheme of the special fiber has only ordinary double points for singularities) 

\begin{theorem}\label{Tmain}(See Theorem~\ref{Tminnormalcrossingsbase})
Let $X \to \proj^1_K$ be a $\ints/d$-cover, with $\cha k \nmid d$.  There is an explicit normal model $\mc{Y}$ of $\proj^1_K$
whose normalization in $K(X)$ is the minimal normal crossings regular
model of $X$.  
%The algorithm for constructing $\mc{Y}$ is given in Theorem~\ref{Tminnormalcrossingsbase}. 
\end{theorem}

Let us elaborate on what ``explicit'' means.
There are two natural approaches to computing regular models of
curves. One approach is to first construct a semistable model of $X
\times_K K'$ over a Galois extension $K'/K$ (the construction
  of which is well-known in the case of $\ints/d$-covers of $\proj^1_K$, see e.g., \cite{BW}), take the quotient by $\Gal(K'/K)$ to create a normal model $\mc{X}'$ of $X$ over $K$, and 
%attempt to 
explicitly resolve singularities on this normal model. This can be
difficult, since wild quotient singularities may appear even though
$\chara k \nmid d$. Another approach is to start with a singular
$\mc{O}_K$-model of $X$ and run Lipman's algorithm for resolving singularities. This involves recursively computing invariants of singularities in coordinate charts of blow-ups and computing normalizations, which can be hard in practice. We give a third non-recursive approach that uses the defining equation of the cover to describe the normal model as an explicit set of extensions of discrete valuations on the function field $K(\mathbb{P}^1)$ corresponding to the irreducible components in the special fiber. 
%The description is also particularly well-suited to computing specializations of points on the generic fiber in the model, which might be useful for arithmetic applications (may omit this sentence).  
Such descriptions of discrete valuations are already available in Sage
\cite{RuthSage}, where we hope to include an implementation of our
algorithm in the future. Along the way, Algorithm~\ref{AY0} provides a similar
description of a \emph{strict}\footnote{that is, all the
irreducible components of the reduced special fiber are smooth.}
regular normal crossings model in Theorem~\ref{TX0regular}, which is not
necessarily (but often) minimal. About a third of the paper
(\S\ref{Sinfty}, \S\ref{Sinftycrossing}, and \S\ref{Ssecondmodel}) is devoted to identifying contractible components in this model, and can be skipped on a first reading.

\subsection{Explicit models of $\proj^1_K$ via Mac Lane valuations}\label{Sexplicit}
%One knows (see \S\ref{Smaclanemodels}) that 
Normal models of
$\proj^1_K$ are in one-to-one correspondence with finite sets of
so-called ``geometric valuations'' on $K(t)$, with each irreducible component of
the special fiber of the normal model corresponding to one valuation
in the set (see \S\ref{Smaclanemodels}).  Geometric valuations correspond to Type 2 points
of the Berkovich space $(\proj^1_K)^{\rm Berk}$ over $K$; namely, they are
valuations on $K(t)$ extending the valuation on $K$ whose residue
field has transcendence degree $1$ over $k$.  In \cite{MacLane}, Mac
Lane introduced an explicit notation to write down geometric valuations,
which involves writing down only finitely many polynomials and rational
numbers.  Geometric valuations are also called ``Mac Lane valuations'' in
his honor.  In Theorem~\ref{Tminnormalcrossingsbase}, we present the
model $\mc{Y}$ from Theorem~\ref{Tmain} as a finite set of Mac Lane
valuations.  This has advantages beyond ease of presentation and not
having to work with charts and blow ups.  Mac Lane
valuations are very well suited to computing multiplicities of
components in models of covers of $\proj^1_K$, and they are also well
suited to computing divisors of rational functions on such models.  In
particular, an important test for regularity for us will be to check
whether certain vertical divisors on models of $X$ are locally
principal, and these computations are naturally facilitated using Mac Lane valuations.

\begin{remark}
The model $\mc{Y}$ in Theorem~\ref{Tmain} always has normal crossings, and it is immediate to
read off the dual graph of the special fiber $\ol{Y}$ of $\mc{Y}$ as well as the
multiplicities of the irreducible components of $\ol{Y}$ from the
description of $\mc{Y}$ in terms of
Mac Lane valuations.  We content ourselves in this paper with a
description of $\mc{Y}$, rather than explicitly writing down the dual
graph and multiplicities of the resulting normal crossings regular
model $\mc{X}$ of $X$.  In any individual case, it is not hard to
write down this description of $\mc{X}$ given $\mc{Y}$, see e.g., Example~\ref{Ebasic}.  
\end{remark}

\subsection{A high-level summary of Algorithm~\ref{AY0} and the proof
  of Theorem~\ref{Tminnormalcrossingsbase}}\label{Sidea}

Recall that Algorithm~\ref{AY0} produces a normal model $\mc{Y}$ of $\proj^1_K$
whose normalization $\mc{X}$ in $K(X)$ is regular with strict normal crossings.  Let
$\pi_K$ be a uniformizer of $K$, and suppose $f = \pi_K^a f_1^{a_1}\cdots
f_r^{a_r}$ is an irreducible factorization of $f$ where the $f_i$ are monic
polynomials in $\mc{O}_K[t]$ and $a$ and the $a_i$ are nonnegative integers. To build $\mc{Y}$, analogously to the
semistable case in \cite{BW}, we
start by creating a normal crossings normal model $\mc{Y}'$ of $\proj^1_K$ on which the
horizontal divisors of the zeros of the $f_i$ are regular and do not
meet.
On the valuation side, this requires including a certain valuation
$v_{f_i}$ for each $f_i$ (this is the unique valuation over which
$f_i$ is a so-called ``key polynomial''), then including all
``predecessors'' of the $v_{f_i}$, and then throwing in enough
other valuations so that the set is closed under taking infima under a
certain partial order on Mac Lane valuations (see
\S\ref{Smaclane} for definitions of these terms).  The singularities
of the normalization of $\mc{Y}'$ in $K(X)$ are relatively
manageable, and we modify $\mc{Y}'$ by
adding in ``tails'' and ``links'' to resolve them.  This process is
parallel to the process of resolving the singularities of $\mc{Y}'$ itself,
as described in \cite[Corollaries 7.5, 7.6]{ObusWewers}, but the
formulas are more complicated when working
with a cyclic cover.   This completes Algorithm~\ref{AY0}, giving
models $\mc{Y}$ of $\proj^1_K$ and $\mc{X}$ of $X$ as above.

Let $\ol{Y}$ and $\ol{X}$ be the special fibers of $\mc{Y}$ and
$\mc{X}$, respectively.  To finish the proof of Theorem~\ref{Tminnormalcrossingsbase}, we use
the fact that our explicit Mac Lane descriptions of the components of $\ol{Y}$
allow us to compute the neighboring components
of a given component of $\ol{X}$, in terms of the neighboring
components of its image in $\ol{Y}$ (via the partial order on Mac Lane
valuations) as well as their multiplicities and the ramification locus
(from the value groups of the Mac Lane valuations). This in turn
allows us to use Castelnuovo's criterion to immediately rule out contractibility of
most components while preserving regularity and the normal crossings property
(Lemmas~\ref{Lramificationnotremovable}, \ref{LCastelnuovo}, and
\ref{Lremovable}).  We identify which of the
remaining components are indeed contractible in the remainder of the
section.
%We then identify contractible components of
%$\ol{X}$ by translating Castelnuovo's criterion
%into various criteria in terms of the value groups and the partial
%ordering of the Mac Lane valuations coresponding to $\ol{Y}$. 

\begin{comment}
\padma{Moved to earlier.} We remark that Algorithm~\ref{AY0} is \emph{not} recursive.  That is,
the idea of the algorithm is not ``come up with an invariant of a
singularity, find blowups that decrease this invariant, and repeat''.
Rather, each step described in the above paragraph is performed exactly once,
and then the algorithm terminates.
\end{comment}

\subsection{Relationship with recent related work}\label{Srelated}
There has been a flurry of recent work on explicit regular models of
curves, stemming from work of Dokchitser (\cite{Dokchitser}) as well
as Dokchitser--Dokchitser--Maistret--Morgan (\cite{DDMM}).  The paper
\cite{DDMM} gives an explicit regular model of the hyperelliptic curve
with affine equation $y^2 = f(x)$ with semistable
reduction over $K$ when $\cha k \neq 2$.  This is done in terms of the \emph{cluster
  picture} of $f$, which encodes the distances between the roots of
$f$ in terms of the absolute value on $K$.  This work was later
combined with resolution of tame quotient singularities in \cite{FN} to exhibit the minimal normal crossings regular model
of any hyperelliptic curve with semistable reduction over a
\emph{tame} extension of $K$ (again, assuming $\cha k \neq 2$).

On the other hand, in
\cite{Dokchitser}, an explicit description of the minimal regular
model of (the projective smooth model of a) plane curve $f(x,y) = 0$
over $K$ is given, provided that $f$ satisfies a property called
$\Delta_v$-\emph{regularity}.  This result is quite general, although
it does not work on all curves in Theorem~\ref{Tmain}.  In fact, in
\cite{Muselli1}, Muselli combined Dokchitser's work with the technique
of cluster
pictures to compute the
minimal normal crossings regular model of more general hyperelliptic curves
with $\cha k \neq 2$,
including many that require a wild extension of $K$ to attain
semistable reduction.   Muselli's method even works sometimes when
$\cha k = 2$.  But it does not work on all hyperelliptic curves with
$\cha k \neq 2$.\footnote{Roughly, if a hyperelliptic curve is given
  by $y^2 = f(x)$, \cite{Muselli1} requires that for all irreducible
  factors $f_i$ of $f$, the Mac Lane valuation $v_{f_i}$ (see Proposition/Definition~\ref{Pbestlowerapprox})  has
  inductive length $1$.}

In subsequent work (\cite{Muselli2}), Muselli computed the minimal
normal crossings regular model for \emph{all} hyperelliptic curves over $K$ with $\cha
k \neq 2$.  For this computation, he introduced the technique of
\emph{Mac Lane clusters}, a combination of cluster pictures and Mac
Lane valuations.  Several ideas in \cite{Muselli2} are similar to
those we use in this paper (the first two steps in Algorithm~\ref{AY0}
are similar to computing the Mac Lane cluster picture for a
hyperelliptic curve), but our result is independent of \cite{Muselli2}.  In
fact, we do not use any results from any of the papers mentioned in
this subsection.

In \cite{KW}, the authors build a
normal model for any superelliptic curve as in Theorem~\ref{Tmain}
having only \emph{rational} singularities as a cyclic cover of a model
of $\proj^1_K$ where the branch locus has been resolved.  As in our paper, this
model is built by explicitly presenting the Mac Lane valuations
corresponding to a model of $\proj^1_K$.  The model we
construct in Algorithm~\ref{AY0} is related to the model in \cite{KW},
although neither dominates the other; our model simultaneously
resolves the singularities in \cite{KW} while removing extraneous components.
Similarly, in earlier work (\cite{OShoriz}), we described how to use the machinery of Mac Lane valuations to describe the minimal embedded resolution of a divisor on $\mathbb{P}^1_{\mc{O}_K}$. When $\chara k \neq 2$, in \cite{OS1}, we used regular models of the cover obtained from an embedded resolution of its branch divisor (without any additional semistability hypothesis) to prove an inequality between the conductor and the minimal discriminant for hyperelliptic curves.

\subsection{Outline of the paper}\label{Soutline}
In \S\ref{Sregularprelims}, we collect various preliminary results on
arithmetic surfaces.  Of possible independent interest
(although its proof is essentially the same as the argument in
\cite[\S6.1]{LL}) is Lemma~\ref{L:LiuLorAdapt}, which gives
a formula relating $\mathbb{Q}$-valued intersection numbers of $\mathbb{Q}$-Cartier divisors on a \emph{normal}
arithmetic surface to those on a branched cover. In \S\ref{Smaclane} we
introduce Mac Lane valuations and prove some results in ``pure''
valuation theory.  In \S\ref{Smaclanemodels}, we relate Mac Lane
valuations to normal models of $\proj^1_K$, and we show how certain
valuation-theoretic properties translate to properties of the
corresponding models.  In \S\ref{Ssmoothvertical}, we give some sufficient
criteria for a point on the reduced special fiber of a model of a
cyclic cover of $\proj^1_K$ to be smooth on an irreducible component on which
it lies.
After a short interlude on lattice theory in
\S\ref{SLattices}, the heart of the paper begins in
\S\ref{Sdetecting}, where we give criteria for detecting whether a
the normalization of a normal model $\mc{Y}$ of $\proj^1_K$ in $K(X)$
is, in fact, regular with normal crossings at a given point.  Here
$\mc{Y}$ is given as a set of Mac Lane valuations, and the criteria
are given directly in terms of these valuations.  In
\S\ref{Sconstruction}, we present and prove the correctness of
Algorithm~\ref{AY0}, constructing a model $\mc{Y}^{\reg}$ of
$\proj^1_K$ (corresponding to a set $V^{\reg}$ of Mac Lane valuations)
whose normalization in $K(X)$ is regular with normal crossings.  In
\S\ref{Ssecondmodel}, we prove
Theorem~\ref{Tminnormalcrossingsbase}, which summarizes which
valuations must be removed from $V^{\reg}$ in order to get the
\emph{minimal} regular normal crossings model.  Lastly, we illustrate our algorithm 
with some examples in \S\ref{Sexamples}.

\section*{Notation and conventions}
Throughout, $K$ is a complete field with respect to a
discrete valuation $v_K$. Let $\mc{O}_K$ denote the ring of integers of $K$. 
%\padma{Henselization and completion are the same, $S$ is excellent -- cite Liu -- complete Noetherian complete local rings are excellent. Propn. 3.5 and 3.6 in Andrew and Stefan's paper.11.24. Delete hats on maximal ideals.}  
We further assume
that the residue field $k$ of $K$ is \emph{algebraically
  closed}.  The case where $k$ is perfect immediately reduces to this case
  since regular models satisfy \'{e}tale descent.  More specifically,
  if $k$ is perfect, then to find the minimal regular normal crossings model
  of $X/K$, first find the minimal regular normal
  crossings model after base changing to the completion
  $\widehat{K^{ur}}$ of the maximal
unramified extension of $K$, which has algebraically closed residue
field.  Then take the quotient
by $\Gal(\widehat{K^{ur}}/K)$.

We denote an algebraic
closure of $K$ by $\ol{K}$.
We fix a uniformizer $\pi_K$ of $v_K$ and normalize $v_K$ so that $v_K(\pi_K) = 1$. Note that the valuation $v_K$ uniquely extends to a valuation on $\ol{K}$, which we also call $v_K$.

For a reduced $K$-scheme or $\mc{O}_K$-scheme $S$, we denote the
corresponding total ring of fractions by by $K(S)$. If $S$ is
integral, then $K(S)$ is the function field of $S$.
If $\mc{Y} \rightarrow \Spec \mc{O}_K$ is an arithmetic
surface, an irreducible Weil divisor of $\mc{Y}$ is called
\emph{vertical} if it lies in a fiber of $\mc{Y} \to \Spec \mc{O}_K$, and
\emph{horizontal} otherwise. Let $f \in K(\mathcal{Y})$. We denote the
divisor of zeroes of $f$ on $\mc{Y}$ by $\divi_0(f)$. 
For any discrete valuation $v$, we denote the corresponding value group by $\Gamma_v$.
If $P$ is a closed point on $\mc{Y}$, we denote the corresponding local ring by $\mc{O}_{\mc{Y},P}$ and maximal ideal by $\mathfrak{m}_{\mc{Y},P}$. 

Throughout this paper, we fix a system of homogeneous coordinates $\P^1_{K} =
\Proj K[t_0,t_1]$, and a smooth model $\P^1_{\O_K}
\colonequals \Proj \mc{O}_K[t_0,t_1]$. We also set $t \colonequals t_1/t_0$.  

All minimal polynomials are assumed to be monic.  
When we refer to the \emph{denominator} of a rational
number, we mean the positive denominator when the rational number is
expressed as a reduced fraction.

\section*{Funding}
The first author was supported by the National Science
  Foundation under CAREER Grant DMS-2047638.  He was also
  supported by a grant from the Simons Foundation (\#706038: AO). The second author was supported by the Simons Collaboration in Arithmetic Geometry, Number Theory, and Computation via the Simons Foundation grant 546235 and by the National Science
  Foundation under Grant DMS-2401547.

\section*{Acknowledgements}
We thank Jordan Ellenberg for coming up with the term ``aligned with
the coordinate axes'' for the concept in Definition~\ref{Daxes}.

\section{Preliminaries on normal and regular models}\label{Sregularprelims}
\subsection{Definitions}
An \emph{arithmetic surface} is a normal, integral, projective, flat $\mc{O}_K$-scheme of relative
dimension $1$. A \emph{local arithmetic surface} is an affine
$\mc{O}_K$-scheme whose coordinate ring is isomorphic to
$\hat{\mc{O}}_{\mc{X},x}$, where $\hat{\mc{O}}_{\mc{X},x}$ is the
completed local ring at a closed point $x$ of an arithmetic surface
$\mc{X}$. An arithmetic surface is said to have \emph{normal
  crossings} if for every closed point $x$ of $\mc{X}$, there is a finite \'{e}tale morphism $\mc{Z} \rightarrow \mc{X}$ such that for every closed point $z$ lying about $x$ in $\mc{Z}$ the completed
local ring $\hat{\mc{O}}_{\mc{Z},z}$ is isomorphic to
$\mc{O}_K[[t_1,t_2]]/(t_1^at_2^b-u \pi_K)$ for some unit $u$ in
$\hat{\mc{O}}_{\mc{X},x}$ and integers $a,b \geq 0$ with $a+b >
0$. (See for e.g. \cite[\S9.1, Definition~1.6, Remark~1.7]{LiuBook} and \cite[\S9.2.4, Proposition~2.34]{LiuBook}) .
%\padma{Which definition of normal crossings are we allowing -- the   nonstrict version (definition above) is naturally \'{e}tale local in   nature, and what is needed in the proof of   Proposition ~\ref{P:RegandNormalize}} \andrew{Yes, I think we want the non-strict version.}
%\padma{Check if we are allowing self-intersection of components
%e.g. like $\mc{O}_K[[t_1,t_2]]/(t_2^2-(t_1^2-\pi_K))$, and if we want
%strict normal crossings instead.} 

Let $\mc{X}$ be a normal model of an algebraic curve $X$.  A morphism
$\pi \colon \widetilde{\mc{X}} \to \mc{X}$ is called a \emph{minimal
  regular resolution of $\mc{X}$} if $\widetilde{\mc{X}}$ is a
(proper) regular model of $X$ such
that the special fiber of $\widetilde{\mc{X}}$ contains no $-1$-components
(\cite[Definition 2.2.1]{CES}).  Such minimal regular
resolutions exist and are unique, e.g., by \cite[Theorem 2.2.2]{CES}.
A morphism
$\pi \colon \widetilde{\mc{X}} \to \mc{X}$ is called a \emph{minimal
  normal crossings resolution of $\mc{X}$} if $\widetilde{\mc{X}}$ is a
(proper) regular model of $X$ such that the special fiber of
$\widetilde{\mc{X}}$ has normal crossings, and if $\pi' \colon
\widetilde{\mc{X}}' \to \mc{X}$ is any other morphism with
$\widetilde{\mc{X}}'$ a proper regular normal crossings model, there
is a unique morphism $f' \colon  \widetilde{\mc{X}}' \to
\widetilde{\mc{X}}$ such that $\pi' = \pi \circ f'$.  By the universal
property, the minimal normal crossings resolution is unique.
%\andrew{Also define minimal regular resolutions wiht normal crossings.} 

\begin{remark}\label{R:SNCRealm}
The construction of the minimal crossings model in \cite[\S9.3.4, Definition~9.3.31,
Proposition~3.36]{LiuBook} shows that one can start with an arbitrary
regular model with normal crossings, and successively contract a
subset of the $-1$ curves that preserve the property of being normal
crossings (see \cite[\S9.3.4, Lemma~3.35]{LiuBook} for how to
identify such $-1$ curves) until we obtain the minimal normal
crossings model.
%\andrew{Add citation of Liu Def. 9.3.31 to this remark}
% \padma{We can make this a citeable remark if we
                 %want.} \andrew{Basic generalities on regular NC
                 %models (i.e., can contract to minimal regular NC
                 %model without leaving the realm of NC-ness).} 
\end{remark}

\subsection{Intersection theory of $
\mathbb{Q}$-Cartier divisors on normal arithmetic surfaces}\label{Sintersectiontheory} 
%A \emph{locally $\mathbb{Q}$-factorial} arithmetic surface $\mc{X}$ is a normal arithmetic surface such that every Weil divisor is $\mathbb{Q}$-Cartier. Equivalently, for every Weil divisor $D$, there is an integer $m_D \geq 1$ (not unique) such that $m_D D$ is a Cartier divisor. 

Let $\mc{X}$ be a normal arithmetic surface. Let $\Div(\mc{X})$ denote the subgroup of Weil divisors such that some multiple is a Cartier divisor on such a
surface. Recall that there is a well-defined bilinear intersection
pairing of Cartier divisors on any normal arithmetic surface $\mc{X}$
-- if $f$ and $g$ are functions defining two relatively prime Weil divisors $D_f$ and
$D_g$ on the local arithmetic surface $\hat{\mc{O}}_{\mc{X},x}$, then
the local intersection number $(D_f, D_g)$ is the length of the scheme $\hat{\mc{O}}_{\mc{X},x}/(f,g)$, and the global intersection number on $\mc{X}$ is the sum of local intersection numbers over all closed points of $\mc{X}$. This extends to a well-defined bilinear $\mathbb{Q}$-valued intersection pairing
\begin{align*}
 \Div(\mc{X}) \times \Div(\mc{X}) &\rightarrow \mathbb{Q} \\
 (D,D') &\mapsto \frac{1}{m_D m_{D'}} (m_D D,m_{D'}D'), \textup{ where}
\end{align*}
$m_D,m_{D'}$ are integers chosen such that $m_D D$ and $m_{D'}D'$ are Cartier divisors. 

We have the following lemma about the behaviour of intersection numbers under finite morphisms of $\mathbb{Q}$-Cartier divisors on normal arithmetic surfaces, adapting \cite[\S6.1]{LL} to the $\mathbb{Q}$-factorial setting.
\begin{lemma}\label{L:LiuLorAdapt}
 Let $\mc{W}$ and $\mc{Z}$ be two local normal
 %$\mathbb{Q}$-factorial 
 schemes
 of dimension $2$, with closed points $w$ and $z$ respectively. Let
 $\phi \colon \mc{W} \rightarrow \mc{Z}$ be a dominant finite
 morphism. Let $\Gamma_1,\Gamma_2$ be two irreducible $\mathbb{Q}$-Cartier Weil divisors on
 $\mc{W}$, and let $\Delta_i \colonequals \phi(\Gamma_i)$. Assume that
 $\phi^{-1}(\Delta_1) = \Gamma_1$, and for $i \in \{1,2\}$, let
 $e_{\Gamma_i/\Delta_i}$ be the ramification index of $\Gamma_i$ over
 $\Delta_i$.
 %\andrew{$\phi^*$ and
 %  $\phi_*$ should not be defined in the statement of the theorem,
  % since they are not used in the statement of the theorem.}
  % $\phi^*(\Delta_1) \colonequals e_{\Gamma_1/\Delta_1} \Gamma_1$,
  % Let $\phi_*(\Delta_2) = [k(\Gamma_2) \colon k(\Delta_2)]
  % \Gamma_2$.
 Then
 \[ \deg(\phi) (\Delta_1, \Delta_2) = e_{\Gamma_1/\Delta_1} e_{\Gamma_2/\Delta_2} [k(w):k(z)] (\Gamma_1, \Gamma_2) .\]
\end{lemma}
\begin{proof}
 Note that if we define $e_{m\Gamma_i/m\Delta_i}$ by the equation
 $\phi^*(m\Delta_i) \colonequals e_{m\Gamma_i/m\Delta_i} (m\Gamma_i)$,
 we have  $e_{m\Gamma_i/m\Delta_i} = e_{\Gamma_i/\Delta_i}$ since
 $\phi^*$ is a group homomorphism. The projection formula holds for
 intersections of Cartier divisors on normal schemes by \cite[\S9.2,
 Remark~2.13]{LiuBook}. Now, repeat the argument in \cite[\S6.1]{LL} that uses the projection formula after replacing $\Delta_i$ and $\Gamma_i$ with a suitably large integer multiple to make them all Cartier and combine with the first sentence to conclude that
 \[ \deg(\phi)( m\Delta_1, m\Delta_2) = e_{\Gamma_1/\Delta_1} e_{\Gamma_2/\Delta_2} [k(w):k(z)] (m\Gamma_1 , m\Gamma_2) ,\]
 for an integer $m$. Finally divide both sides by $m^2$ and use the bilinearity of the extended intersection pairing. 
\end{proof}

\subsection{Normalizations and regularity}
%\padma{Made $\mc{Z}$ in both Corollary~\ref{Clinear} and Proposition~\ref{P:RegandNormalize} a local arithmetic surface -- from a quick search, we only seem to apply them in the local situation anyway. The only mild annoyance is the normalization $\mc{W}$ is no longer irreducible, but we can still talk about its total ring of fractions instead of fraction field -- it is an \'{e}tale $K(Z)$ algebra. Added the references we talked about to Remark~\ref{R:AllLocal} below}
%\begin{remark}\label{R:AllLocal} 
In the rest of this section, we are interested in understanding local properties (such as regularity, normal crossings etc.) at a closed point of an arithmetic surface $\mc{W}$ obtained as the normalization of an arithmetic surface $\mc{Z}$ in a finite cyclic extension of $K(Z)$. For these purposes, we claim that we may assume that $\mc{Z}$ is a {\textit{local}} arithmetic surface without any loss of generality. Indeed, by \cite[Proposition~11.24]{AM} a Noetherian local ring is regular if an only if its completion is. Furthermore, since $\mc{O}_K$ is a complete discrete valuation ring, and $\mc{W}$ and $\mc{Z}$ are finite type $\mc{O}_K$ schemes, \cite[Chapter~8.2, Theorem~2.39]{LiuBook} guarantees that $\mc{W}$ and $\mc{Z}$ are excellent, and hence taking normalizations and completions commute by \cite[Chapter~8.2, Proposition~2.41]{LiuBook} -- more precisely, if $\phi$ is the finite map $\mc{W} \rightarrow \mc{Z}$, then $\mc{W} \otimes_{\mc{O}_{\mc{Z},z}} \hat{\mc{O}}_{\mc{Z},z} \cong \prod_{w \in \phi^{-1}(z)} \hat{\mc{O}}_{\mc{W},w}$. 
%\end{remark}

\begin{lemma}\label{L:ComputingNormalization}\label{L:WhenRegular}
Let $(R,\mathfrak{m})$ be a regular complete local $2$-dimensional integral
 domain with fraction field $K$. Fix an integer $d \geq 2$ coprime to
 $\chara (R/\mathfrak{m})$. Let $f$ be a
 nonzero element of $R$ with irreducible factorization $f_1^{a_1}
 \ldots f_r^{a_r}$ for some integers $a_i$, such that $\gcd(d,a_1,\ldots,a_r)=1$. Let $L =
 K[v]/(v^d-f)$. Let $e_i \colonequals d/\gcd(d,a_i)$. Let $S$ be the normalization of $R$ in $L$. Then,
 \begin{enumerate}[\upshape (i)]
 \item\label{L:eiisramindex} The integer $e_i$ is the ramification index of every prime divisor lying above $(f_i)$. 
 \item\label{L:NotReg} If the $e_i$ are not pairwise relatively prime, then $S$ is not regular.
 \end{enumerate}
 For the remainder of the lemma, assume that the $e_i$ are pairwise relatively prime.
 \begin{enumerate}[\upshape (i)]
 \setcounter{enumi}{2}
 \item\label{L:viExist} For $1 \leq i \leq r$, there are elements $v_i$ in $S$ satisfying $v_i^{e_i} = f_i$ and $L=K(v_1,v_2,\ldots,v_r)$.
  \item\label{L:Bpresent} Suppose $v_i$ are as the previous part. Then $S \cong R[v_1,\ldots,v_r]/(v_1^{e_1}-f_1, \ldots, v_r^{e_r}-f_r)$. 
  \item\label{L:RegOnlyIf} $S$ is regular if and only if one of the following $3$ conditions hold: (a) $r=0$, (b) $r=1, f_1 \in \mathfrak{m} \setminus \mathfrak{m}^2$, and (c) $r=2, \mathfrak{m} = (f_1,f_2)+\mathfrak{m}^2$.
 \end{enumerate}
\end{lemma}
\begin{proof}\hfill
\begin{enumerate}[\upshape (i)]
\item This is immediate.
\item We will assume that $\gcd(e_i,e_j) > 1$ for some $i \neq j$ and $S$ is regular and arrive at a contradiction. Since $S$ is an integral extension of $R$, by the going-up theorem, for every $i$, there is a height $1$ prime ideal $\mathfrak{q}_i$ lying above $(f_i)$.  Since $S$ is regular, it is a unique factorization domain, and every height $1$ prime ideal is principal \cite[\href{https://stacks.math.columbia.edu/tag/0FJH}{Lemma~15.121.2}, \href{https://stacks.math.columbia.edu/tag/034O}{Lemma~10.120.6}]{StacksProject}. Let $v_i$ be a generator for $\mathfrak{q}_i$. Since the ramification index of $\mathfrak{q}/(f_i)$ is $e_i$, and since $e_i$ divides $d$ and all units are $e_i^{\mathrm{th}}$ powers by Lemma~\ref{Lallnthroots}, we may assume that $v_i^{e_i} = f_i$ without any loss of generality. Since $L/K$ is a Kummer extension, there is a unique subextension of degree $e \colonequals \gcd(e_i,e_j)$, contained in both the unique subextension $K(v_i)$ of degree $e_i$ and the unique subextension $K(v_j)$ of degree $e_j$. This extension can therefore be generated by both $v_i' \colonequals v_i^{e_i/e}$ and $v_j' \colonequals v_j^{e_j/e}$. Since $K$ has all $e^{\mathrm{th}}$ roots of unity, since $v_i'^{e} = f_i, v_j'^{e_j} = f_j$, Kummer theory says the two extensions $K(v_i')$ and $K(v_j')$ are equal if and only if $f_i/f_j$ is an $e^{\mathrm{th}}$ power in $K$. Since $f_i,f_j$ are distinct irreducible elements in the unique factorization domain $R$, this is a contradiction. 
 \item 
 Since $\gcd(e_i,a_i) = 1$ by definition of $e_i$, it follows that there are integers $k_i,c_i$ such $a_i k_i = 1 + c_i e_i$. Suppose $i \neq j$. Since $\gcd(e_i,e_j) = 1$ and $d= \gcd(d,a_j) e_j = \gcd(d,a_i) e_i$, it follows that $e_i$ divides $\gcd(d,a_j)$, and hence $e_i$ divides $a_j$. Define the integer $c_j \colonequals a_j/e_i$ for $j \neq i$. Consider the element $v_i \colonequals v^{k_i \gcd(d,a_i)}/f_i^{c_i} \prod_{\substack{j=1, j \neq i}}^r{f_j^{c_j k_i}}$. Then combining $v^d=f$ with the definitions of $c_i, k_i, e_i, v_i$ we get that
 \[ v_i^{e_i} = \frac{v^{k_i \gcd(d,a_i) e_i}}{f_i^{c_i e_i} \prod_{\substack{j=1, j \neq i}}^r{f_j^{c_j k_i e_i}}} = \frac{(v^d)^{k_i}}{f_i^{a_i k_i-1} \prod_{\substack{j=1 \\ j \neq i}}^r{f_j^{a_j k_i}}} = \frac{\prod_{j=1}^r f_j^{a_j k_i}}{f_i^{a_ik_i-1} \prod_{\substack{j=1 \\ j \neq i}}^r{f_j^{a_j k_i}}} = f_i. \]
 %where the last equality follows from $a_i k_i = 1 + c_i e_i$ and $a_j k_i = c_j e_i k_i$ for $j \neq i$.
Since $x^{e_i}-f_i$ is Eisenstein at $f_i$, it follows that $K(v_i)/K$ is a degree $e_i$ extension of $K$ that is totally ramified above the prime ideal $(f_i)$. Since $\gcd(d,a_1,\ldots,a_r) = 1$ implies that $\gcd(d/e_1,d/e_2,\ldots,d/e_r) = 1$, it follows that $v$ is in the compositum of the extensions $K(v^{d/e_i})$ of $K$. It remains to show $K(v^{d/e_i}) = K(v_i)$. This follows since both $K(v_i)$ and $K(v^{d/e_i})$ are both subextensions of $K(v)$ of degree $e_i$ over $K$, and the fact that $K(v)/K$ has a unique subextension of degree $e_i$ by virtue of being a Kummer extension of degree $d$ (since $K$ has all $d$-th roots of unity by Lemma~\ref{Lallnthroots} and our assumption that $d$ is coprime to $\chara (R/\mathfrak{m})$).

\item 
%Since $K(v_i)/K$ is totally ramified above $(f_i)$, the integral closure of $R$ in $K(v_i)$ is $R[v_i]/(v_i^{e_i}-f_i)$. \padma{justify?} Furthermore, since the degrees of the extensions $K(v_i)/K$ are relatively prime, and the discriminants $e_i f_i^{e_i-1}$ of $K(v_i)/K$ are relatively prime in $R$, it follows that the integral closure of $R$ in the compositum $K(v_1,\ldots,v_r)/K$ is the tensor product of $R[v_i]/(v_i^{e_i}-f_i)$ over $R$. \padma{add ref, maybe do we need the rest?} 
%\padma{Replace by $K(v_i)/K$ a totally ramified extensions of coprime degrees with relatively prime discriminant ideals, so the ring of integers is the tensor product of the ring of integers $R[v_i]/(v_i^{e_i}-f_i)$ for the individual extensions.}

  %Since the $e_i$ are pairwise relatively prime, $\deg(K(v_i)/K) = e_i$ and $L=K(v_1,v_2,\ldots,v_r)$, it follows that $d = \deg(L/K) = e_1e_2 \ldots e_r$. If the degree of $v_i$ over $K(v_1,\ldots,v_{i-1})$ is strictly smaller than $e_i$ for any $i$, then considering the product of degrees of successive extensions in the tower $K \subset K(v_1) \subset \cdots K(v_1,v_2,\ldots,v_r) = L$  would imply that $d = \deg(L/K) < e_1 e_2 \ldots e_r$, a contradiction. This shows that $C$ is an integral domain.

We check that the ring $C \colonequals R[v_1,\ldots,v_r]/(v_1^{e_1}-f_1, \ldots, v_r^{e_r}-f_r)$ is a subring of $S$ that satisfies Serre's R1+S2 criterion for normality  \cite[\href{https://stacks.math.columbia.edu/tag/031S}{Lemma  031S}]{StacksProject}. Since $\Frac(C) = L$ by the previous part, this would tell us that $C=S$.   Since $C$ is visibly integral over $R$, it follows that $C$ is also a local ring of dimension $2$. Therefore, to verify that $C$ satisfies $S2$ it suffices to check that $\depth C_{\mathfrak{m}'} \geq \min (2,\height(\mathfrak{m}')) = 2$ for the unique height $2$ prime ideal $\mathfrak{m}'$. This follows since $C_{\mathfrak{m}'}/\mathfrak{m}' \cong C/\mathfrak{m}'$ is reduced.

To check the R1 condition, we have to check that the localization of $C$ at every height $1$ prime ideal $\mathfrak{q}$ is a discrete valuation ring. Let $\mathfrak{p} \colonequals \mathfrak{q} \cap R$. The normality of $R$ implies that $R$ is R1 and hence $R_{\mathfrak{p}}$ is regular. If $\mathfrak{p}$ is not supported on any of the $f_i$, then $R_{\mathfrak{p}} \hookrightarrow C_{\mathfrak{q}}$ is an \'{e}tale extension, which in turn implies that $C_{\mathfrak{q}}$ is also regular by \cite[p.49, Proposition 9]{BLR}. If $\mathfrak{p} = (f_i)$, let $v_i$ be as in the previous part. We will show that $(v_i) = \mathfrak{q}$ by arguing that $v_i$ is an element of minimal positive $\mathfrak{q}$-adic valuation.  Let $v_{\mathfrak{q}}$ be the valuation on $K$ extending the valuation $v_{\mathfrak{p}}$ on $K$. Since $v_i^{e_i} = f_i$, it follows that $v_i \in \mathfrak{q}$ and $v_{\mathfrak{q}}(v_i) = v_{\mathfrak{q}}(f_i)/e_i = v_{\mathfrak{p}}(f_i)/e_i = 1/e_i$. For any $j \neq i$, since $f_j \notin \mathfrak{p} = \mathfrak{q} \cap K$, it follows that $v_{\mathfrak{q}}(f_j) = 0$, and hence \begin{equation}\label{E:ValGrp} a_i = \sum_{j=1}^r a_j v_{\mathfrak{q}}(f_j) = v_{\mathfrak{q}}(f_1^{a_1} \ldots f_r^{a_r}) = v_{\mathfrak{q}}(v^d) = d v_{\mathfrak{q}}(v).   \end{equation} This shows $v_{\mathfrak{q}}(v) = a_i/d$, and since $v$ generates the Kummer extension $K$ of $L$, the value group of $v_{\mathfrak{q}}$ is $(a_i/d) \mathbb{Z} = (1/e_i) \mathbb{Z}$ by definition of $e_i$.  Since the value group of $v_{\mathfrak{q}}$ is $(1/e_i) \mathbb{Z}$ by \eqref{E:ValGrp}, and $v_{\mathfrak{q}}(v_i) = 1/e_i$, it follows that $\mathfrak{q} = (v_i)$ and $C_{\mathfrak{q}}$ is a discrete valuation ring as claimed.

\item Let $\mathfrak{m}'$ be the unique maximal ideal of $S$. Then $S$ is regular if and only if $\mathfrak{m}'/\mathfrak{m'}^2$ is generated by $2$ elements. We first show that $S$ is regular in the three cases listed. If $r = 0$, then $S \cong R$ and is regular. If $r = 1$ and  $f_1 \in \mathfrak{m} \setminus \mathfrak{m}^2$, we may complete $f_1$ to a system of parameters $f_1,g_2$ for the regular local ring $R$. Then the maximal ideal of $S = R[v_1]/(v_1^{e_1}-f_1)$ is generated by $v_1,g_2$ modulo $\mathfrak{m'}^2$ and is therefore also regular. If $r=2, \mathfrak{m} = (f_1,f_2)+\mathfrak{m}^2$, and $v_1,v_2$ satisfy $v_1^{e_1} = f_1$ and $v_2^{e_2} = f_2$, then the unique maximal ideal $\mathfrak{m}'$ of $S = R[v_1,v_2]/(v_1^{e_1}-f_1,v_2^{e_2}-f_2)$ is $(v_1,v_2,f_1,f_2) = (v_1,v_2) + \mathfrak{m'}^2$. Therefore $R$ is regular.

If we are not in one of the three cases above, then either (i) $r=1$ and $f_1 \in \mathfrak{m}^2$, or (ii) $r=2$ and $(f_1,f_2) + \mathfrak{m}^2$ is a proper subideal of $\mathfrak{m}$, or (iii) $r \geq 3$. We now need to show that $\dim(\mathfrak{m}'/\mathfrak{m'}^2) \geq 3$ in each of these $3$ cases. Since $\mathfrak{m'} = \mathfrak{m}+(v_1,\ldots,v_r)$ and $e_i \geq 2$ for every $i$, we have that $f_i = v_i^{e_i} \in \mathfrak{m'}^2$ for every $i$ and therefore
\begin{equation}\label{E:Spresent} S/\mathfrak{m'}^2 \cong R[v_1,\ldots,v_r]/\mathfrak{m'}^2 \cong \left( R/ \mathfrak{m'}^2 \cap R \right) [v_1,\ldots,v_r]/(v_1^2,\ldots,v_r^2). \end{equation}

First assume that $r=1$ and $f_1 \in \mathfrak{m}'^2$. We will first show that $\mathfrak{m}'^2 \cap R = \mathfrak{m}^2$. Since $\mathfrak{m}' = \mathfrak{m}+(v_1)$, it follows that $\mathfrak{m'}^2 = \mathfrak{m}^2 + (v_1)\mathfrak{m}+(v_1^2)$.  If $e_1 \geq 3$, then $1,v_1,v_1^2$ are linearly independent over $K$, and it follows that $\mathfrak{m'}^2 \cap R = \left( \mathfrak{m}^2 + (v_1)\mathfrak{m}+(v_1^2) \right) \cap R = \mathfrak{m}^2$. If $e_1 = 2$, then $1,v_1$ are linearly independent over $K$ and $(v_1^2) = (f_1) \subset \mathfrak{m}^2$, and once again $\mathfrak{m'}^2 \cap R = \left( \mathfrak{m}^2 + (v_1)\mathfrak{m}+(v_1^2) \right) \cap R = \mathfrak{m}^2$. It follows that 
\[ S/\mathfrak{m'}^2 \cong \left( R/ \mathfrak{m'}^2 \cap R \right) [v_1]/(v_1^2) \cong \left( R/\mathfrak{m}^2 \right) [v_1]/(v_1^2).\] This presentation shows that if $g_1,g_2$ are a basis for $\mathfrak{m}/\mathfrak{m}^2$, then $g_1,g_2,v_1$ are a basis for $\mathfrak{m}'/\mathfrak{m'}^2$.

Now assume that $r=2$ and $(f_1,f_2) + \mathfrak{m}^2$ is a proper subideal of $\mathfrak{m}$. 
%\[ S/\mathfrak{m'}^2 \cong R[v_1,v_2]/\mathfrak{m'}^2 \cong \left( R/ \mathfrak{m'}^2 \cap R \right) [v_1,v_2]/(v_1^2,v_2^2). \] 
Let $g \in \mathfrak{m} \setminus \left( (f_1,f_2) + \mathfrak{m}^2 \right)$. To show that that $g,v_1,v_2$ are linearly independent elements of $\mathfrak{m}'/\mathfrak{m'}^2$, by \eqref{E:Spresent} and the third isomorphism theorem, it suffices to show that $g \notin \mathfrak{m'}^2 \cap R$.  Since $\mathfrak{m}' = \mathfrak{m}+(v_1,v_2)$, it follows that $\mathfrak{m'}^2 = \mathfrak{m}^2 + (v_1)\mathfrak{m}+ (v_2) \mathfrak{m}+(v_1v_2) + (v_1^2,v_2^2)$. Since $1,v_1,v_2,v_1v_2$ are linearly independent over $K$ and since $(v_1^2,v_2^2) \cap R \subset(f_1,f_2)$, it follows that $\mathfrak{m}'^2 \cap R \subset \mathfrak{m}^2 + (f_1,f_2)$. Since $g \notin \mathfrak{m}^2 + (f_1,f_2)$ by assumption, $g$ is also not in $\mathfrak{m}'^2 \cap R$. 

If $r \geq 3$, then $v_1,v_2,v_3$ are linearly independent elements of $S/\mathfrak{m'}^2$ and hence $S$ is not regular. \qedhere
\end{enumerate}
\end{proof}

\begin{remark}\label{R:WildNormalizeNP} Normalizations of a ring $A$ in an extension $L$ of its fraction field $K$ are harder to compute when the residue characteristic of $A$ divides the degree of the extension $L/K$ already when $\dim A = 1$. For example, let $A = \mathbb{Z}_2, K = \mathbb{Q}_2, L = \mathbb{Q}_2(\sqrt{-3})$. Then since $3$ is a unit in $A$, the analogue of the ring $B$ in the lemma above would be the ring $B \colonequals \mathbb{Z}_2[x]/(x^2+3)$ -- this ring is wildly ramified at $2$ above $\mathbb{Z}_2$, and is not regular at the unique prime $\mathfrak{m} = (2,x-1)$ above $2$ since the defining equation $x^2+3 \in \mathfrak{m}^2$. The normalization is obtained adjoining the element $y \colonequals (x-1)/2$ satisfying $y^2+y+1=0$ to $B$.
\end{remark}

\begin{lemma}\label{Lallnthroots}
If $\mc{X}$ is a local arithmetic surface, $x \in \mc{X}$ is a closed
point, and $d$ is prime to the residue characteristic, then all units in
$\mc{O}_{\mc{X}, x}$ are $d$th powers.
\end{lemma}

\begin{proof}
Let $u \in \mc{O}_{\mc{X},x}^{\times}$.  Since the residue
field $k$ is
algebraically closed, we may assume, after multiplying $u$ by a $d$th
power, that $u = 1 + m$, with $m \in \mf{m}_{\mc{X}, x}$.  Then, since
$d$ is a unit in $\mc{O}_{\mc{X}, x}$,
one can explicitly construct an $d$th root of $u$ using the binomial
expansion and the fact that $\mc{O}_{\mc{X}, x}$ is
$\mf{m}_{\mc{X}, x}$-adically complete. 
\end{proof}

Lemma~\ref{Lallnthroots} shows in particular that if $y \in \mc{Y}$ is a closed
point on an arithmetic surface, then the normalization of $\Spec \hat{\mc{O}}_{\mc{Y},y}$ in a
Kummer extension
$\hat{\mc{O}}_{\mc{Y},y}[z]/(z^n - g)$ is completely determined by the \emph{divisor} of $g$ in $\Spec \hat{\mc{O}}_{\mc{Y},y}$.

\begin{comment}
\andrew{I'm worried about the second part here: If we take something
      like the $v_0$-model of $y^3 = x^2 - \pi_K$ and take $z$ to be
      the closed point at $x = 0$, then all hypotheses are satisfied
      and $\mc{W}$ is regular at $w$, but $\mc{W}$ is \emph{not}
      normal crossings at $w$.  However, we only need the
      normal crossings assertion when $z$ lies on a \emph{vertical}
      component of $B$, which shouldn't be a problem.} 
\end{comment}
\begin{prop}\label{P:RegandNormalize}
 Let $\phi \colon \mc{W} \rightarrow \mc{Z}$ be a finite morphism
  of local arithmetic surfaces over $\mc{O}_K$ with branch divisor
  $B$.  Let $w$, $z$ be the closed
 points of $\mc{W}$ and $\mc{Z}$ respectively, such that $\phi(w) =
 z$. Assume that $\phi$ is
 cyclic of degree $\delta$ with $\chara k \nmid \delta$, and
 % of degree $d$.
 %\andrew{it will gel better with the applications
%   to the rest of the paper if we actually let $d$ be the ramification
%   index at $z$, and possibly call it something else.}
 that $\mc{Z}$ is regular with normal crossings.  Let $\phi_K \colon W
 \to Z$ be the generic fiber of $\phi$.

 \begin{enumerate}[\upshape (i)]
   \item If $B$ is irreducible and either empty or vertical,
 %Let $B_{\mathrm{bad}}$ be the set of nonregular points of the branch locus $B$. Then
% \begin{enumerate}[\upshape (i)] 
%\item\label{P:wheninBbad} The closed point $z$ is in $B \setminus
%  B_{\mathrm{bad}}$ if and only if the total ring of fractions $K(W)$ is
%  generated as a $K(Z)$-algebra by an element $v$ that satisfies
%  $v^d=h^r$ for some $h \in \mathfrak{m}_{\mc{Z},z} \setminus
%  \mathfrak{m}_{\mc{Z},z}^2$ and some $1 \leq r < d$.
  %\andrew{Make
  %  sure the slightly weaker statement here doesn't break anything
  %  further on.}
%and $1 \leq r < d$ with $\gcd(r,d)=1$. 
%\item\label{P:WhenIsUpstairsReg} The scheme $\mc{W}$ is regular at $w$ if and only if $z$ is not contained in $B_{\mathrm{bad}}$.
%\item If additionally $\mc{Z}$ is normal crossings at $z$ and $z$
%lies on a unique vertical component of $B$ and $z$ is not in
%$B_{\mathrm{bad}}$,
 then $\mc{W}$ is regular with normal crossings.  Furthermore, 
 if $z$ is non-nodal, then $w$ is non-nodal.
\item Assume further that
$\mc{Z}$ is \emph{smooth} over $\Spec
\mc{O}_K$.  If $q \in Z$ is a branch point of $\phi_K$ specializing
to $z$, let $s$ be the degree of $q$ over $K$.  Then $\mc{W}$ is regular with normal crossings if and
only if one of the following two cases holds:
\begin{enumerate}[\upshape (a)]
  \item $B$ is irreducible and regular, with either $s = 1$ or
    $\delta = s = 2$.
  \item $B$ consists of unique horizontal and
    vertical irreducible components $B_1$ and $B_2$, the
    ramification indices of $\phi$ over $B_1$ and $B_2$ are relatively
    prime, and $s = 1$.
  \end{enumerate}

\end{enumerate}
\end{prop}
\begin{proof}

We first prove part (i).  If $B$ is empty, then the cover of local rings
$\hat{\mc{O}}_{\mc{W},w} \rightarrow \hat{\mc{O}}_{\mc{Z},z}$ is
\'{e}tale above $z$, and therefore by \cite[p.49, Proposition~9]{BLR}
$\mc{W}$ is regular (and additionally normal crossings, resp.\ normal
crossings and non-nodal) at $w$ if and
only if $\mc{Z}$ is regular (and additionally normal crossings, resp.\
normal crossings and non-nodal) at
$z$. Now assume $B$ is vertical. Since $\mc{Z}$ is complete, regular and normal crossings at $z$, the
   local ring $\mc{O}_{\mc{Z},z}$ is isomorphic to
   $\mc{O}_K[[x,y]](x^ay^b-u \pi_K)$ for some unit $u$ in
   $\hat{\mc{O}}_{\mc{Z},z}$ and integers $a > 0$ and $b \geq 0$, and we may
   assume that $B = \divi(x)$.  Let $f$ be such that
   $\Frac(\hat{\mc{O}}_{\mc{W},w}) \cong \Frac
   (\hat{\mc{O}}_{\mc{Z},z})[v]/(v^{\delta} - f)$.  We may assume $f = wx^r$ where $w$ is a unit in
   $\hat{\mc{O}}_{\mc{Z},z}$ and $\gcd(\delta,r) = 1$ because $\mc{W}$ is connected.  Noting that $w$ is a $\delta$th power by
   Lemma~\ref{Lallnthroots} and raising $f$ to a prime-to-$\delta$th power
   (which does not change the extension), we may assume $f = x$.
   So the local ring $\hat{\mc{O}}_{\mc{W},w}$ is isomorphic to (the
   normalization of)
   $\mc{O}_K[[x,y]][v]/(x^ay^b-u \pi_K,v^{\delta}-x) \cong
   \mc{O}_K[[y,v]]/(v^{\delta a}y^b-u \pi_K)$.  But this ring is already
   regular (thus normal), and normal crossings.  In this situation,
   $z$ being non-nodal corresponds to $b = 0$, which shows that $w$ is
   non-nodal as well. This completes the
   proof of (i).

   For part (ii), first assume that $B$ is irreducible,
   so that $q$ is the only branch
point of $W \to Z$ specializing to $z$ and no vertical part of the
branch locus passes through $z$.  By the smoothness assumption, we
have $\mc{Z} \cong \Spec \mc{O}_K[[t]]$.  By the
Weierstrass preparation theorem, there is a monic irreducible
polynomial $g$ in $t$ of degree $s$ such that $q$ is given by $g(t) = 0$.
Furthermore, the reduction $\ol{g}(t)$ of $g(t)$ to $k[t]$ is $t^s$.

Since all units in
$\mc{O}_K[[t]]$ are $\delta$th powers by Lemma~\ref{Lallnthroots}, 
Kummer theory (or more specifically,
Lemma~\ref{L:ComputingNormalization}\ref{L:Bpresent} with $r = 1$ and $f_1 = g(t)$) gives us that
$\hat{\mc{O}}_{\mc{W},w} \cong \mc{O}_K[[t]][v]/(v^\delta - g(t))$.
%\padma{by Lemma~\ref{L:ComputingNormalization}~\ref{L:PrimePreimage}. (In the notation of Lemma~\ref{L:ComputingNormalization}, $\mc{W}$ is the local ring at the unique maximal ideal lying above the maximal ideal $(g_i)$ in $\mc{Z}$ -- up to units (which are $\delta$-th powers), this maximal ideal is generated by a $\delta$-th root $v$ of $g_i$.)}
%\andrew{Justify using Lemma 3.3 with $r = 1$ and $a_1 = 1$}.
The special fiber
of $\Spec \hat{\mc{O}}_{\mc{W},w}$ is thus isomorphic to $\Spec k[[t]][v]/(v^\delta - t^s)$.
Since $\delta > 1$ (because $q$ is a branch point),
this has normal crossings if and only if $s = 1$ or $\delta = s =
2$.  Furthermore, by
Lemma~\ref{L:WhenRegular}\ref{L:RegOnlyIf}, $B$ must be regular for
$\mc{W}$ to be regular.  This completes the case when $B$ is irreducible.

Now, assume that $B$ is reducible.  By Lemma~\ref{L:WhenRegular}\ref{L:RegOnlyIf},
$\mc{W}$ is regular if and only if $B$ has normal crossings, consists
of two irreducible components $B_1$, $B_2$, and has coprime
ramification indices $e_1$, $e_2$ above $B_1$, $B_2$ respectively.  So
assume $e_1$ and $e_2$ are relatively prime, and
let $B_1$ be the closure of $q$.  As above, we may assume that $B_1 = \divi(g)$, where $g$ is a
monic polynomial whose reduction $\ol{g} \pmod{\pi_K}$ is
$t^s$. 

First, assume $B_2$ is vertical, so 
$B_2 = \divi(\pi_K)$.  Then $B$ has normal crossings if and only if
$s = 1$ (because the
ideal $(g,\pi_K) = (t^s, \pi_K)$ in $\hat{\mc{O}}_{\mc{Z},z}$).  Thus
$\mc{W}$ is regular if and only if $s = 1$.
So $g \equiv t
\pmod{\pi_K}$ and \eqref{Egh} shows that
$\hat{\mc{O}}_{\mc{W},w}/\pi_K \cong k[[t]][v_1,v_2]/(v_1^{e_1} - t,
v_2^{e_2}) \cong k[[[v_1]][v_2]/v_2^{e_2}$,
which has normal crossings. 

It remains to show that if $B_2$ is horizontal, then $\mc{W}$ is not
regular with normal crossings.
Write $B_2 = \divi(h)$.
By Lemma~\ref{L:ComputingNormalization}\ref{L:Bpresent}, we have
\begin{equation}\label{Egh}
\hat{\mc{O}}_{\mc{W},w} \cong \mc{O}_K[[t]][v_1, v_2]/(v_1^{e_1} - g,
v_2^{e_2} - h).
\end{equation}
Again by the Weierstrass preparation theorem we can
take $h$ to be a polynomial with reduction $\ol{h}(t) \cong t^m \pmod{\pi_K}$ for some $m \in
\nats$.   For $B$ to have normal crossings, we must have either $s =
1$ or $m = 1$ (otherwise $t$ is not in the ideal $(g,h)$ in $\hat{\mc{O}}_{\mc{Z},z}$).  Without loss of generality,
asssume $s = 1$.  Then \eqref{Egh} shows that
$$\hat{\mc{O}}_{\mc{W},w}/\pi_K \cong k[[t]][v_1, v_2]/(v_1^{e_1} - t,
v_2^{e_2} - t^m) \cong k[[v_1]][v_2]/(v_2^{e_2} - v_1^{me_1}).$$  Since
$e_1, e_2 \geq 2$ and are relatively prime, one of $e_2$ or $me_1$ is
at least $3$ and both are at least $2$, which shows that the
special fiber of $\mc{W}$ does not have normal crossings.  We are done.
 %\padma{Very confused this morning -- the line above gives the
 %presentation of a regular arithmetic surface with normal crossings
 %special fiber -- how about dealing with the cases where the local
 %structure of $\hat{\mc{O}}_{\mc{Z},z}$ is something like
 %$\mc{O}_K[[x,y]]/(y^2-x^3-\pi_K)$ which is regular but not normal
 %crossings -- do
 %  we have to worry about covers of bases like these in our cases, or can we always assume our base is nc?} 
 %(In other words, the regularity assumption on $\mc{O}_{B,z}$ allows us to choose a defining equation for $B$ as part of the system of parameters in $\mc{O}_{\mc{Z},z}$.) If $B$ is not regular at $z$, then we can choose a presentation of $\hat{\mc{O}}_{\mc{Z},z}$ as above such that $B$ is cut out by 
 %\padma{I have gotten a little confused about nonregular points of branch locus, versus nonregular of multiplicity not divisble by$d$ -- if multiplicity is divisible by $d$, then it is not part of the branch locus, I guess.}
%or equivalently, when $r=0$,  both of the rings above are regular
%local rings in dimension $2$, since the cotangent space in the first
%case is generated by $v$ and $\pi_K$, and by $y$ and $v$ in the
%second case. The explicit presentation of these rings also show
% that 
%  \end{enumerate}
\end{proof}

 \begin{defn}\label{Dramified}
A morphism $S \to T$ of curves over $k$ is \emph{geometrically
  ramified} above a point $t \in T$ if the induced morphism $S^{\red}
\to T^{\red}$ on reduced induced subschemes is ramified above the
preimage of $t$ under $T^{\red} \hookrightarrow T$.  The
\emph{geometric ramification index} at a point of $S$ or $T$ is the
analogous ramification index on $S^{\red}$ or $T^{\red}$. 
\end{defn}

\begin{example}
The cover $y^d = \pi_K$ over $\proj^1_{\mc{O}_K}$ has geometric
ramification index $1$ at all points of the special fiber, whereas the
actual ramification index at any of these points is $d$.
\end{example}

\begin{corollary}\label{Cregularbranchindex}
  In the situation of Proposition~\ref{P:RegandNormalize} above,
  let $q \in Z$ be a point of $W \to Z$ of ramification index
  $e \geq 1$ specializing to $z$, assume that $z$ lies on a unique irreducible component $\ol{Z}$
  of the special fiber of $\mc{Z}$, assume that $w$ is regular in
  $\mc{W}$, and assume that no branch point of $W \to Z$ (other than
  possibly $q$), specializes to $z$.  Let $\ol{W}$ be the preimage of $\ol{Z}$ in $\mc{W}$.

  Then there exists $\mu \in \nats$ relatively prime to $e$ such that the
  multiplicity of each irreducible component of $\ol{W}$ in the
  special fiber of $\mc{W}$ equals  $\mu m_{\ol{Z}}$, where
  $m_{\ol{Z}}$ is the multiplicity of the component $\ol{Z}$ of $\mc{Z}$.  Furthermore, the geometric
  ramification index of $z$ in $\ol{W} \to \ol{Z}$ is $e$.
\end{corollary}

\begin{proof}
Since $\mc{W}$ is regular and there is a horizontal component of the
branch divisor of $\mc{W} \rightarrow \mc{Z}$ with ramification index
$e$ by assumption, by
Lemma~\ref{L:ComputingNormalization}\ref{L:eiisramindex}, \ref{L:NotReg}, we conclude that the ramification index $\mu$ above the unique vertical component $\ol{Z}$ in
$\mc{W} \to \mc{Z}$ is prime to $e$.  This gives the statement on
multiplicities.  The geometric ramification index at $z$ is unchanged
when replacing $\mc{W}$ with $\mc{V} := \mc{W}/(\ints/\mu)$. Now,
$\mc{V} \to \mc{Z}$ is
unramified along the special fiber $\ol{V} \to \ol{Z}$, and the ramification
index of $q$ in $\mc{V} \to \mc{Z}$ is still $e$.  The geometric ramification index
of $z$ in $\ol{V} \to \ol{Z}$ is thus the actual ramification index, which we call
$e_z$.  Now, $e
\leq e_z$ because the cardinality of the fiber can only go down under
specialization.  To show $e \geq e_z$, note that the cover
$\mc{V}/(\ints/e) \to \mc{Z}$ is unramified at $q$, and since
it is also unramified along $\ol{Z}$, purity of the branch
locus (\cite[X, Th\'eor\`eme~3.1]{SGA1}) shows that it is unramified
at $z$.  This means $e_z \mid e$.  Thus $e = e_z$ as desired. 
\end{proof}
  
\begin{prop}\label{P:ExtendingCoversToCanonicalModels}
 Suppose $X \rightarrow Y$ is a Galois cover of curves over $K$ with Galois group $G$. Then the action of $G$ extends to both the minimal proper regular model $\mc{X}^{\mathrm{min}}$ and the minimal normal crossings model $\mc{X}_{\mathrm{nc}}^{\mathrm{min}}$. The corresponding scheme-theoretic quotients $\mc{Y}^{\mathrm{min}} \colonequals \mc{X}^{\mathrm{min}}/G$ and $\mc{Y}_{\mathrm{nc}}^{\mathrm{min}} \colonequals \mc{X}_{\mathrm{nc}}^{\mathrm{min}}/G$ are normal models of $Y$. Equivalently, $\mc{X}^{\mathrm{min}}$ and $\mc{X}_{\mathrm{nc}}^{\mathrm{min}}$ are normalizations of the normal models $\mc{Y}^{\mathrm{min}}$ and $\mc{Y}_{\mathrm{nc}}^{\mathrm{min}}$ in the function field of $X$.
\end{prop}
\begin{proof}
 By uniqueness of the minimal regular and minimal normal crossings models, the action of the Galois group extend to both models. Cover $\mc{X}^{\mathrm{min}}$ (and respectively
 $\mc{X}_{\mathrm{nc}}^{\mathrm{min}}$) by open affine subschemes
 $\Spec A$ that are invariant under the action of the finite group
 $G$. Then the quotients $\mc{Y}^{\mathrm{min}}$ (and respectively $\mc{Y}_{\mathrm{nc}}^{\mathrm{min}}$) are covered by the schemes $\Spec A^G$. If $A$ is a normal domain, then the ring of invariants $A^G$ for the action of a finite group $G$ is also normal. \qedhere
%  \padma{True for quasiprojective schemes, and a short argument is
%    given here --
%    \href{https://dept.math.lsa.umich.edu/~mmustata/appendix.pdf}{https://dept.math.lsa.umich.edu/~mmustata/appendix.pdf}
%    -- will still look for something citeable} \andrew{Or could just
%    argue directly based on local rings --- each local ring downstairs
%    is the quotient of a local ring upstairs by the action of the
%    inertia group.  You can probably get away without giving any
%    justification for this.}\padma{The issue is the quotient may not exist as a scheme (only as an algebraic space) without some additional conditions -- every orbit being contained in an affine open is enough -- we are probably okay because we are in the relative quasiprojective situation. Even Liu and Lorenzini don't get into this much detail, so not worth justifying.}    
\end{proof}

%Include here:
%\begin{itemize}
%\item Definition of arithmetic surface and local arithmetic surface (i.e., formal neighborhood of point on arithmetic surface)
%\item Definition of intersection numbers on locally $\rats$-factorial schemes, incl. normal arithmetic surfaces.
%\item Generalization of Liu-Lorenzini formula to context above.
%\item Lemma relating local regularity to local UFD on arithmetic surface where reduced subschemes of irreducible components of special fiber are smooth. \andrew{Done.}
%\item In particular, corollary that, in the case above, being regular where two components intersect is equivalent to either one of the components being a principal divisor. \andrew{Done.}
%\item Basic generalities on regular NC models (i.e., can contract to minimal regular NC model without leaving the realm of NC-ness).
%\item Lemma that cover of normal models is regular above points that
%  are regular downstairs and do not lie on non-regular parts of the
%  branch locus. Furthermore, in this case, NC downstairs implies NC upstairs.
%\item Partial converse: A cover of a regular model is \emph{not} regular
%  above a non-regular point of the branch locus.
%\item Both minimal regular model and minimal regular model with normal
%  crossings upstairs come from normalizations of a normal model downstairs.
%\end{itemize}

\begin{lemma}[cf.\ {\cite[Lemma 7.2(ii)]{ObusWewers}}]\label{LregularUFD}
  Let $\mc{X}$ be a local arithmetic surface with a smooth vertical prime divisor $D$.  Then the following are equivalent:
  \begin{enumerate}[\upshape (i)]
  \item $\mc{X}$ is regular.
  \item $D$ is principal.
  \item Every Weil divisor on $\mc{X}$ is principal.
  \item $\mc{X}$ is factorial.
  \end{enumerate}
  Furthermore, even if $D$ is not smooth, we have that statement (i) implies the
  other statements.
\end{lemma}

\begin{proof}
  By the Auslander--Buchsbaum theorem, (i) implies (iv).  Also,
  (iv) implies (iii) because every height 1 prime ideal in a UFD is
  principal.  That (iii) implies (ii) is trivial.  If (ii) holds, then
  the closed point $\ol{x}$ of $\mc{X}$ is a smooth point of the
  smooth divisor $D$, so it is a principal divisor of $D$, which means the ideal $m_{\mc{X}, \ol{x}}$ is generated by two elements.  So $\mc{X}$ is regular, proving (i).
\end{proof}

\begin{lemma}\label{Lbothprincipal}
Let $\mc{X}$ be a local arithmetic surface over $\mc{O}_K$ with two
relatively prime vertical reduced
divisors $D$ and $E$.  If $D$ and $E$ are
principal and $(D, E) = 1$, then $\mc{X}$ is regular.  A posteriori,
one can conclude that $D$ and $E$ are themselves prime divisors.
%Furthermore, if
%$\mc{X}$ has no other vertical prime divisors, then $\mc{X}$ has
%normal crossings.
\end{lemma}

\begin{proof}
  Let $D = \divi(\alpha)$ and $E = \divi(\beta)$.  By definition,
  $\mc{O}_{\mc{X}, x}/(\alpha, \beta)$ has length $1$ as an
  $\mc{O}_K$-module, which means $(\alpha, \beta)$ is the
  maximal ideal.  So $\mc{X}$ is regular.  If $D'$ and $E'$ are
  irreducible components of $D$ and $E$, then since all divisors meet
  at the closed point, $0 < (D', E') \leq 1$ with
  equality only if $D = D'$ and $E = E'$.  But since $\mc{X}$ is
  regular, $(D', E')$ is an integer, which shows that $D' = D$ and $E'
  = E$, so $D$ and $E$ are prime as desired.
\end{proof}

\subsection{Totally ramified morphisms}\label{Sramified}

\begin{lemma}\label{Ltotallyramified}
Let $\phi \colon \mc{W} \to \mc{Z}$ be a finite morphism of
arithmetic surfaces that is totally ramified above a prime Weil divisor $D$ of
$\mc{Z}$.  If $D$ is normal and $\deg(\phi)$ is prime to all residue characteristics of $D$, then the morphism $\phi^{-1}(D) \to D$ induces an
isomorphism $(\phi^{-1}(D))_{\red} \to D$.
\end{lemma}

%\andrew{I can't tell if the proof below is overkill --- if the result
%  is somehow a tautology.}

\begin{proof}
Localizing, we may assume $\mc{Z}$ is affine, so let $\mc{Z} = \Spec
A$ and $\mc{W} = \Spec B$.  Then $A \hookrightarrow B$ via $\phi^\#$ with
$B$ finite over $A$, and we identify $A$ with its image in $B$.  Let
$I$ be the ideal of $D$ in $\Spec A$, so that $IB$
is the ideal of $\phi^{-1}(D)$ in $\Spec B$.  By the totally ramified assumption,
the induced extension $\Frac(A/I) \subseteq \Frac(B/\sqrt{IB})$ of
fraction fields is purely inseparable, and by the assumption on $\deg(\phi)$ it is an
isomorphism.  We wish to show that the ring extension $A/I \subseteq B/\sqrt{IB}$
is an equality.

Let $b \in B$.  Since $B$
is integral over $A$, the minimal polynomial $f(T)$ of $b$ over $A$ is monic.
If $f(T)$ is purely inseparable modulo $I$, say,  
$f(T) \equiv (T - a)^e \pmod{I}$ for some $a \in A$, then $b - a \in \sqrt{IB}$,
so the residue of $b$ in $B/\sqrt{IB}$ is in $A/I$.  If not, then
letting $\ol{f(T)}$ be the residue of $f(T)$ modulo $I$, we have
$\deg(\rad(\ol{f(T)})) \geq 2$, which means that the image of $b$ in
$B/\sqrt{IB}$ has degree $\geq 2$ over $A/I$.  Since $D$ is normal, $A/I$ is
integrally closed, and the existence of $b$ thus contradicts the
equality $\Frac(A/I) = \Frac(B/\sqrt{IB})$.
\end{proof}

\section{Preliminaries on Mac Lane Valuations}\label{Smaclane}

\subsection{Definitions and facts}\label{Sbasicmaclane}

We recall the theory of inductive valuations,
which was first developed by Mac Lane in \cite{MacLane}.  We also use
the more recent \cite{Ruth} as a reference.  Inductive valuations give us an
explicit way to talk about normal models of $\proj^1$.

Define a \emph{geometric valuation} of
$K(x)$ to be a discrete valuation that restricts to $v_K$ on $K$ and
whose residue field is a finitely generated extension of $k$ with
transcendence degree $1$.   We place a partial order $\preceq$ on valuations by defining $v
\preceq w$ if $v(f) \leq w(f)$ for all $f \in K[x]$.  Let $v_0$ be the
\emph{Gauss valuation} on $K(x)$.  This is defined on $K[x]$ by
$v_0(a_0 + a_1x + \cdots a_nx^n) = \min_{0 \leq i \leq n}v_K(a_i)$,
and then extended to $K(x)$.  If $v$ is a geometric valuation, write
$\Gamma_v \subseteq \rats$ for its value group.

We consider geometric valuations $v$ such that $v \succeq v_0$. By the
non-archimedean triangle inequality, these are precisely those geometric valuations for which
$v(x) \geq 0$.  This entails no loss of generality, since $x$ can
always be replaced by $x^{-1}$. We would like an explicit formula for describing geometric valuations, similar to the formula above for the Gauss valuation, and this is achieved by the so-called
\emph{inductive valuations} or \emph{Mac Lane valuations}. Observe that the Gauss valuation is described using the $x$-adic expansion of a polynomial. The idea of a Mac Lane valuation is to ``declare'' certain polynomials $\phi_i$ to have higher valuation than expected, and then to compute the valuation recursively using $\phi_i$-adic expansions.

More specifically, if $v$ is a geometric valuation such that $v
\succeq v_0$, the concept of a \emph{key polynomial} over $v$ is
defined in \cite[Definition 4.1]{MacLane} (or \cite[Definition
4.7]{Ruth}).  Key polynomials are certain monic irreducible polynomials
in $\mc{O}_K[x]$ --- we do not give a definition, which would require more
terminology than we need to develop, but see Lemma \ref{Lfdegree} below for the most useful properties.  If $\phi \in \mc{O}_K[x]$ is a
key polynomial over $v$, then for $\lambda \geq v(\phi)$, 
we define an \emph{augmented valuation} $v' = [v, v'(\phi) = \lambda]$ on $K[x]$ by 
\begin{equation}\label{E:augval} v'(a_0 + a_1\phi + \cdots + a_r\phi^r) = \min_{0 \leq i \leq r}
v(a_i) + i\lambda \end{equation} whenever the $a_i \in K[x]$ are polynomials with
degree less than $\deg(\phi)$.  We should think of this as a ``base
$\phi$ expansion'', and of $v'(f)$ as being the minimum valuation of a
term in the base $\phi$ expansion of $f$ when the valuation of $\phi$ is
declared to be $\lambda$.  By \cite[Theorems 4.2, 5.1]{MacLane} (see also
 \cite[Lemmas 4.11, 4.17]{Ruth}), $v'$ is in fact a discrete
 valuation.  In fact, the key polynomials are more or less the polynomials
 $\phi$ for which the construction above yields a discrete valuation
 for $\lambda \geq v(\phi)$.  Note that if $\lambda = v(\phi)$, then
 the augmented valuation $v'$ is equal to $v$.  
 The valuation $v'$ extends to $K(x)$. 

We extend this notation to write Mac Lane valuations in the following
form: $$[v_0, v_1(\phi_1(x)) = \lambda_1, \ldots, v_n(\phi_n(x)) = \lambda_n].$$
Here each $\phi_i(x) \in \mc{O}_K[x]$ is a key polynomial over $v_{i-1}$, we
have that $\deg(\phi_{i-1}(x)) \mid \deg(\phi_i(x))$, and each $\lambda_i$
satisfies $\lambda_i \geq v_{i-1}(\phi_i(x))$. By abuse of notation,
we refer to such a valuation as $v_n$ (if we have not given it another
name), and we identify $v_{i}$ with $[v_0, v_1(\phi_1(x)) = \lambda_1, \ldots,
v_{i}(\phi_{i}(x)) = \lambda_{i}]$ for each $i \leq n$.  The valuations
$v_i$ are called \emph{predecessors} of $v_n$ and are uniquely determined, following \cite[Definition~2.12]{KW} (in
our earlier work we have called them \emph{truncations}).  

It turns out that the set of Mac Lane valuations on $K(x)$ exactly
coincides with the set of geometric valuations $v$ with $v \succeq
v_0$ (\cite[Corollary 7.4]{FGMN} and \cite[Theorem 8.1]{MacLane}, or \cite[Theorem
4.31]{Ruth}). Furthermore, 
every Mac Lane valuation is equal to one where the degrees of the
$\phi_i$ are strictly increasing (\cite[Lemma 15.1]{MacLane} or
\cite[Remark 4.16]{Ruth}), and where $v_i \neq v_{i+1}$ for all $i <
n$.  Such a presentation for a Mac Lane valuation is called
\emph{minimal}, and unless otherwise noted, we assume that all
presentations are minimal for the rest of the paper.
This has the consequence
that the number $n$ is well-defined.  We call $n$ the
\emph{inductive length} of $v$.  In fact, by \cite[Lemma
15.3]{MacLane} (or \cite[Lemma
4.33]{Ruth}), the degrees of the $\phi_i$ and the values of the $\lambda_i$ are invariants of $v$, once
we require that they be strictly increasing.  If $f$ is a key polynomial
over $v = [v_0,\, v_1(\phi_1) = \lambda_1, \ldots,\, v_n(\phi_n) =
\lambda_n]$ and either $\deg(f) > \deg(\phi_n)$ or $v = v_0$, we call $f$ a \emph{proper
  key polynomial over $v$}.  By our convention, each $\phi_i$ is a proper key
polynomial over $v_{i-1}$.  This has the immediate consequence that
$v_n(\phi_i) = \lambda_i$ for all $i$ between $1$ and $n$.

In general, if $v$ and $w$ are two
Mac Lane valuations such that the value group $\Gamma_w$
contains the value group $\Gamma_{v}$, we write $e(w/v)$ for the
ramification index $[\Gamma_w : \Gamma_v]$.  If $v$ is a Mac Lane
valuation, we simply write $e_v$ for $e(v/v_0)$, i.e.,
$\Gamma_v = (1/e_v)\ints$.

%\begin{remark}\label{Rrelram}
%Observe that if $[v_0,\, v_1(\phi_1) = \lambda_1, \ldots,\, v_n(\phi_n) =
%\lambda_n]$ is a Mac Lane valuation, where each $\lambda_i = b_i/c_i$
%in lowest terms, then $e_{v_n}  = \lcm(c_1, \ldots, c_n)$.
%Consequently, $e(v_i/v_j) = \lcm(c_1, \ldots, c_i)/\lcm(c_1,\ldots,
%c_j)$ for $i \geq j$. 
%\end{remark}

We can enlarge the set of Mac Lane valuations by allowing $\lambda_n =
\infty$ (this enlarged set is called the set of \emph{Mac Lane
  pseudovaluations}, see \cite[\S2.1, \S2.3]{KW}).  More specifically,
this means that if $g \in K[x]$ and $g = a_e\phi_n^e +
a_{e-1}\phi_n^{e-1} + \ldots + a_0$ is the $\phi_n$-adic expansion of
$g$, then $v(g) = v_{n-1}(a_0)$, with $v(g) = \infty$ when $a_0 = 0$.
A Mac Lane pseudovaluation with $\lambda_n = \infty$ is called an
\emph{infinite Mac Lane pseudovaluation}.  Mac Lane pseudovaluations
have predecessors defined identically to the case of Mac Lane valuations.

It is easy to see that if $v = [v_0,\, \ldots,\, v_n(\phi_n) = \infty]$ is a Mac Lane pseudovaluation, then the
set of $g \in K[x]$ such that $v(g) = \infty$ is a prime ideal,
generated by $\phi_n$.  Furthermore, since there is a unique way to
extend $v_K$ from $K$ to $K[x]/\phi_n$, an infinite Mac Lane
pseudovaluation can be specified by the ideal it sends to $\infty$.

We collect some basic results on Mac Lane valuations and key
polynomials that will be used repeatedly.  
\begin{lemma}\label{Lfdegree}
Suppose $f$ is a proper key polynomial over $v = [v_0,\, v_1(\phi_1) = \lambda_1,
\ldots,\, v_n(\phi_n) = \lambda_n]$.
\begin{enumerate}[\upshape (i)]
\item If $n = 0$, then $f$ is linear.
  % If $n \geq 1$, then $\phi_1$ is linear.
  Every monic linear polynomial in $\mc{O}_K[x]$ is a key polynomial over $v_0$.
\item If $n \geq 1$, and $f = \phi_n^e + a_{e-1}\phi_n^{e-1} + \cdots
  + a_0$ is the $\phi_n$-adic expansion of $f$, then $v_n(a_0) =
  v_n(\phi_n^e) = e\lambda_n$, and $v_n(a_i\phi_n^i) \geq e\lambda_n$
  for all $i \in \{1, \ldots, e-1\}$. In particular, $v_n(f) = e\lambda_n$.
\item If $n \geq 1$, then $\deg(f)/\deg(\phi_n) = e(v_n/v_{n-1})$.
\end{enumerate}
\end{lemma}

\begin{proof}
For (i) and (iii), see \cite[Lemma 2.10]{OShoriz}. For (ii), see \cite[Lemma 2.2]{OShoriz}. 
\end{proof}

\begin{corollary}\label{CvalueofN}
Let $v = [v_0,\, v_1(\phi_1) = \lambda_1,
\ldots,\, v_n(\phi_n) = \lambda_n]$ be a Mac Lane valuation of
inductive length $n \geq 1$.
% Write $\lambda_i = b_i/c_i$ in lowest terms for all $i$.
Then, for all $1 \leq j \leq n$, we have $\deg(\phi_j) = e_{v_{j-1}}$.  In
particular, $\deg(\phi_n) = e_{v_{n-1}}$.
\end{corollary}

\begin{proof}
  See \cite[Corollary 2.12]{OShoriz}.
\end{proof}

\begin{example}\label{Ebasickey}
If $K = \text{Frac}(W(\ol{\mathbb{F}}_3))$, then the polynomial $f(x) = x^3 - 9$ is a
proper key polynomial over $[v_0,\, v_1(x) = 2/3]$.  In
accordance with Lemma~\ref{Lfdegree}(ii), we have $v_1(f) = v_1(9) = v_1(x^3) =
3 \cdot 2/3 = 2$.  If we extend $v_1$ to a valuation $[v_0,\, v_1(x) =
2/3,\, v_2(f(x)) = \lambda_2]$ with $\lambda_2 > 2$, then the
valuation $v_2$ notices ``cancellation'' in $x^3 - 9$ that $v_1$ does not.
\end{example}

\begin{lemma}\label{Lvnvprime}
Let $[v_0,\, v_1(\phi_1) = \lambda_1,\ldots,\, v_n(\phi_n) =
\lambda_n]$ be a valuation over which there exists a proper key
polynomial. If $n \geq 1$, then $e(v_n/v_{n-1}) > 1$.
\end{lemma}
\begin{proof}
  See \cite[Lemma 2.13]{OShoriz}.
\end{proof}

\begin{propdef}\label{Pbestlowerapprox}
  % Let $\alpha \in \mc{O}_{\ol{K}}$, and
  Let $f \in K[x]$ be monic and irreducible.
  % the minimal polynomial for $\alpha$.
  Then there exists a unique Mac Lane valuation $v_f$ over which $f$ is a proper key polynomial. 
\end{propdef}

\begin{proof}
  See \cite[Proposition~2.5]{OShoriz}. 
\end{proof}

\begin{defn}\label{Dinfspeudo}
 If $g \in K[x]$ is monic and irreducible, we write $v_g^{\infty}$ for $[v_g,\, v(g) = \infty]$, the unique infinite Mac Lane pseudovaluation sending $g$ to $\infty$. 
\end{defn}

%\begin{remark}\label{Rinf}
%By Proposition~\ref{Pbestlowerapprox}, there is always a unique Mac
%Lane valuation $v_g$ over which such an irreducible $g$ is a key polynomial.  So
%$v_g^{\infty}$ exists and equals $[v_g,\, v(g) = \infty]$.
%\end{remark}

\begin{prop}\label{Pmaximalforlength}
If $v = [v_0,\ \ldots,\, v_n(\phi_n) = \lambda_n]$ is a Mac Lane
pseudovaluation, and if $w$ is a Mac Lane valuation with $v_i \prec w
\preceq v$ for some $1 \leq i \leq n$, then
the inductive length of $w$ is greater than that of $v_i$.
\end{prop}

\begin{proof}
Since $v_i(\phi_i) = v(\phi_i)$, we have $w(\phi_i) = v_i(\phi_i)$.
The result now follows from \cite[Lemma~4.35]{Ruth}.
\end{proof}

The following lemma is the only place in the paper where we will need
the concept of \emph{diskoid}.  Recall (e.g., \cite[\S2.2]{OShoriz})
that if $\phi \in \mc{O}_K[x]$ is monic and $\lambda \in \rats_{\geq
  0}$, then the diskoid $D(\phi, \lambda)$ is the set $\{\alpha \in
\ol{K} \mid v_K(\phi(\alpha)) \geq \lambda\}$.  It can be thought of as
being ``centered'' at the roots of $\phi$.  By \cite[Theorem
4.56]{Ruth} (see also \cite[Proposition 2.4]{OShoriz}), there is a one-to-one
correspondence between Mac Lane valuations and diskoids inside $\mc{O}_{\ol{K}}$, sending the
valuation $v = [v_0,\, \ldots,\, v_n(\phi_n) = \lambda_n]$ to the
diskoid $D_v := D(\phi_n, \lambda_n)$.  Furthermore, for two Mac Lane
valuations $v$ and $w$, we have $v \preceq w$ if and only if $D_v
\supseteq D_w$.

\begin{lemma}\label{Lgfollows}
  Let $v = [v_0,\, \ldots,\, v_n(\phi_n) = \lambda_n]$ be a Mac Lane
  valuation, and let $g \in K[x]$ be a monic irreducible polynomial
  with a root $\theta \in \ol{K}$.
  Then $v \prec v_g^{\infty}$ if and only if $v_K(\phi_n(\theta)) \geq \lambda_n$.
 In this situation, $\deg(\phi_n) \mid \deg(g)$.
\end{lemma}

\begin{proof}
We have that $v_K(\phi_n(\theta)) \geq \lambda_n$ is equivalent to
$\theta \in D(\phi_n, \lambda_n)$, which is equivalent to $D(g, \lambda) \subseteq
D(\phi_n, \lambda_n)$ for all large enough $\lambda$.  
If we set $v_{g, \lambda} := [v_g, v(g) = \lambda]$, then $D(g,
\lambda) \subseteq D(\phi_n, \lambda_n)$ is equivalent to $v \preceq
v_{g, \lambda}$. Since $v_g^{\infty} = [v_g, v(g) = \infty]$, the statement $v \preceq
v_{g, \lambda}$ for all large enough $\lambda$ is
equivalent to $v \preceq v_g^{\infty}$, and is simultaneously
equivalent to $\theta \in D(\phi_n, \lambda_n)$, proving the equivalence.

By \cite[Remark~4.36]{Ruth}, $v \prec v_{g, \lambda}$ is equivalent to $v_{g, \lambda}$
augmenting $v$.  If this is true for some $\lambda$ (which it is if
the statements in the proposition hold), then
Lemma~\ref{Lfdegree}(iii) shows that $\deg(\phi_n) \mid \deg(g)$.
\end{proof}

The following lemma is extracted from \cite{FGMN}.
\begin{lemma}\label{Ldominantterm}
  Let $v = [v_0,\, \ldots,\, v_n(\phi_n) = \lambda_n]$ be a Mac Lane
  valuation, let $g \in K[x]$ be a monic irreducible polynomial, and let $g
  = \sum_{i=0}^r a_i \phi_n^i$ be the $\phi_n$-adic expansion of $g$.
  \begin{enumerate}[\upshape (i)]
    \item If $v_K(\phi_n(\theta)) \geq \lambda_n$ for one
      (equivalently all) roots $\theta$ of $g$, then $\deg(\phi_n) \mid
      \deg(g)$ and $v(g) = (\deg(g)/\deg(\phi_n))\lambda_n$.
    \item If $v_K(\phi_n(\theta)) \leq \lambda_n$ for one
      (equivalently all) roots $\theta$ of $g$, then $v(g) = v(a_0) =
      v_{n-1}(a_0)$.
  \end{enumerate}
\end{lemma}

\begin{proof}
  Let $\ell = \deg(g)/\deg(\phi_n)$. 
  For part (i), first note that $\ell \in \ints$ by
  Lemma~\ref{Lgfollows}.  Now, noting that $\phi_n$ is a key
  polynomial over $v$, we apply \cite[Theorem~6.2(2)]{FGMN}, taking $F$, $\mu$,
$\phi$ in the notation of that theorem to be $g$, $v$, and $\phi_n$.
This implies that if $v_K(\phi_n(\theta)) > \lambda_n$, then $v(g -
\phi_n^{\ell}) > \max(v(g),
v(\phi_n^{\ell}))$.  By continuity, $v_K(\phi_n(\theta)) \geq
\lambda_n$ implies $v(g - \phi_n^{\ell}) \geq \max(v(g),
v(\phi_n^{\ell}))$.  This implies that $v(g) =
v(\phi_n^{\ell})$.

For part (ii), \cite[Theorem~6.2]{FGMN} implies in this situation that
$\phi_n \nmid_v g$\footnote{For a definition of $\mid_v$, see
  \cite[Definition~1.2]{FGMN}, but we don't need the actual definition
    to proceed.}, which implies by
  \cite[Lemma~1.3(4)]{FGMN}\footnote{The criterion (3) of
    \cite[Lemma~1.3]{FGMN} is satisfied by definition.} that $v(g) =
  v(a_0)$.  By the definition of inductive valuation, it follows that $v(a_0) =
  v_{n-1}(a_0)$.   
\end{proof}

\begin{corollary}\label{Cdominantterm}
  Let $v = [v_0,\, \ldots,\, v_n(\phi_n) = \lambda_n]$ be a Mac Lane
  valuation and let $g \in K[x]$ be a monic irreducible polynomial.
%  \begin{enumerate}[\upshape (i)]
     Suppose $v \prec v_g^{\infty}$.  Then $v(g)
      = v(\phi_n^{\deg(g)/\deg(\phi_n)})$.
%    \end{enumerate}
    
\end{corollary}

\begin{proof}
This follows immediately from Lemmas~\ref{Lgfollows} and
\ref{Ldominantterm}(i). 
\end{proof}

\subsection{Partial order structure: the inf-closed property and neighbors}\label{Sinfclosed}
If $v$ and $w$ are Mac Lane pseudovaluations, we define $\inf(v, w)$
to be the maximal Mac Lane pseudovaluation $x$ such that $x \preceq v$
and $x \preceq w$.  This exists by \cite[Proposition~2.26]{KW}.  Following \cite{KW}, we say that a set $V$ of Mac Lane
pseudovaluations is \emph{inf-closed} if for all $v, w \in V$, we have
$\inf(v, w) \in V$.

\begin{lemma}\label{Linfclosed}
  Suppose $V$ is a set of Mac Lane pseudovaluations, and
  let $u \in V$.  Let $W$ be an inf-closed set of Mac Lane
  pseudovaluations such that $u \preceq w$ for all $w \in W$, and such
  that if $v \in V$ with $u \preceq v$, then $\inf(v, w) = u$ for all
  $w \in W$.
  Under these assumptions, if $V$ is
  inf-closed, then so is $V \cup W$.
\end{lemma}

\begin{proof}
Since $V$ and $W$ are inf-closed, one need only check that $\inf(v, w) \in
V \cup W$ for $v \in V$ and $w \in W$.  If $u \preceq v$, then
$\inf(v, w) = u \in V \subseteq V \cup W$ by assumption. So assume $u \not \preceq v$.  It suffices to show that $\inf(v, w) = \inf(v, u)$, since $\inf(v, u)\in V \subseteq V \cup W$ by our assumption that $V$ is inf-closed.  

Since $u \preceq w$, it follows that $\inf(v, u)\preceq \inf(v, w)$. We now show $\inf(v, w) \preceq \inf(v, u)$. Let $v'$ be a Mac Lane pseudovaluation with $v' \preceq v$ and $v' \preceq w$. We need to show $v' \preceq u$. If $v' \not \preceq u$, then since the set of Mac Lane pseudovaluations bounded above by $w$ is
totally ordered (\cite[Proposition 2.25]{KW}) and $u \preceq w$ and $v' \preceq w$ by assumption, we have $u \prec v'$. Combined with our assumption that $v' \preceq v$, we get $u \prec v' \preceq v$, which contradicts $u \not \preceq v$.  
\end{proof}

Let $V^*$ be a finite set of Mac Lane pseudovaluations, let $V \subset V^*$ be
the subset of all valuations, and let $\mc{Y}$ be
the $V$-model of $\proj^1_K$.  Two pseudovaluations $v, w \in V^*$ are called
\emph{adjacent} in $V$ if $v \prec w$ and there exists no $y \in V^*$ with $v
\prec y \prec w$, or if the same holds with the roles of $v$ and $w$
reversed.  We will often omit mentioning $V^*$ when it is clear.  The
pseudovaluations $w$ adjacent to $v$ in $V^*$ are called $v$'s \emph{neighbors}.
  
\section{Mac Lane valuations and normal models}\label{Smaclanemodels}

%\begin{itemize}
%\item \andrew{This section should contain all of the terminology about
%  forward, backward directions, etc.  Should also give a convenient
%  way to refer to points corresponding to these directions on
%  corresponding $v$-models.  That is, we can talk about a point on a
%  $v$-component ``in the direction'' of another valuation $w$.}
%\end{itemize}

A \emph{normal model} of $\proj^1_K$ is a flat, normal, proper
$\mc{O}_K$-curve with generic fiber isomorphic to $\proj^1_K$.
By \cite[Corollary 3.18]{Ruth},\footnote{See also \cite[Theorems 1.1,
  2.1]{GMP} for a stronger result in more general context, but from
  which it takes a small amount of work to extract the exact statement
  that we want.} normal models $\mc{Y}$ of $\proj^1_K$ are in one-to-one correspondence with
non-empty finite collections of geometric valuations on $K(\proj^1)$, by sending
$\mc{Y}$ to the collection of geometric valuations corresponding to
the local rings at the generic points of the irreducible components of
the special fiber of $\mc{Y}$. We fix a coordinate on
$\proj^1_K$ so that each Mac Lane valuation gives a geometric
valuation (all geometric valuations $v$ we deal with in this paper will
have $v \succeq v_0$, so in fact all geometric valuations we care
about will be Mac Lane valuations, see \S\ref{Smaclane}).  Then, via the correspondence in \cite[Corollary
3.18]{Ruth}, the multiplicity of an irreducible
component of the special fiber of a normal model $\mc{Y}$ of $\proj^1_K$
corresponding to a Mac Lane valuation $v$ equals $e_v$.

We say that a normal model $\mc{Y}$ of $\proj^1_K$ \emph{includes} a Mac Lane
valuation $v$ if a component of the special fiber corresponds to
$v$. If $\mc{Y}$ includes $v$, we call the corresponding irreducible component of its
special fiber the \emph{$v$-component} of the special fiber of
$\mc{Y}$ (or by abuse of language, 
the $v$-component of $\mc{Y}$, even though it is not an irreducible
component of $\mc{Y}$).
If $V$ is a finite set of Mac Lane valuations, then the
\emph{$V$-model of $\proj^1_K$} is the normal model including exactly
the valuations in $V$.  If $V = \{v\}$, we simply say the $v$-model
instead of the $\{v\}$-model.  
Recall that we fixed a coordinate $t$ on $\proj^1_K$, that is, a rational function $t$ on $\proj^1_K$ such that
$K(\proj^1_K) = K(t)$.

\subsection{Specialization of horizontal divisors}\label{Shorizontaldivs}
Each $\alpha \in \ol{K}$ has minimal polynomial $g \in
K[x]$ over $K$, corresponding to a closed point of $\proj^1_K$. If
$\mc{Y}$ is a normal model of $\proj^1_K$, the closure of this point in
$\mc{Y}$ is a subscheme that we call $D_{\alpha}$ or $D_g$, depending
on context; note that
$D_{\alpha}$ is a horizontal divisor (the model will be clear from
context, so we omit it to lighten the notation). We also write
$D_{\infty}$ for the closure of the point at $\infty$ in $\mc{Y}$.  

If $v$ is a Mac Lane valuation, then the reduced special fiber of the
$v$-model of $\proj^1_K$ is isomorphic to $\proj^1_k$ (see, e.g.,
\cite[Lemma~7.1]{ObusWewers}). Roughly, the propositions below means we can ``parameterize'' the
special fiber of the $v$-model of $\proj^1_K$ by the reduction of the
values of $\phi_n/c$, where $c \in \ol{K}$ has valuation $\lambda_n$.

\begin{prop}\label{Pparameterize}Let $v = [v_0,\, v_1(\phi_1) = \lambda_1, \ldots,\, v_n(\phi_n) =
\lambda_n]$ and $v' = [v_0,\, v_1(\phi_1) = \lambda_1, \ldots,\, v_n'(\phi_n) =
\lambda_n']$ be Mac Lane valuations with $\lambda_n < \lambda_n'$. 
\begin{enumerate}[\upshape (i)]
\item\label{Psmallerdiskoid}
Let $\mc{Y}$ be the 
$v$-model of $\proj^1_K$.  As $\alpha$ ranges over $\ol{K}$, all
$D_{\alpha}$ with $v_K(\phi_n(\alpha)) > \lambda_n$ meet on the special
fiber, all $D_{\alpha}$ with $v_K(\phi_n(\alpha)) < \lambda_n$ meet
at a different point
on the special fiber, and no $D_{\alpha}$ with $v_K(\phi_n(\alpha))
\neq \lambda_n$ meets any $D_{\beta}$ with $v_K(\phi_n(\beta)) = \lambda_n$.
\item\label{Pannulus}
Let $\mc{Y}$ be 
a model of $\proj^1_K$ including $v$ and $v'$ on which the $v$- and
$v'$-components intersect, say at a point $z$.  Then $D_{\alpha}$
meets $z$ if and only if $\lambda_n < v_K(\phi_n(\alpha)) < \lambda_n'$.
\end{enumerate}
\end{prop}
\begin{proof}
These are \cite[Proposition 3.2]{OShoriz} and \cite[Corollary 3.4]{OShoriz}.
\end{proof}

We reproduce a result from \cite{KW} that will be used repeatedly in
this paper.

\begin{prop}[{\cite[Proposition~3.5]{KW}}]\label{P35}
  Let $V^*$ be a finite set of Mac Lane pseudo-valuations, let $V
  \subseteq V^*$ be the subset consisting of all valuations, and let $\mc{Y}$ be the
  $V$-model of $\proj^1_K$.  If $v$ and $w$ are neighbors in $V^*$, then
  the $v$- and $w$-components intersect on $\mc{Y}$ (where for a pseudovaluation
  $v = v_g^{\infty}$, we consider the $v$-component to be $D_g$).  The converse is true if $V^*$
  is inf-closed.\footnote{As stated, \cite[Proposition 3.5]{KW} requires that
$V^*$ be inf-closed for both directions, but that assumption is not
used in the proof of
the ``if'' direction.} 
\end{prop}

\begin{prop}\label{PAllspecializations}
 Let $v = [v_0,\, v_1(\phi_1) = \lambda_1, \ldots,\, v_n(\phi_n) =
\lambda_n]$ be a Mac Lane valuation and let $\mc{Y}$ be the 
$v$-model of $\proj^1_K$. 
\begin{enumerate}[\upshape (i)]
 \item\label{Cphiizeroes} Then $D_{\phi_i}$ and $D_{\infty}$ for $i <
n$ meet at 
the same point on the special fiber of $\mc{Y}$.  Furthermore,
$D_{\phi_n}$ does not meet this point.
 \item\label{Clowvalspecialization}  If $g \in \mc{O}_K[t]$ is a monic
irreducible polynomial, then $D_g$ meets $D_{\infty}$ if and only if $v \not \prec v_g^{\infty}$.
\end{enumerate}
\end{prop}
\begin{proof}\hfill
 \begin{enumerate}[\upshape (i)]
  \item Let $\mc{Y}'$ be the model corresponding to $\{ v_i,v \}$. Since $v_i \prec v$, the first result follows from  \cite[Lemma~3.7(ii)]{KW} applied to $\mc{Y}' \rightarrow \mc{Y}$. Since $v < v_{\phi_n}^\infty$, the second result from \cite[Lemma~3.6(iii)]{KW}.
  \item This follows from Lemma~\ref{Lgfollows},
    Proposition~\ref{Pparameterize}\ref{Psmallerdiskoid} and the
    previous part. \qedhere
 \end{enumerate}
\end{proof}

\begin{prop}\label{Pcontractionisomorphism}
Let $S \subseteq W$ be non-empty finite sets of Mac Lane valuations, and let $V
= \{ w \in W \mid \exists s \in S \text{ such that } s\preceq w\}$, so
that $S \subseteq V \subseteq W$.  Let
$\nu \colon \mc{Y}_W \to \mc{Y}_V$ be the birational morphism from the $W$-model
to the $V$-model of $\proj^1_K$ which contracts all $w$-components for
$w \notin V$. Let $z$ be the point where
$D_{\infty}$ meets the special fiber of $\mc{Y}_V$. Then,

\begin{enumerate}[\upshape (i)]
  \item The point $z$ lies on the
    $v$-component of $\mc{Y}_V$ if and only if $v$ is minimal in $V$
    (equivalently $v$ is minimal in $S$).
  \item The morphism $\nu \colon \mc{Y}_W \to \mc{Y}_V$ is an
    isomorphism outside of $\nu^{-1}(z)$.
  \end{enumerate}  
\end{prop}

\begin{proof}
  Let $V' := V \cup \{v_0\}$, let $\mc{Y}_{V'}$ be the $V'$-model of
  $\proj^1_K$, and let $\mc{Y}_0$ be the $v_0$-model of $\proj^1_K$.
Since $v_0$ is minimal in $V'$, \cite[Lemma~3.7(ii)]{KW} with $v_0 =
v$ in that lemma\footnote{There is a typo in \cite[Lemma 3.7(ii)]{KW}
  --- it should read ``$\phi_v$ contracts the vertical component $E_{v'}$ to a
  closed point \ldots''} shows
    that $D_{\infty}$ on $\mc{Y}_{V'}$ does not meet the image of the
    exceptional locus of the contraction $\mc{Y}_{V'} \to
    \mc{Y}_0$.  So $D_{\infty}$ meets only the $v_0$-component of
    $\mc{Y}_{V'}$.  If $v_0 \in S$, so that $v_0 \in V$ and $V = V'$, this proves part
    (i).  If not, then 
    Proposition~\ref{P35} shows that the $v_0$-component and the
    $v$-component of
    $\mc{Y}_{V'}$ meet if and only
      if $v$ is minimal in $V$.  Since contracting the $v_0$-component
      of $\mc{Y}_{V'}$ yields $\mc{Y}_V$,
    we see that the $v$-component of $\mc{Y}_V$ meets $D_{\infty}$ if
    and only if $v$ is minimal in $V$, and this meeting is at $z$.  Observing that, by construction, the minimal valuations in $V$ are exactly
    the minimal valuations in $S$, the proof of part (i) is complete.

    Let $w \in W \setminus V$. By construction, for all $v \in V$,
    $w \not \preceq W$.  Take $v$ minimal in $V$, and let $\mc{Y}_{v}$
    be the $v$-model of $\proj^1_K$.  Consider the composition of
    morphisms $\mc{Y}_W \stackrel{\nu}{\to} \mc{Y}_V \stackrel{g}{\to} \mc{Y}_v$,
    where $g$ contracts all components except the $v$-component.  By
    \cite[Lemma~3.7(i)]{KW}, $g$ is a homeomorphism on the
    $v$-component, and by \cite[Lemma~3.7(ii)]{KW}, $g \circ \nu$
    contracts the $w$-component to the speicalization of $D_{\infty}$
    on $\mc{Y}_v$.  Combining these two assertions shows that $\nu$
    contracts the $w$-component to $z$, which proves part (ii). 
\end{proof}

\begin{corollary}\label{CAllspecializations}\hfill
  \begin{enumerate}[\upshape (i)]
    \item\label{Cinftyspecialization}
Suppose $V \subseteq W$ are finite sets of Mac Lane valuations such that $V$
has a unique minimal valuation $v$, and let $\nu \colon \mc{Y}_W \to
\mc{Y}_V$ be the birational morphism from the $W$-model to the
$V$-model of $\proj^1_K$ which contracts all $w$-components for $w
\notin V$.  The specialization $z$ of
$D_{\infty}$ lies only on the $v$-component of $\mc{Y}_V$, and $\nu$
is an isomorphism outside of $\nu^{-1}(z)$.   
\item\label{Cinftyspectwominimal}
Suppose $V \subseteq W$ are finite sets of Mac Lane valuations such that $V$
has two minimal valuations $v$ and $v'$, and let $\nu \colon \mc{Y}_W \to
\mc{Y}_V$ be the birational morphism from the $W$-model to the
$V$-model of $\proj^1_K$ which contracts all $w$-components for $w
\notin V$.  The specialization $z$ of
$D_{\infty}$ lies at the intersection of the $v$- and $v'$-components of $\mc{Y}_V$, and $\nu$
is an isomorphism outside of $\nu^{-1}(z)$.   
  \end{enumerate}
\end{corollary}

\begin{proof}
  Part (i) (resp.\ part (ii)) follows from Proposition~\ref{Pcontractionisomorphism},
  taking $S = \{v\}$ (resp.\ $S = \{v, v'\}$).
\end{proof}

\begin{prop}\label{Pbranchspecialization}
Suppose $V$ is a finite set of Mac Lane valuations, and $\mc{Y}$
is the $V$-model of $\proj^1_K$.  Let $g$ be a monic irreducible
polynomial in $\mc{O}_K[x]$, and suppose that there exists $w \in V$
such that $w \prec v_g^{\infty}$.  
Then among those $w \in V$ such that $w \prec
v_g^{\infty}$, there is a unique maximal one $v$, and the divisor $D_{g}$
meets the special fiber of $\mc{Y}$ (only) on the $v$-component.
\end{prop}

\begin{proof}
The existence and uniqueness of $v$ follow from \cite[Proposition
2.25]{KW}.  The rest of the proposition is immediate from Proposition~\ref{P35}
applied to the valuations $v_g^{\infty}$
and $v$, with $V^*$ in Proposition~\ref{P35} equal to $V \cup
V_g^{\infty}$.
\end{proof}

\subsection{Standard crossings and finite cusps}\label{Sadjacency}

In this subsection, we define two special types of closed points on $\mc{Y}$, which
figure prominently in the rest of the paper:
\begin{defn}\label{Dstandardcrossing} \hfill
  \begin{enumerate}[\upshape (i)]
\item  A \emph{standard crossing} is a point $y \in \mc{Y}$ lying on exactly two
irreducible components of the special fiber, whose corresponding Mac
Lane valuations are $v = [v_0,\, v_1(\phi_1) = \lambda_1,\, \ldots,\,
v_{n-1}(\phi_{n-1}) = \lambda_{n-1},\, v_n(\phi_n) = \lambda_n]$ and $v' =
[v_0,\, v_1(\phi_1) = \lambda_1,\, \ldots,\,
v_{n-1}(\phi_{n-1}) = \lambda_{n-1},\, v_n'(\phi_n) = \lambda_n']$, with
$\lambda_n < \lambda_n'$.  We allow the possibility that $v =
v_{n-1}$, so that $v$ is not necessarily minimally presented (but
$v_{n-1}$ is, as is $v'$).
\item A \emph{finite cusp} is a non-regular point $y \in \mc{Y}$ lying on exactly one
  irreducible component of the special fiber, such that $y$ does not
  lie on $D_{\infty}$.
  %whose corresponding Mac
 % Lane valuation (minimally presented) is $[v_0,\, v_1(\phi_1) = \lambda_1,\, \ldots,\,
%v_n(\phi_n) = \lambda_n]$, such that $y$ is where $D_{\phi_n}$ meets
%the special fiber $\mc{Y}$.
%\item If the set of Mac Lane valuations corresponding to the
 % irreducible components of the special fiber of $\mc{Y}$ has a unique
%  minimal element $v$ for $\preceq$, then the \emph{standard
 %   $\infty$-specialization} is the point $y \in \mc{Y}$ that is the
%  specialization of $\infty$ to the irreducible component corrsponding to $v$.
\end{enumerate}
\end{defn}

We show that what will be called a ``standard
$\infty$-crossing'' (see \S\ref{Sinftycrossing}) is just a standard
crossing under a change of variables.

\begin{prop}\label{Pchangeofvariable}
Let $c, c' \in \mc{O}_K$ with $v_K(c - c') = 0$, let $\mu,
\mu' \in \rats_{>0}$, and let $\alpha
\in \nats$ such that $\alpha > \mu$.  Under the
change of variable $u = \pi^{\alpha}_K(t-c')/(t-c)$, we have
$$[v_0,\, v_1(t - c') = \mu'] = [v_0,\, v_1(u) = \alpha + \mu']$$
and
$$[v_0,\, v_1(t-c) = \mu] = [v_0,\, v_1(u) = \alpha - \mu].$$
\end{prop}

\begin{proof}
We prove the second equality --- the proof of the first one is
similar and easier.  Suppose $f = \sum_{i=0}^r a_i u^i$ is a polynomial in
$K[u]$.  Letting $v = [v_0,\, v_1(u) = \alpha - \mu]$, we have that
$v(f) = \min_i (v_K(a_i) + (\alpha - \mu)i)$.  Writing $f$ in terms of
$t$ and multiplying by $(t-c)^r$, we obtain
\begin{align*}
  (t - c)^rf &= \sum_{i=0}^r
                 a_i\pi_K^{\alpha i} (t-c')^i(t - c)^{r - i} \\
  &= \sum_{i=0}^r a_i\pi_K^{\alpha i} (t-c + c - c')^i(t -
    c)^{r - i} \\
  &= \sum_{i=0}^r a_i \pi_K^{\alpha i} \left((c - c')^i(t-c)^{r - i} +
    O((t - c)^{r-i+1})\right)
\end{align*}
So letting $w =
[v_0, v_1(t-c) = \mu]$, we have
$$w(f) = -\mu r + \min_i (v_K(a_i) + \alpha i + \mu(r-i))
= \min_i (v_K(a_i) + (\alpha - \mu)i). $$
So $v(f) = w(f)$.  Since $v = w$ on $K[u]$, they are equal on $K(u)$.
\end{proof}

\subsubsection{Location of standard crossings and finite cusps}

Note that by
Proposition~\ref{P35}, the two Mac Lane
valuations making a standard crossing are adjacent in $V$. The
converse is not true in general. For example
the valuations $v_0$ and $v \colonequals [v_0,v_1(x) = 2/3,v_2(x^3-2)
= 2]$ are adjacent in the $\{v_0,v\}$-model, but do not form a
standard crossing. However, under the following assumption, the
converse is true.

\begin{lemma}\label{Ladjacentstandard}
Suppose that for each valuation in $V$, all its predecessors are in
$V$ as well.  Then every adjacent pair of valuations $v \prec w
\in V$ forms a standard crossing in the $V$-model of $\proj^1_K$.
%(although this may require presenting
%$v$ as an inductive valuation non-minimally).
\end{lemma}

\begin{proof}
Since $v$ and $w$ are adjacent, the corresponding components
intersect. It suffices to show that $v$ and $w$ have a presentation as
in Definition~\ref{Dstandardcrossing}.  Write $w = [w_0 := v_0,\, \ldots,\, w_n(\phi_n) = \lambda_n]$.  Then $w_{n-1}$ is a predecessor of $w$, so by
assumption we have $w_{n-1} \in V$, which means $w_{n-1} \preceq v \prec w$.  If $w_{n-1} = v$,
then we can write $v = [w_{n-1}, v_n(\phi_n) = w_{n-1}(\phi_n)]$ and
$w_{n-1}(\phi_n) < \lambda_n$, proving the lemma (here $v$ is
presented non-minimally as an inductive valuation).  If not, we know
in any case that $v(\phi_{n-1}) = \lambda_{n-1}$.  So by
\cite[Proposition 4.35 and Remark 4.36]{Ruth} applied to $w_{n-1}$ and $v$, the
valuation $v$ is an augmentation of $w_{n-1}$.  By \cite[Proposition
4.35 and Remark 4.36]{Ruth} applied to $v$ and $w$, the
augmentation must be by $\phi_n$, so $v = [w_{n-1}, v_n(\phi_n) =
\lambda']$ with $w_{n-1}(\phi_n) < \lambda' < \lambda$, proving the lemma.
\end{proof}

\begin{corollary}\label{Cminlength1}
Suppose $v_0 \in V$, and for each valuation in $V$, all its
predecessors are in $V$ as well.  If $v$ is adjacent to $v_0$ in $V$, then $v$ has inductive length 1.
\end{corollary}

\begin{proof}
By Lemma~\ref{Ladjacentstandard}, $v_0 \prec v$ forms a standard
crossing.  By the definition of standard crossing, this happens only
if $v$ has inductive length $1$.
\end{proof}

\begin{lemma}\label{Lallpredincl} Let $V_1$ be the set of all predecessors
  of a finite set of Mac Lane pseudovaluations, and let $V_2$ be the
  inf-closure of $V_1$.
  If $v$ is a predecessor of a valuation in $V_2$, then $v \in V_2$.
\end{lemma}
\begin{proof}
Suppose $v = \inf(w, w')$ with $w, w' \in V_1$.  Since $v \preceq w$,
\cite[Proposition 4.35]{Ruth} shows that every predecessor of $v$
(other than possibly $v$ itself, which is in $V_2$) is a predecessor
of $w$.  Since $w \in V_1 \subseteq V_2$, all its predecessors are as
well.  Thus, in either case, $v \in V_2$.
\end{proof}

\begin{lemma}\label{Lstandardendpointunique}\label{Rwhereisthecusp}
  Let $v$ be a valuation of inductive length $n$ with length
  $n-1$ predecessor $v_{n-1}$.  If $e_v > e_{v_{n-1}}$, then the $v$-model of $\proj^1_K$
    has a unique finite cusp at the point where $D_{\phi_n}$ meets the
  special fiber.  If $e_v = e_{v_{n-1}}$, then the $v$-model
    of $\proj^1_K$ does not have a finite cusp.
 \end{lemma}

\begin{proof}
This is \cite[Lemma~7.3]{ObusWewers}.
\end{proof}
\begin{comment}
By \cite[Lemma~7.3]{ObusWewers}, $e_v > e_{v_{n-1}}$ is equivalent to the $v$-model having
a singularity outside of where $D_{\infty}$ meets the special
fiber, and this singularity is unique.
\end{comment}
%In case (ii), suppose $v = [v_0,\, \ldots,\, v_n(\phi_n) =
%\lambda_n]$.  Since $\Gamma_v = \Gamma_{v_{n-1}}$, there is a monomial
%$h$ of degree less than or equal to $\deg(\phi_n)$ such that 
%in $\phi_1, \ldots,\, \phi_{n-1}$ over $K$ such that $v(h) =
%\lambda_n$.  Then, replacing $\phi_n$ with $\phi_n + ch$ for any $c
%\in \mc{O}_K$ yields the same valuation $v$, see, e.g.,
%\cite[Theorem~4.33]{Ruth}.  On the other hand, we claim that for any
%point on the special fiber of the $v$-model other than the
%specialization of $\infty$, there exists a $c \in \mc{O}_K$ such that 
%$D_{\phi_n + ch}$ meets the special fiber of the $v$-model at this
%point. 

\begin{comment}
\begin{remark}\label{Rwhereisthecusp}
  If $v = [v_0,\, \ldots,\, v_n(\phi_n) = \lambda_n]$ is minimally
  presented and $e_v > e_{v_{n-1}}$, then the finite cusp on the
  $v$-model of $\proj^1_K$ guaranteed by
  Lemma~\ref{Lstandardendpointunique} is where $D_{\phi_n}$ meets the
  special fiber, see \cite[Lemma~7.3(i)]{ObusWewers}.
\end{remark}
\end{comment}

\begin{corollary}\label{Cmaximalcusp}
Let $v = [v_0,\, \ldots,\, v_n(\phi_n) = \lambda_n] \in V$, and assume $e_v
> e_{v_{n-1}}$.
If all $w \in V$ with $w \succeq v$ satisfy $w(\phi_n) =
\lambda_n$, then the $V$-model of $\proj^1_K$ has a (unique) finite
cusp on the $v$-component, and $D_{\phi_n}$ meets this finite cusp.  In particular, this holds if $v$ is maximal in $V$.
\end{corollary}

\begin{proof}
  Observe that if $v \prec w \prec v_{\phi_n}^{\infty}$, then $w(\phi_n)
> \lambda_n$.  So $w \not \prec v_{\phi_n}^{\infty}$  if $v \prec w$ by the assumption that $w(\phi_n) =
\lambda_n$.  Thus $v$ is
maximal among those valuations in $V$ bounded above by
$v_{\phi_n}^{\infty}$.  Proposition~\ref{Pbranchspecialization} shows that $D_{\phi_n}$ meets the special
  fiber of the $V$-model of $\proj^1_K$ only on the $v$-component.  By Lemma~\ref{Lstandardendpointunique}, this meeting point
  is the unique finite cusp of the $v$-component.
\end{proof}

\begin{corollary}\label{Cstandardendpointguaranteed}
Suppose that for each valuation in $V$, all its predecessors are in
$V$ as well.  If $v$ has only one neighbor $w \succ v$, and if the inductive length of $w$ is greater than that of
$v$, then $v$ has a (unique) finite cusp on the $V$-model of
$\proj^1_K$.  
\end{corollary}

\begin{proof}
By Lemma~\ref{Ladjacentstandard}, $v \prec w$ forms a standard
crossing in the $V$-model of $\proj^1_K$.  Since $w$ has inductive
length greater than that of $v$, we can write $v = [v_0,\, \ldots,\,
v_{n-1}(\phi_{n-1}) = \lambda_{n-1} ,\, v_n(\phi_n) = \lambda_n]$ and
$w = [v_0,\, \ldots,\,
v_{n-1}(\phi_{n-1}) = \lambda_{n-1} ,\, w_n(\phi_n) = \lambda_n']$
with $w$ presented minimally and $v = v_{n-1}$.  So $\phi_n$ is a
proper key polynomial over $w_{n-1} = v_{n-1} = v$, which means that $e_v =
e_{v_{n-1}} > e_{v_{n-2}}$ by Lemma~\ref{Lvnvprime}. Furthermore,
$w(\phi_{n-1}) = \lambda_{n-1} = v(\phi_{n-1})$.  We conclude
using Corollary~\ref{Cmaximalcusp} applied to $v = v_{n-1}$ that
the $v$-component has a unique finite cusp on the $V$-model of $\proj^1_K$.
\end{proof}

We also state a lemma here for future use about horizontal divisors that do
\emph{not} intersect special points and/or each other.

\begin{lemma}\label{Lnonspecialize}
\hfill
\begin{enumerate}[\upshape (i)]
  \item Suppose $y \in \mc{Y}$ is a standard crossing, lying on two
    irreducible components with corresponding Mac Lane valuations $v =
    [v_0,\, v_1(\phi_1) = \lambda_1,\, \ldots,\,
v_n(\phi_n) = \lambda_n]$ and $v' = [v_0,\, v_1(\phi_1) = \lambda_1,\, \ldots,\,
v_n(\phi_n) = \lambda_n']$.  Then $D_{\phi_i}$ does not
meet $y$ for any $1 \leq i \leq n$.

   \item Let $v = [v_0,\, v_1(\phi_1) =
    \lambda_1,\, \ldots,\, v_n(\phi_n) = \lambda_n]$ and let $\mc{Y}$
    be a normal model of $\proj^1_K$ including $v$.  Then
    $D_{\phi_i}$ does not meet $D_{\phi_n}$ on $\mc{Y}$ for any $1
    \leq i \leq n-1$.
 \item Suppose $g$ is a proper key polynomial over $v = [v_0,\, v_1(\phi_1) =
   \lambda_1,\, \ldots,\, v_n(\phi_n) = \lambda_n]$, and let $\mc{Y}$
   be a normal model of $\proj^1_K$ including $v$.  Then
   $D_{\phi_i}$ does not meet $D_g$ on $\mc{Y}$ for any $1 \leq i \leq n$.
  \end{enumerate}
\end{lemma}

\begin{proof} 
%\padma{Move parts (ii) and (iii) to before standard crossing, and adjust references.}
In case (i), Proposition~\ref{Pparameterize}(ii) shows that if $\alpha \in
\mc{O}_K$, then $D_{\alpha}$ meets
$y$ if and only if
\begin{equation}\label{Einbetween}
  \lambda_n < v_K(\phi_n(\alpha)) < \lambda_n'.
\end{equation}
In particular, if $D_{\phi_i}$ meets $y$ and $\alpha_i$ is a root of
$\phi_i$, then $v_K(\phi_n(\alpha_i)) > \lambda_n$, so
Proposition~\ref{Pparameterize}(i) shows that $D_{\phi_n}$ and $D_{\phi_i}$
meet on the $v$-model of $\proj^1_K$.  By Proposition~\ref{PAllspecializations}(i),
the only possibility is $i = n$.  But this contradicts
$v_K(\phi_n(\alpha_i)) < \lambda_n'$, proving (i).

%$By \cite[Corollary 2.11]{OShoriz}, if $\alpha_i$ is a root of $\phi_i$,
%then $v_K(\phi_n(\alpha_i)) = \lambda_i$ for $i < n$, and
%$v_K(\phi_n(\alpha_n)) = \infty$.  Since $\lambda_n > \lambda_{n-1}$,
%Equation (\ref{Einbetween})
%is not satisfied for $\alpha = \alpha_i$ for any $i$ from $1$ to $m$,
%which proves part (i).$

Part (ii) follows immediately from Proposition~\ref{PAllspecializations}(i).

% , let $\alpha_i$ be a root of $\phi_i$ for $1 \leq i \leq
%n$.  Then the lemma follows from \cite[Corollary 3.3]{OShoriz}, taking
%$\beta = \alpha_i$ and $\alpha = \alpha_n$ in that lemma.

For part  (iii), if $\beta$ is a root of $g$, then by
\cite[Corollary~2.8]{OShoriz}, $v_K(\phi_n(\beta)) = v(\phi_n) =
\lambda_n$.  On the other hand, if $\alpha_n$ is a root of $\phi_n$,
then $v_K(\phi_n(\alpha_n)) = \infty > \lambda_n$.  Also, by
Proposition~\ref{PAllspecializations}(i), all $D_{\phi_i}$ with $1 \leq i \leq n-1$
meet $D_{\infty}$ on the
$v$-model of $\proj^1_K$, which means by
Proposition~\ref{Pparameterize}(i) that $v_K(\phi_n(\alpha_i)) <
\lambda_n$ for $\alpha_i$ a root of $\phi_i$.  By
Proposition~\ref{Pparameterize}(i) applied to $\alpha_i$ and $\beta$, no
$D_{\phi_i}$ meets $D_g$ on the $v$-model of $\proj^1_K$ for any $1
\leq i \leq n$, and thus the same is true for any model including $v$. 
\end{proof}

\subsubsection{Some explicit $\mathbb{Q}$-Cartier divisors and their intersection multiplicities}

\begin{prop}\label{Pstandardcrossingmultiplicity}
  Suppose $y \in \mc{Y}$ is a standard crossing, lying on two
    irreducible components with corresponding Mac Lane valuations $v =
    [v_0,\, v_1(\phi_1) = \lambda_1,\, \ldots,\,
v_n(\phi_n) = \lambda_n]$ and $v' = [v_0,\, v_1(\phi_1) = \lambda_1,\, \ldots,\,
v_n(\phi_n) = \lambda_n']$, with $\lambda_n < \lambda_n'$. Let $N \colonequals
e_{v_{n-1}}$. Let $D_1$ and $D_{2}$ be the irreducible
divisors of $\mc{Y}$ corresponding to $v$ and $v'$.
\begin{enumerate}[\upshape (i)]
 \item\label{Rextract} There exist $h\in K(Y)$ and an integer $a$ such that $\divi(h) = aD_2$ in $\Spec
  \hat{\mc{O}}_{\mc{Y},y}$ and $(D_1, aD_2)_y = 1$ (in particular,
  $D_2$ is $\rats$-Cartier). Such an $a$ is
  minimal amongst $a' \in \nats$ such that $a'D_2$ is principal at
  $y$.  
\end{enumerate}
 Now, assume $y \in \mc{Y}$ lies on a single irreducible component of
    the special fiber with reduced divisor $D$ and corresponding Mac
    Lane valuation $v = [v_0,\, v_1(\phi_1) = \lambda_1,\, \ldots,\,
    v_n(\phi_n) = \lambda_n]$.
    \begin{enumerate}[\upshape (i)]
      \setcounter{enumi}{1}
   \item   Suppose that $y = D_{\phi_n}
    \cap D$.
Then there exists $h \in K(Y)$ such that
    $h|_{D}$ has a simple
    zero at $y$, and such that
    $\divi(h) = aD_{\phi_n}$ when restricted to $\Spec \hat{\mc{O}}_{\mc{Y}, y}$, where
  %  $D_{\phi_n}$ is as in Lemma~\ref{Lorder1}(ii) and
    $a \in \nats$ is minimal such that $aD_{\phi_n}$ is locally
    principal at $y$.
  \item Suppose that $g$ is a proper key polynomial over
    $v$ such that $y = D_g \cap D$, and $\deg(g) = e\deg(\phi_n)$.
    Letting $h = g/\phi_n^e$, we have that $h|_{D}$ has a simple zero at $y$, and $\divi(h) = D_g$ when restricted to
    $\Spec \hat{\mc{O}}_{\mc{Y},y}$.
  \end{enumerate}

\end{prop}
\begin{proof}
  We begin with part (i).  By \cite[Lemma 3.1]{OShoriz} applied to $v$, there exists a monomial $t$ in $\phi_1, \ldots,
\phi_{n-1}$ such that if $e \colonequals e(v_n/v_{n-1}) = e_v/N$ and $h \colonequals t \phi_n^e$, then $v(h) = 0$ and $h|_{D_1}$ has a simple zero at the specialization of $D_{\phi_n}$ to the $v$-model of
$\proj^1_K$. Since $h|_{D_1}$ has a simple zero at $y$, by definition
$(D_1, \divi(h))_y = 1$. By Proposition~\ref{Pparameterize}(i), (ii), the specialization of $D_{\phi_n}$ to the $v$-model is the image of
$y$ under the contraction of the $v'$-component of the $\{v,
v'\}$-model of $\proj^1_K$ (it is the point where all $D_{\alpha}$
with $v_K(\phi_n(\alpha)) > \lambda_n$ specialize).  
By Lemma~\ref{Lnonspecialize}(i) and Proposition~\ref{Pparameterize}(ii), $\divi(h)$ has no
horizontal part at $y$.  Since $v(h) = 0$, the divisor $D_v$ is not in
the support of $\divi(h)$. Combining the last two sentences, we get that $\divi(h) = aD_{2}$ in $\Spec
  \hat{\mc{O}}_{\mc{Y},y}$ for some integer $a$.
  
  Since $(D_1,aD_2)_y=1$, if $(D_1,a'D_2)_y \in \ints$, then $a'$ is a multiple of $a$. To prove minimality of $a$, note that if $a'D_2$ is a principal divisor at $y$, then $a'D_2$ gives a $\ints$-divisor when restricted to $D_1$, and
the coefficient of $[y]$ in $a'D_2|_{D_1}$ is the integer
$(D_1,a'D_2)$ by definition.

%\andrew{I don't think we need this part anymore? This proves part (i) when $D_1$ corresponds to $v$.
%If $D_1$ corresponds to $v'$, then the proof above goes through using
%$1/h$, where $h = t\phi_n^e$ as in \cite[Lemma
%3.1]{OShoriz} applied to $v'$. }

\begin{comment}
\andrew{Everything below here in this proof needs to be stitched in
  correctly.}
For part (i), suppose that the two irreducible components $D_1,D_2$ of the
special fiber passing through $y$ correspond to Mac Lane valuations $v
:= [v_0,\, v_1(\phi_1) = \lambda_1,\, \ldots,\,
v_n(\phi_n) = \lambda_n]$ and $v' := [v_0,\, v_1(\phi_1) = \lambda_1,\, \ldots,\,
v_n(\phi_n) = \lambda_n']$ respectively, with $\lambda_n < \lambda_n'$.   First
suppose that $D_1$ corresponds to $v$.  Let $h$ be as in
Proposition~\ref{Pstandardcrossingmultiplicity}(i), so that $\divi(h) = aD_2$ when restricted to $\Spec
\hat{\mc{O}}_{\mc{Y},y}$ for some $a \in \nats$ and $(D_1, aD_2)=1$. This shows that if $(D_1,a'D_2) \in \ints$, then $a'$ is a multiple of $a$. To prove minimality of $a$, note that if $a'D_2$ is a principal divisor at $y$, then $a'D_2$ gives a $\ints$-divisor when restricted to $D_1$, and
the coefficient of $[y]$ in $a'D_2|_{D_1}$ is the integer $(D_1,a'D_2)$ by definition.  This proves part (i) when $D_1$ corresponds to $v$.
If $D_1$ corresponds to $v'$, then the proof above goes through using
$1/h$, where $h = t\phi_n^e$ as in \cite[Lemma
3.1]{OShoriz} applied to $v'$. \padma{Moved part of this into the previous prop}
\end{comment}

For part (ii), take $h$ as in part (i) with $D$ in place of $D_1$.   
By Lemma~\ref{Lnonspecialize}(ii), the horizontal
part of $\divi(h)$ at $y$ is supported on 
$D_{\phi_n}$.  So $\divi(h) = aD_{\phi_n}$ at $y$ for some $a \in
\nats$, and the rest of the proof proceeds exactly as in part (i).

To prove part (iii), note that the intersection number of $D_g$ with
the special fiber $\ol{Y}$ of $\mc{Y}$ is $\deg g$, and the
multiplicity of $D$ in $\ol{Y}$ is $e_v$.  So $$(D, D_g) =
\frac{(\ol{Y}, D_g)}{e_v} = \frac{\deg(g)}{e_v} = 1,$$ with the last
equality following from Corollary~\ref{CvalueofN} applied to $[v,
v_{n+1}(g) = \lambda_{n+1}]$ for any $\lambda_{n+1}$.  By \cite[Lemma
4.19(iii)]{Ruth}, $v(g) = ev(\phi_n)$.  So $\divi(h)$ has no vertical
part, and by Lemma~\ref{Lnonspecialize}(iii),  $\divi(h) = D_g$ on
$\Spec \hat{\mc{O}}_{\mc{Y},y}$.  Since $(D, \divi(h)) = (D, D_g) =
1$, we have that $h|_D$ has a simple zero at $y$. 
\end{proof}

\begin{remark}
Note that Proposition~\ref{Pstandardcrossingmultiplicity}(ii) applies to finite cusps by 
Lemma~\ref{Lstandardendpointunique}.
\end{remark}

As a Corollary to Proposition~\ref{Pstandardcrossingmultiplicity}, we calculate the intersection multiplicity (as in \S\ref{Sintersectiontheory}) of the two prime
vertical divisors in a standard crossing.

\begin{corollary}\label{Lintersectionnumber}
In the situation of Proposition~\ref{Pstandardcrossingmultiplicity}(i), $(D_1, D_2)_y
= \frac{N}{(\lambda_n' - \lambda_n)e_ve_{v'}}$.
\end{corollary}

\begin{proof}
Taking $h$ as in Proposition~\ref{Pstandardcrossingmultiplicity}(i),
and combining $\divi(h) = aD_{2}, v(h) = 0$ and $v(t) = v'(t)$, we get 
\begin{equation*}\begin{split} a = e_{v'}v'(h) = e_{v'}(v'(h) - v(h)) = e_{v'}(v'(t \phi_n^e) - v(t \phi_n^e)) = e_{v'}e(\lambda_n' -
\lambda_n). \end{split}\end{equation*} Since $\divi(h) = aD_{2}$ in $\Spec \hat{\mc{O}}_{\mc{Y},y}$, since $(D_1, \divi(h))_y = 1$ and $e=e_v/N$, $$(D_1, D_{2})_y = \frac{1}{a}(D_1, \divi(h))_y = \frac{N}{(\lambda_n' - \lambda_n)e_ve_{v'}}.\qedhere$$
\end{proof}

%\begin{remark}
%The key content of Lemma~\ref{Lminimalcartier} above is that $h$ can
%be chosen so that its divisor has no horizontal part passing through
%$y$, except perhaps a multiple of $D_{\phi_n}$ in case (ii).  
%\end{remark}

\begin{lemma}\label{Lorder1}
\hfill
\begin{enumerate}[\upshape (i)]
  \item Suppose $y \in \mc{Y}$ is a standard crossing, lying on two
    irreducible components of the special fiber with reduced divisors
    $D_1$ and $D_2$.  Then there exist $h \in K(Y)$ and an integer
    $c$ such that $\divi(h) = D_1 + cD_2$ when restricted to $\Spec
    \hat{\mc{O}}_{\mc{Y}, y}$.
   \item Suppose $y \in \mc{Y}$ lies on a single
    irreducible component of the special fiber with reduced divisor 
    $D$ and corresponding Mac Lane valuation $v = [v_0,\, v_1(\phi_1) =
    \lambda_1,\, \ldots,\, v_n(\phi_n) = \lambda_n]$.  Furthermore,
    suppose that $y = D_g \cap D$, where either $g = \phi_n$ or $g$ is a proper key polynomial over $v$.  Then there
    exists $h \in K(Y)$ and an integer $c$ such that $\divi(h) = D +
    cD_{g}$ when restricted to $\Spec
    \hat{\mc{O}}_{\mc{Y}, y}$.
    %, where $D_{\phi_n}$ is the closure in $\mc{Y}$ of $\divi_0(\phi_n)$ in $Y$.
  \end{enumerate}
\end{lemma}

\begin{proof}
First, suppose $y$ is a standard crossing, and the two irreducible
components of the special fiber it lies on have corresponding Mac Lane
valuations $[v_0,\, v_1(\phi_1) = \lambda_1,\, \ldots,\,
v_n(\phi_n) = \lambda_n]$ and $[v_0,\, v_1(\phi_1) = \lambda_1,\, \ldots,\,
v_n(\phi_n) = \lambda_n']$, with $\lambda_n < \lambda_n'$.  Let $\phi$
be a monomial in $\phi_1, \ldots, \phi_n$ such that 
$v(\phi) = 1/e_v$. Lemma~\ref{Lnonspecialize}(i) shows that no $D_{\phi_i}$ has a horizontal
part passing through $y$, so $\divi(\phi)$ in $\Spec
    \hat{\mc{O}}_{\mc{Y}, y}$ is purely vertical.  Since $\divi(\pi_K)$ contains $D_1$ with multiplicity
$e_v$ by \cite[Lemma 5.3(ii)]{ObusWewers} and $v(\pi_K) = 1$ by definition, $\divi(\phi)$ on $\mc{Y}$ contains $D_1$ with
multiplicity $1$.  Taking $h = \phi$ proves part (i).

Next, suppose we are in case (ii).  If $g = \phi_n$, we
construct a
monomial $\phi$ in $\phi_1, \ldots, \phi_n$ as in the previous case such that $\divi(\phi)$
contains $D$ with multiplicity $1$.  Furthermore,
Lemma~\ref{Lnonspecialize}(ii) shows that no $D_{\phi_i}$
for $1 \leq i \leq n-1$ passes through $y$.
Since the horizontal part of $\divi(g)$ passing through $y$ is
$D_g = D_{\phi_n}$,
taking $h = \phi$ proves part (ii).

If, instead, $g$ is a proper key polynomial over $v$, we write $v =
[v_0,\, \ldots,\, v_n(\phi_n) = \lambda_n, v_{n+1}(g) =
\lambda_{n+1}]$, where $\lambda_{n+1} = v(g) = v_n(g)$.  The argument
in the previous paragraph now carries through exactly, using
Lemma~\ref{Lnonspecialize}(iii) instead of Lemma~\ref{Lnonspecialize}(ii).  
\end{proof}

\section{Smoothness of closed points on vertical prime divisors in cyclic covers}\label{Ssmoothvertical}
Let $\mc{Y}$ be a normal model of $Y :=
\proj^1_K$, and let $d \in \nats$ be prime to $\chara k$. Let $f \in K(Y)$, and let $\nu \colon \mc{X} \to \mc{Y}$ be the normalization of
$\mc{Y}$ in the Kummer extension $K(Y)[z]/(z^d - f)$. The point of
this section is to prove Corollary~\ref{Csmoothoncomponent}, which
shows that, if we choose $\mc{Y}$ carefully, then if one takes the
normalization of $\mc{Y}$ in an $\ints/d$-cover, the points lying
above the standard crossings and finite cusps of $\mc{Y}$ (see Definition~\ref{Dstandardcrossing}) are smooth on the irreducible components of the special fiber where they appear. This will ultimately allow us to apply
Lemma~\ref{LregularUFD} to show that these points are regular. We also collect various preliminary results on generators of divisor class groups/value groups associated at points/components lying above finite cusps/standard crossings.

\begin{prop}\label{Pnormalizationsmooth} \hfill

\begin{enumerate}[\upshape (i)]
  \item Suppose $y \in \mc{Y}$ is a standard crossing, lying on two
    irreducible components of the special fiber with reduced divisors
    $D_1$ and $D_2$.  If the only  
 part of $\divi(f)$ passing through $y$ is a multiple of $D_2$, then
    $\nu^{-1}(D_1)$ is smooth above $y$ when given
    the reduced subscheme structure. 
   \item Suppose $y \in Y$ lies on a single irreducible component of
    the special fiber with reduced divisor $D$ and corresponding Mac
    Lane valuation $[v_0,\, v_1(\phi_1) = \lambda_1,\, \ldots,\,
    v_n(\phi_n) = \lambda_n]$.  Suppose further that $y = D_g
    \cap D$, where $g = \phi_n$ or $g$ is a proper key polynomial over
    $\phi_n$.  If
    the only part of $\divi(f)$ passing through $y$ (if any) is a
    multiple of $D_{g}$, then $\nu^{-1}(D)$
    is smooth above $y$ when given the reduced subscheme structure.
      \end{enumerate}  
\end{prop}

\begin{proof}
  Let $h$ be as in Proposition~\ref{Pstandardcrossingmultiplicity}(i).  Since $\divi(f)$ is
locally Cartier at $y$, Proposition~\ref{Pstandardcrossingmultiplicity} implies
that $\divi(f)$ is an integer multiple of $\divi(h)$ when restricted
to $A := \hat{\mc{O}}_{\mc{Y}, y}$, say $\divi(f) = b\divi(h)$.  By
Lemma~\ref{Lallnthroots}, $A[z]/(z^d - f) \cong A[z]/(z^d - h^b)$, so we
may assume $f = h^b$.  Furthermore, the normalization of $A[z]/(z^d - h^b)$ decomposes as a
direct product of rings isomorphic to $A[z]/(z^{d'} - h)$ for some $d'
\mid d$.  Since direct products of rings correspond to disjoint unions
of spectra, we may replace $d$ with $d'$ and assume that $f
= h$.

By the construction of $h$, we have that $D_1 \cap \Spec A = \Spec
k[[h]]$, and the point $y$ corresponds to $h = 0$.  So $\nu^{-1}(D_1) \cap \Spec A[z]/(z^d - h) = \Spec
k[[z]]$.  This is a regular local ring, showing that $\nu^{-1}(D_1)$ is
smooth above $y$.

The proof of part (ii) is the same, using
Proposition~\ref{Pstandardcrossingmultiplicity}(ii) (resp.\ (iii)) in place of
Proposition~\ref{Pstandardcrossingmultiplicity}(i) when $g = \phi_n$ (resp.\ $g$ is a
proper key polynomial over $\phi_n$).
\end{proof}

The following corollary is the main result of this subsection.

\begin{corollary}\label{Csmoothoncomponent}
  Let $\mc{Y}$ be a normal model of $Y := \proj^1_K$.
  %such that all points lying on more than one component of the special fiber
%are standard crossings and all non-regular points lying on exactly one
%component of the special fiber are standard endpoints.
  Let $f \in
K(Y)$, and let $\nu \colon \mc{X}
\to \mc{Y}$ be the normalization of $\mc{Y}$ in the
Kummer extension $K(Y)[z]/(z^d - f)$.
Let $x \in \mc{X}$ be a closed point such that either
\begin{enumerate}[\upshape (a)]
 \item $\nu(x)$ is a standard
crossing and no horizontal part of $\divi_0(f)$ passes through
$\nu(x)$, or,
\item $\nu(x)$ lies on a single irreducible component of the
special fiber of $\mc{Y}$, with reduced divisor $D$ and corresponding
Mac Lane valuation $v = [v_0,\, v_1(\phi_1) = \lambda_1,\, \ldots,\,
v_n(\phi_n) = \lambda_n]$ and that the only horizontal part of $\divi_0(f)$ passing through $\nu(x)$ (if any) is
$D_g$, where either $g = \phi_n$ or $g$ is a proper key polynomial
over $\phi_n$.  
\end{enumerate}

\begin{comment}
Let $x \in \mc{X}$ be a closed point such that $\nu(x)$ is a standard
crossing or $\nu(x)$ lies on a single irreducible component of the
special fiber of $\mc{Y}$, with reduced divisor $D$ and corresponding
Mac Lane valuation $v = [v_0,\, v_1(\phi_1) = \lambda_1,\, \ldots,\,
v_n(\phi_n) = \lambda_n]$.  Assume that if $\nu(x)$ is a standard
crossing, then no horizontal part of $\divi_0(f)$ passes through
$\nu(x)$, and otherwise, that the only horizontal part of $\divi_0(f)$ passing through $\nu(x)$ (if any) is
$D_g$, where either $g = \phi_n$ or $g$ is a proper key polynomial
over $\phi_n$.  
\end{comment}
If $\widetilde{D}$ is the reduced induced
subscheme of an irreducible
component of the special fiber of $\mc{X}$ containing $x$, then
$x$ is smooth on $\widetilde{D}$, and furthermore $\widetilde{D}$ is the only irreducible
  component of $\nu^{-1}(\nu(\widetilde{D}))$ containing $x$.
%  or $x$ is regular on $\mc{X}$. 
\end{corollary}

\begin{proof}
  First, suppose that $y := \nu(x)$ is a standard crossing of
  $\mc{Y}$.  Let $D_1$ and $D_2$ be the two reduced vertical divisors
  passing through $y$, and assume without loss of generality that $\widetilde{D}$
  lies above $D_1$.  By assumption, we have that $\divi(f) = aD_1 +
  bD_2$ when restricted to $\Spec \hat{\mc{O}}_{\mc{Y}, y}$, for some
  integers $a$ and $b$.  Since the ramification index of $D_1$ in
  $\nu$ is $d/\gcd(a, d)$, Lemma~\ref{Ltotallyramified} shows that $\widetilde{D}$ is isomorphic to the
  reduced induced subscheme of a component above $D_1$ when $d$ is
  replaced by $\gcd(a, d)$. So we may assume that $d = \gcd(a, d)$;
  that is, $d \mid a$.  By Lemma~\ref{Lorder1}(i), there exists $h \in K(Y)$ whose
  divisor when restricted to $\Spec \hat{\mc{O}}_{\mc{Y}, y}$ is $D_1
  + cD_2$ for some integer $c$.  Replacing $f$ with $f/h^{a}$, which
  doesn't change the cover because $h^a$ is an $d$th power, we may
  assume that $a = 0$.  Now Proposition~\ref{Pnormalizationsmooth}(i)
  applies to prove the corollary.

  Next, suppose that $y$ lies on a single irreducible component as in
  the corollary. By
  assumption we have $\divi(f) = aD + bD_g$ when restricted to $\Spec
  \hat{\mc{O}}_{\mc{Y}, y}$, for some integer $a$ and $g$ as in the corollary. As in the previous
  case, we may assume $d \mid a$. By Lemma~\ref{Lorder1}(ii) applied
  to $v_n$ (or to $v_{n-1}$ if $v_n = v_{n-1}$), there exists $h \in K(Y)$
  whose divisor when restricted to $\Spec \hat{\mc{O}}_{\mc{Y}, y}$ is
  $D + cD_g$ for some integer $c$.  Again as in the previous case, we
  replace $f$ with $f/h^a$ and assume that $a = 0$.  Now
  Proposition~\ref{Pnormalizationsmooth}(ii) applied to $v_n$ proves the corollary.
%
%  The only case that remains is that $y$ is a regular point of
%  $\mc{Y}$ lying on exactly one irreducible component of the
%  special fiber.  If $\nu$ is \'{e}tale above $y$, then $x$ is
%  automatically regular on $\mc{X}$, and if not, the the branch divisor $D$ of $\nu$
%  passing through $y$ is just the (reduced induced subscheme of) the irreducible component itself, which
%  is isomorphic to $\proj^1_k$ (\cite[Lemma 7.1]{ObusWewers}).
%  Since $y$ is a smooth point of $D$, it follows from
%  \cite[Theorem 1.7]{KWsuper} that $x$ is regular on $\mc{X}$.  
\end{proof}

\subsection{Generators for divisor class groups and their value groups}

\begin{corollary}\label{C:Multupstairs}
 In the sitatution of Corollary~\ref{Csmoothoncomponent},
 $\widetilde{D}$ generates the group generated by $\nu^*D$ and the vertical part of $\divi(z)$ in $\Div (\Spec \hat{\mc{O}}_{\mc{X},x})$. 
\end{corollary}

\begin{proof}
Let $\eta_{\widetilde{D}}$ (resp.\ $\eta_D$) be the generic point of
$\widetilde{D}$ (resp.\ $D$).  Then, since $\hat{\mc{O}}_{\mc{X},
  \eta_{\widetilde{D}}} / \hat{\mc{O}}_{\mc{Y}, \eta_D} $ is a tame
Kummer extension of discrete valuation rings given by $z^d = f$, the maximal ideal of
$\hat{\mc{O}}_{\mc{X}, \eta_{\widetilde{D}}}$ is generated by $z$ and
the maximal ideal of $\hat{\mc{O}}_{\mc{Y}, \eta_D}$.  In the language
of divisors, this is the corollary.
\end{proof}

Now we compute generators for the value groups of the discrete valuations on $K(X)$ extending the discrete valuations on $K(Y)$ corresponding to the two irreducible components of $\mc{Y}_k$ in a standard crossing. 
For a standard crossing $y \in \mc{Y}$
(Definition~\ref{Dstandardcrossing}) corresponding to two Mac Lane
valuations $v := [v_0,\, v_1(\phi_1) = \lambda_1,\, \ldots,\,
v_{n-1}(\phi_{n-1}) = \lambda_{n-1},\, v_n(\phi_n) = \lambda_n]$ and $v' :=
[v_0,\, v_1(\phi_1) = \lambda_1,\, \ldots,\,
v_{n-1}(\phi_{n-1}) = \lambda_{n-1},\, v_n(\phi_n) = \lambda_n']$, with
$\lambda_n < \lambda_n'$, let $N_y \colonequals e_{v_{n-1}}$ (so
$(1/N_y)\ints$ is the group generated by $1,
\lambda_1, \ldots, \lambda_{n-1}$), let $\psi_y$ be a monomial in
$\phi_1, \ldots, \phi_{n-1}$ over $K$ such that $v(\psi_y) =
v'(\psi_y) = 1/N$, and let $\phi_y \colonequals \phi_n$.  
\begin{lemma}\label{LuValueGrp}
 Let $D_1,D_2$ be reduced divisors on $\mc{X}$ meeting at a point $x$ as in Proposition~\ref{Pnormalizationsmooth}(i), lying above a standard crossing $y \in \mc{Y}$, and let $v,v'$ be the Mac Lane valuations corresponding to $y$ as in Definition~\ref{Dstandardcrossing}. Let $\psi_y,\phi_y \in K(Y)$ be as above. 
 \begin{enumerate}[\upshape (i)]
\item The divisors $D_1$ and $D_2$ are locally irreducible at
  $x$.
\item The value group of the extension of $v$ to
  $\hat{\mc{O}}_{\mc{X},x}$ is generated by $v(\phi_y)$, $v(\psi_y)$,
  and $v(z)$, and similarly for $v'$. 
\end{enumerate}
\end{lemma}
\begin{proof}
That $D_1$ and
$D_2$ are locally irreducible follows from
Corollary~\ref{Csmoothoncomponent}, proving (i).
The order functions on $D_1$
and $D_2$ give rise respectively to (the extensions of) the valuations $v$ and $v'$,
appropriately scaled.  The value group of $v$ on $K(\proj^1) = K(t)$ is
generated by $v(\psi_y)$ and $v(\phi_y)$, and thus, by rephrasing
Corollary~\ref{C:Multupstairs} in terms of valuation theory, the value group of
 the extension of $v$ to $K(X)$ is generated by $v(\psi_y)$, $v(\phi_y)$, and $v(z)$.  The
 analogous results hold for $v'$, proving (ii).
\end{proof}

%We finish this section with a lemma that will be applied throughout \S\ref{Sdetecting}.

\section{Some lattice theory}\label{SLattices}
In this section, we prove some results on lattices that will
be used in the next section to show that closed points in $\mc{X}$ lying above a standard crossing $y \in \mc{Y}$ are regular.  In Lemma~\ref{LuValueGrp} and Corollary~\ref{Csmoothoncomponent}, we showed that if $x \in \mc{X}$ maps to a standard crossing $y$, then $x$ is the intersection of two vertical prime divisors $D_1,D_2$ of $\mc{X}_k$, and $x$ is a smooth point on each of these components. Lemma~\ref{LregularUFD} and Lemma~\ref{Lbothprincipal} show that for $x$ to be regular on $\mc{X}$, it is necessary and sufficient that both $D_1$ and $D_2$ are principal at $x$ and that they intersect transversally. 

Let $v_x \colonequals (v,v') \colon K(X) \rightarrow \Q^2$ denote the ordered pair of discrete valuations corresponding to $D_1,D_2$, and let $L \subset \mathbb{Q}^2$ be a lattice generated by $v_x(g)$ for rational functions $g$ with divisors supported purely on $D_1,D_2$. Then, if $(1/x_0)\Z$ and $(1/y_0)\Z$ are the value groups for the discrete valuations corresponding to $D_1,D_2$ respectively, it suffices to show $(1/x_0,0)$ and $(0,1/y_0)$ generate the subgroup $L$ to establish local principality of $D_1,D_2$. With this in mind, we define the notion of a lattice $L \subset \mathbb{Q}^2$ being ``aligned with the coordinate axes'' in Definition~\ref{Daxes} when it has generators along the coordinate axes as above.
%, a condition that reflects the local prinicipality of $D_1,D_2$. 
In Lemma~\ref{LuValueGrp}, we computed three generators for the special lattices $L \subset \mathbb{Q}^2$ appearing in our setting (the $(v,v')$ valuations of the functions $\psi_y,\phi_y,z$ in Lemma~\ref{LuValueGrp}) -- these generators will be rewritten more explicitly in the next section (see \eqref{Evv'z} in Lemma~\ref{Ldivisormultiplicities}) and shown to have generators as in lattices considered in Corollary~\ref{CregularNpath2}. The main result of this section is Corollary~\ref{CregularNpath3}, a numerical criterion for the special lattices $L \subset \mathbb{Q}^2$ in Lemma~\ref{CregularNpath2} to be aligned with the coordinate axes, which will then be applied in Proposition~\ref{Pstandardcrossingregular} to establish principality of the divisors $D_1,D_2$ for well-chosen $\mc{Y}$. 

%\padma{OLD: \sout{In particular, Definition~\ref{Daxes} below of a lattice being ``aligned with the coordinate axes''  will correspond to certain vertical divisors in the model $\mc{X}$ of the cyclic cover being locally principal, allowing us to apply Lemma~\ref{LregularUFD}.}} 
%\padma{Figure out if there is a one line summary , using Lemma~\ref{Ldivgroupgenerators} and Lemma~\ref{Ldivisormultiplicities}, of the intuition behind the three numbers, and the lattices in $\mathbb{Q}^2$ considered here. Word dump to revisit: By previous section, every point lying above a standard crossing lies on exactly two components, the three integers are valuation of $\psi$ (an invariant function with equal valuation on the two components at a crossing), $\phi_n$, $z$ }

%\subsection{Shortest $N$-paths}

\subsection{Some special lattices in $\rats^2$ and their generators}
\begin{lemma}\label{Llatticebasis}
  Let $L \subseteq \rats^2$ be a lattice containing $(r, r)$ for some $r \in \rats_{> 0}$ minimal.  Let $(x, y)$ be an element of $L$ minimizing $y - x$ subject to $y > x$.  Then $L$ is generated by
  $(r, r)$ and $(x, y)$.
\end{lemma}

\begin{proof}
By the assumption on $y-x$, if $(a, b) \in L$, then $(b-a) = c(y - x)$ for some $c \in \ints$.  So $(a, b) - c(x, y) = (s, s)$ for some $s \in \rats$.  By minimality of $r$, we have $(s, s) = d(r, r)$ for some $d \in \ints$.
\end{proof}

\begin{corollary}\label{CregularNpath2}
  Let $N, d, e, s \in \nats$ and $\lambda, \lambda' \in \rats$, and
  let $L \subseteq \rats^2$ be the lattice generated by $$(1/N, 1/N),
  \ (\lambda, \lambda'), \text{and }  (\frac{e}{d}\lambda + \frac{s}{Nd},\,
  \frac{e}{d}\lambda' + \frac{s}{Nd}).$$
Then $L$ is generated by $(1/\widetilde{N},
1/\widetilde{N})$ and $(\widetilde{\lambda}, \widetilde{\lambda}')$, where
$$\widetilde{\lambda} = \frac{\gcd(d, e)}{d}\lambda + \frac{rs}{Nd},\ \
\widetilde{\lambda}' = \frac{\gcd(d, e)}{d}\lambda' + \frac{rs}{Nd},\ \
\widetilde{N} = N\frac{\gcd(d, e)}{\gcd(d, e,s)},$$ and $r$ is any
integer such that $re/\gcd(d, e) \equiv 1 \pmod{d/\gcd(d, e)}$.
 %Furthermore, $L$ is aligned with the coordinate axes if and only if $\widetilde{\lambda}' > \widetilde{\lambda}$ is a shortest $\widetilde{N}$-path. 
\end{corollary}

\begin{proof}
Let $(a, a') = ((e/d)\lambda + s/Nd,\, (e/d)\lambda' + s/Nd)$. By Lemma~\ref{Llatticebasis}, $L$ is generated
    by a generator $(1/\widetilde{N},
1/\widetilde{N})$ for the sublattice $L_{\Delta}$ of $L$ with
both coordinates equal, and an element $(a,b) \in L$ that achieves the minimum positive value of $b-a$. 
  Now $L_{\Delta}$ is generated by 
$$(1/N, 1/N), \quad \textup{and } \quad \frac{d}{\gcd(d,e)}(a, a') - \frac{e}{\gcd(d,e)}(\lambda, \lambda') = \left(\frac{s}{N\gcd(d,e)},
  \frac{s}{N\gcd(d,e)}\right),$$ in other words, by
$$\left(\frac{\gcd(d,e,s)}{N\gcd(d, e)}, \frac{\gcd(d,e,s)}{N\gcd(d, e)}\right) =
\left(\frac{1}{\widetilde{N}}, \frac{1}{\widetilde{N}} \right).$$

On the other hand, the minimal positive value of $b-a$ for $(a,b) \in
L$ is $(\gcd(d,e)/d)(\lambda' - \lambda)$.  An element of $L$
realizing this difference can be written by letting $c \in \ints$ be such that $re/\gcd(d,e) = 1 +
c(d/\gcd(d,e))$, and then taking $r(a, a') - c(\lambda,
\lambda')$, which equals
$$\left(\frac{\gcd(d,e)}{d}\lambda + \frac{rs}{Nd},
  \frac{\gcd(d,e)}{d}\lambda' + \frac{rs}{Nd}\right) =
(\widetilde{\lambda}, \widetilde{\lambda}').\qedhere$$
 %By Corollary~\ref{CregularNpath}, the lattice $L$ is aligned with the coordinate axes if and only if $\widetilde{\lambda}' > \widetilde{\lambda}$ is a shortest $\widetilde{N}$-path.
\end{proof}

\begin{defn}\label{Daxes}
  We say that a lattice $L \subseteq \rats^2$ is \emph{aligned with
    the coordinate axes} if there exist elements $(x_0, 0), (0, y_0) \in L$ which
  generate $L$.
\end{defn}

\subsection{Shortest $N$-paths and lattices aligned with the coordinate axes}
We recall the notion of \emph{shortest $N$-path}, introduced in \cite{ObusWewers}. 
\begin{defn}\label{DNpath}
Let $N$ be a natural number, and let $a > a' \geq 0$ be rational
numbers.  An \emph{$N$-path} from $a$ to $a'$  is a
decreasing sequence $a = b_0/c_0 > b_1/c_1 > \cdots > b_r/c_r = a'$ of rational numbers in lowest terms such that
$$\frac{b_i}{c_i} - \frac{b_{i+1}}{c_{i+1}} = \frac{N}{\lcm(N, c_i)\lcm(N, c_{i+1})}$$ for
$0 \leq i \leq r-1$.  If, in addition, no proper subsequence of $b_0/c_0 > \cdots > b_r/c_r$ containing
  $b_0/c_0$ and $b_r/c_r$ is an $N$-path, then the sequence is called
  the \emph{shortest $N$-path} from $a$ to $a'$.  
  \end{defn}

\begin{remark}
By \cite[Proposition A.14]{ObusWewers}, the shortest $N$-path from
$a'$ to $a$ exists and is unique.
\end{remark}

\begin{remark}\label{R1pathconsec}
  Observe that two successive entries $b_i/c_i > b_{i+1}/c_{i+1}$ of a shortest $1$-path satisfy $b_i/c_i - b_{i+1}/c_{i+1} = 1/(c_ic_{i+1})$.
\end{remark}

\begin{example}\label{Efarey}
The sequence $1 > 1/2 > 2/5 > 3/8 > 1/3 > 0$ is a concatenation
of the shortest $1$-path from $1$ to $3/8$ with the shortest $1$-path
from $3/8$ to $0$.  The entire sequence is a $1$-path from $1$ to $0$,
but the \emph{shortest} $1$-path from $1$
to $0$ is simply $1 > 0$.  
\end{example}

%\andrew{I think subsection \S\ref{Sassociated} below can be
%  eliminated.  It is a relic from what we originally cut and pasted
%  into the paper.}
%\subsection{The Mac Lane valuation associated to a polynomial}\label{Sassociated}
%Let $\alpha \in \mc{O}_{\ol{K}}$ such that $v_K(\alpha) > 0$ and the
%minimal polynomial $f(x) \in K[x]$ of $\alpha$ has degree at least
%$2$. In this section, we define a canonical Mac Lane valuation $v_f$ attached to $f$. 
%
%Write
%$$v_f = [v_0,\, v_1(\phi_1) = \lambda_1,\, \ldots,\, v_n(\phi_n) =
%  \lambda_n] $$
%  for the unique Mac Lane valuation on $K(x)$ over which $f$ is a proper
%key polynomial (Proposition \ref{Pbestlowerapprox}(iv)).  As usual, write $v_0, v_1,\ldots, v_n = v_f$ for the intermediate
%valuations.  For $1 \leq i \leq n$,
%write $\lambda_i = b_i/c_i$ in lowest terms.  Let $N_i = \lcm(c_1,
%\ldots, c_{i-1}) = \deg(\phi_i)$ (Corollary \ref{CvalueofN}).
%Furthermore, pick once and for all a root $\alpha$ of $f$.

\begin{lemma}\label{LregularNpath}
  Let $L \subseteq \rats^2$ be a lattice generated by $(r, r)$ and $(x, y)$ as in Lemma~\ref{Llatticebasis} above. 
  Then $L$ is aligned with
    the coordinate axes if and only if $y/r > x/r$ is a
  (necessarily shortest) $1$-path.
\end{lemma}

\begin{proof}
  By dividing all elements of $L$ by $r$, we may assume $r = 1$.
  Write $x = a/b$ and $y = c/d$ in lowest terms with positive
  denominators. Then $L$ is aligned with
    the coordinate axes if and only if it
  contains $(1/b, 0)$ and $(0, 1/d)$.  Note that $y > x$ is a $1$-path if and only if
  $bc - ad = 1$.
  
  The covolume of $L$ is $(bc -
  ad)/bd \geq 1/bd$.  Strict inequality holds if $y > x$ is not a $1$-path, which is incompatible with $L$ containing $(1/b, 0)$
  and $(0, 1/d)$.  On the other hand, if $y > x$ is a $1$-path, then
  $c(1, 1) - d(x,  y) = ((bc - ad)/b, 0) = (1/b, 0)$.  So $(1/b, 0)
  \in L$, and since there exists some element of $L$ of the form $(q,
  1/d)$ with $1/b \mid q$, we conclude that $(0, 1/d) \in L$.
\end{proof}

%\begin{corollary}\label{CregularNpath}
%In the situation of Lemma~\ref{LregularNpath}, if $r = 1/\widetilde{N}$ for an
%integer $\widetilde{N}$, then $L$ is aligned with
%    the coordinate axes if and only if $y > x$ is a (necessarily shortest) $\widetilde{N}$-path.
%\end{corollary}

%\begin{proof}
%This follows immediately from Lemma~\ref{LregularNpath} and \cite[Lemma~A.7]{ObusWewers}.
%\end{proof}

\begin{corollary}\label{CregularNpath3}
The lattice $L$ in Corollary~\ref{CregularNpath2} is aligned with the coordinate axes if and only if
$\widetilde{\lambda}' > \widetilde{\lambda}$ is a shortest
$\widetilde{N}$-path.
\end{corollary}
\begin{proof}
By Corollary~\ref{CregularNpath2}, the lattice $L$ is generated by $(1/\widetilde{N}, 1/\widetilde{N})$ and $(\widetilde{\lambda}, \widetilde{\lambda}')$. The corollary now follows from from Lemma~\ref{LregularNpath} and \cite[Lemma~A.7]{ObusWewers}.
%By Corollary~\ref{CregularNpath}, the lattice $L$ is aligned with the coordinate axes if and only if $\widetilde{\lambda}' > \widetilde{\lambda}$ is a shortest $\widetilde{N}$-path.
\end{proof}

\begin{comment}
\padma{OLD: Remove, or suitably incorporate in intro paragraph? \sout{Corollary~\ref{CregularNpath3} will be used in the following situation:
Suppose we have a local arithmetic surface $\mc{X}$ whose special
fiber is supported on two smooth prime divisors $D_1$ and $D_2$, along
with $s$ elements $g_i \in K(\mc{X})$ ($1 \leq i \leq s$) such that
$\divi(g_i) = a_iD_1 + b_i D_2$.  By Lemma~\ref{LregularUFD}, to show
that $\mc{X}$ is regular it suffices to show that $D_1$ is principal.
That is, it suffices to show that $(1, 0)$ is contained in the lattice
generated by $(a_i, b_i)$, $1 \leq i \leq s$, which can be
accomplished via Corollary~\ref{CregularNpath3}.} }
\end{comment}

\section{A numerical criterion for regularity on models of superelliptic curves}\label{Sdetecting}
%\section{Detecting regularity with normal crossings locally on models of superelliptic curves}\label{Sdetecting}
As before, $\mc{Y}$ is a normal model of $Y = \proj^1_K$, and $\nu \colon \mc{X} \to \mc{Y}$ is
the normalization of $\mc{Y}$ in the Kummer extension $K(Y)[z]/(z^d -
f)$ with $f \in K(Y)$ and $\chara k \nmid d$.  We further assume in
this section that $d \mid \deg(f)$ and that all roots of $f$ are
integral over $\mc{O}_K$ (as will be explained in \S\ref{Sreductions}, these
new restrictions do not entail a fundamental loss of generality).  By Lemma~\ref{Lallnthroots}, we may replace $f$ by
its product with a $d$th power and thus assume that
$f$ has irreducible factorization $\pi_K^af_1^{a_i} \cdots
f_q^{a_q}$ where all the $f_i$ are monic. In this section, we lay the
groundwork for understanding when $\mc{X}$ is
regular.

In earlier work, \cite[Corollaries 7.5, 7.6]{ObusWewers} give a criterion for testing regularity at certain closed points in a normal model $\mc{Y}$ of $\P^1_K$ in terms of $N$-paths of rational numbers (see Definition~\ref{DNpath}) arising from the Mac Lane descriptions of the components in $\mc{Y}_k$. In this section, we show how to lift this numerical $N$-path criterion to a certain $\widetilde{N}$-path criterion for testing regularity at certain closed points in the normalization of $\mc{Y}$ in a cyclic cover of $K(Y)$. The new invariant $\widetilde{N}$ additionally incorporates numerical information from the degree of the cover and the polynomial $f$. More precisely, in \S\ref{Sstandardcrossings},
\S\ref{Sstandardendpoints}, \S\ref{Sinfty}, and \S\ref{Sinftycrossing}
below, we will give regularity criteria for $\mc{X}$ above four types
of closed points of $\mc{Y}$: The \emph{standard crossings}
(\S\ref{Sstandardcrossings}) where the main result is
Proposition~\ref{Pstandardcrossingregular}, the \emph{finite cusps} (\S\ref{Sstandardendpoints}), where the main result is Proposition~\ref{Pstandardendpointregular},
the \emph{standard $\infty$-specialization} (\S\ref{Sinfty}), where the
main result is Proposition~\ref{PinftySNC}, and
the \emph{$\infty$-crossing} (\S\ref{Sinftycrossing}), where the main
result is Proposition~\ref{Pinftycrossingregular}.  The results in
\S\ref{Sinfty} and \S\ref{Sinftycrossing} are only used in
\S\ref{Sminimal}, when the components above the $v_0$-component are contractible in the strict normal crossings regular model that we construct in \S\ref{Sfirstmodel}. The reader content with a regular normal crossings model that is not necessarily minimal can safely skip these sections.

%(See Proposition~\ref{Pstandardcrossingregular}, Proposition~\ref{Pstandardendpointregular}, Proposition~\ref{PinftySNC}, Proposition~\ref{Pinftycrossingregular}) 

%This process isparallel to the process of resolving the singularities of $\mc{Y}'$ itself, as described in \cite[Corollaries 7.5, 7.6]{ObusWewers}, but the formulas are more complicated when working with a cyclic cover.
%In \S\ref{Sstandardcrossings},\S\ref{Sstandardendpoints}, \S\ref{Sinfty}, and \S\ref{Sinftycrossing} below, we will give regularity criteria for $\mc{X}$ above four types of closed points of $\mc{Y}$: The \emph{standard crossings} (\S\ref{Sstandardcrossings}) where the main result is Proposition~\ref{Pstandardcrossingregular}, the \emph{finite cusps} (\S\ref{Sstandardendpoints}), where the main result is Proposition~\ref{Pstandardendpointregular}, the \emph{standard $\infty$-specialization} (\S\ref{Sinfty}), where the main result is Proposition~\ref{PinftySNC}, and the \emph{$\infty$-crossing} (\S\ref{Sinftycrossing}), where the main result is Proposition~\ref{Pinftycrossingregular}.  Before doing this, we include one lemma that will be applied throughout \S\ref{Sregularity}.

%\subsection{Regularity criteria at various closed points}\label{Sregularity}
\begin{lemma}\label{Ldivgroupgenerators}
  Let $y \in \mc{Y}$ be a closed point, let $x \in \mc{X}$ lie above
  $y$, and let $\Sigma =
  \Aut(\hat{\mc{O}}_{\mc{X},x}/\hat{\mc{O}}_{\mc{Y},y})$.  The group
  of $\Sigma$-invariant principal divisors on
  $\Spec(\hat{\mc{O}}_{\mc{X},x})$ is generated by $\divi(z)$ and
  $\divi(\nu^*\beta)$, as $\beta$ ranges through $\hat{\mc{O}}_{\mc{Y},y}$. 
\end{lemma}

\begin{proof}
Suppose $w \in \hat{\mc{O}}_{\mc{X},x}$ gives a $\Sigma$-invariant
principal divisor, so that
$\sigma^*(\divi(w)) = \divi(w)$ for all $\sigma \in \Sigma$.  This means that if $w'
\in \hat{\mc{O}}_{\mc{Y}, y}$
is the norm of $w$, then $\divi(w') = \divi(w^{|\Sigma|})$, so
there is a unit $u \in \hat{\mc{O}}_{\mc{X}, x}$ such that
$w^{|\Sigma|}u = w'$, thinking of $\hat{\mc{O}}_{\mc{Y}, y}$ as a
subring of $\hat{\mc{O}}_{\mc{X},x}$.  By Lemma~\ref{Lallnthroots}, we can write $u =
c^{|\Sigma|}$ for some $c \in \mc{O}_{\mc{X},x}^{\times}$, so
replacing $w$ with $wc$, we may assume that $w^{|\Sigma|} \in
\hat{\mc{O}}_{\mc{Y}, y}$.  By Kummer theory, we conclude that $w$ is
a power of $z$ times an element of $\hat{\mc{O}}_{\mc{Y},y}$, proving the lemma.
\end{proof}

\subsection{Standard crossings}\label{Sstandardcrossings} 

%\cite[Corollaries 7.5, 7.6]{ObusWewers} shows that regularity at the intersection of two vertical components $v,v'$ in a normal model $\mc{Y}$ of $\P^1_K$ is equivalent to certain rational numbers appearing in the Mac Lane descriptions of $v,v'$ being neighbors in an ``$N$-path'' of rational numbers. In this section, we show how to lift this numerical $N$-path criterion to a certain ``$\widetilde{N}$-path criterion'' for testing regularity at certain closed points in the normalization of $\mc{Y}$ in a cyclic cover of $K(Y)$.

Let $y \in \mc{Y}$ be a standard crossing
(Definition~\ref{Dstandardcrossing}) corresponding to two Mac Lane
valuations $v := [v_0,\, v_1(\phi_1) = \lambda_1,\, \ldots,\,
v_{n-1}(\phi_{n-1}) = \lambda_{n-1},\, v_n(\phi_n) = \lambda_n]$ and $v' :=
[v_0,\, v_1(\phi_1) = \lambda_1,\, \ldots,\,
v_{n-1}(\phi_{n-1}) = \lambda_{n-1},\, v_n(\phi_n) = \lambda_n']$, with
$\lambda_n < \lambda_n'$.  Write $N = e_{v_{n-1}}$ (so
$(1/N)\ints$ is the group generated by $1,
\lambda_1, \ldots, \lambda_{n-1}$), and write $\psi$ for a monomial in
$\phi_1, \ldots, \phi_{n-1}$ over $K$ such that $v(\psi) =
v'(\psi) = 1/N$.  

\begin{lemma}\label{Lwriteasprincipal}
Suppose $g \in \mc{O}_K[t]$ is
monic and irreducible with a root $\theta$. If $v_K(\phi_n(\theta))
\geq \lambda_n'$, then $\divi(g) = e \divi(\phi_n)$ on
$\Spec \hat{\mc{O}}_{\mc{Y},y}$, where $e = \deg(g)/\deg(\phi_n)$.  If $v_K(\phi_n(\theta)) \leq \lambda_n$, then
$\divi(g)$ is a multiple of $\divi(\psi)$ on $\Spec \hat{\mc{O}}_{\mc{Y},y}$.
\end{lemma}

\begin{proof}
Observe that in both cases, Proposition~\ref{Pparameterize}\ref{Pannulus} shows
there is no horizontal part of $\divi(g)$ passing through $y$.
  In the first case, letting $\ell = \deg(g)/\deg(\phi_n)$,
Lemma~\ref{Ldominantterm}(i) shows
that $v(g) = e v(\phi_n)$ and $v'(g) = e v'(\phi_n)$, which
implies that $\divi(g) = e \divi(\phi_n)$.  In the second case,
Lemma~\ref{Ldominantterm}(ii) shows that if $g = \sum_{i} a_i\phi_n^i$
is the $\phi_n$-adic expansion of $g$, then $v(g) = v(a_0)$ and $v'(g)
= v'(a_0)$.  Since $\deg(a_0) < \deg(\phi_n)$, we have $v(a_0) =
v'(a_0) \in (1/N)\ints$, so $\divi(g)$ is a multiple of $\divi(\psi)$.  We are done.
\end{proof}

\begin{lemma}\label{Lgeneratingdivisorsdownstairs}
Assume that no horizontal part of $\divi(f)$ passes through $y$.  The
group of vertical principal divisors of $\Spec \hat{\mc{O}}_{\mc{X},x}$ is
generated by $\divi(z)$, $\divi(\nu^*\phi_n)$, and $\divi(\nu^*\psi)$.
\end{lemma}

\begin{proof}
Let $w \in \hat{\mc{O}}_{\mc{X},x}$ such that $\divi(w)$ is vertical.
By Corollary~\ref{Csmoothoncomponent}, there is only one prime
vertical divisor of $\Spec \hat{\mc{O}}_{\mc{X}, x}$ above each prime
vertical divisor of $\Spec \hat{\mc{O}}_{\mc{Y},y}$, so $\divi(w)$ is
$\Sigma$-invariant, for $\Sigma$ as in
Lemma~\ref{Ldivgroupgenerators}.  Applying
Lemma~\ref{Ldivgroupgenerators}, and noting that $\divi(z)$ is a
vertical divisor, it
remains to show that the group of vertical principal divisors of $\Spec
\hat{\mc{O}}_{\mc{Y}, y}$ is generated by $\divi(\phi_n)$ and $\divi(\psi)$.

It suffices to consider a monic irreducible polynomial $g$ such that $\divi(g)$ is
a vertical principal divisor in $\hat{\mc{O}}_{\mc{Y},y}$, and show
that $\divi(g)$ is an integer combination of $\divi(\phi_n)$ and $\divi(\psi)$.  Since
$\divi(g)$ has no horizontal component containing $y$,
Proposition~\ref{Pparameterize}\ref{Pannulus} shows that for any root $\theta$ of $g$,
either $v_K(\phi_n(\theta)) \geq \lambda_n'$ or $v_K(\phi_n(\theta)) \leq
\lambda_n$.  The result now follows from Lemma~\ref{Lwriteasprincipal}.
\end{proof}

Assume no horizontal part of $f$ passes through $y$.  Write $f = gh$, where $g$ is the product of the $f_i^{a_i}$ such that
$v_{f_i}^{\infty} \succ v'$ or equivalently, by Lemma~\ref{Lgfollows},
those $f_i$ with roots $\alpha_i$ such that
$v_K(\phi_n(\alpha_i)) \geq \lambda_n'$.  Let $e = \deg(g)/\deg(\phi_n)$, which
is an integer by Lemma~\ref{Ldominantterm}(i).  By
Proposition~\ref{Pparameterize}(ii), all $f_i$ dividing $h$ have roots
$\alpha_i$ with $v_K(\phi_n(\alpha_i)) \leq \lambda_n$, so let $s$ be the integer
guaranteed by Lemma~\ref{Lwriteasprincipal}
such that $\divi(h) = s\divi(\psi)$ on $\Spec
\hat{\mc{O}}_{\mc{Y},y}$.  Thus $v(h) = s/N$, and we let
$$\widetilde{N} = N\frac{\gcd(d,e)}{\gcd(d, e, s)}.$$  Lastly, note that the residue of $e/\gcd(d,e)$ modulo $d/\gcd(d,e)$ is a
unit, so let $r$ be any integer such that $re/\gcd(d,e) \equiv 1 \pmod{d/\gcd(d,e)}$.

\begin{lemma}\label{Ldivisormultiplicities}
Suppose that no horizontal part of $\divi(f)$ passes through $y$, and
$s,r,\widetilde{N}$ are as above.  Let $D$ and $D'$ be the prime vertical divisors of $\mc{Y}$ corresponding to $v$ and
$v'$ respectively.
Let $\widetilde{D}$ and $\widetilde{D}'$ be the prime divisors
corresponding to the parts of $\nu^{-1}(D)$
and $\nu^{-1}(D')$ respectively passing through $x$. Let
$$\widetilde{\lambda}_n =
\frac{\gcd(d,e)}{d}\lambda_n + \frac{rs}{Nd}, \qquad
\widetilde{\lambda}'_n = \frac{\gcd(d,e)}{d}\lambda_n' +
\frac{rs}{Nd}.$$  Furthermore, let $\widetilde{e}_v$ be such that
$\widetilde{\lambda}_n$ and $1/\widetilde{N}$ generate
$(1/\widetilde{e}_v)\ints$, and similarly define $\widetilde{e}_{v'}$
using $\widetilde{\lambda}'_n$ and $\widetilde{N}$. 
Then
\begin{enumerate}[\upshape (i)]
%\item The divisors $\widetilde{D}$ and $\widetilde{D}'$ are locally irreducible at $x$.
%\item The value group of the extension of $v$ to   $\hat{\mc{O}}_{\mc{X},x}$ is generated by $v(\phi_n)$, $v(\psi)$,   and $v(z)$, and similarly for $v'$. 
\item The multiplicity
of $\widetilde{D}$ (resp.\ $\widetilde{D}'$) in
$\hat{\mc{O}}_{\mc{X},x}$ is $\widetilde{e}_v$ (resp.\
$\widetilde{e}_{v'}$).  
\item We have $$(\widetilde{D}, \widetilde{D}') =
  \frac{\widetilde{N}}{\widetilde{e}_v\widetilde{e}_{v'}(\widetilde{\lambda}_n
    - \widetilde{\lambda}_n')}.$$
\end{enumerate}
\end{lemma}

\begin{proof}

%That $\widetilde{D}$ and $\widetilde{D}'$ are locally irreducible follows from Corollary~\ref{Csmoothoncomponent}, proving (i). The order functions on $\widetilde{D}$ and $\widetilde{D}'$ give rise respectively to (the extensions of) the valuations $v$ and $v'$, appropriately scaled.  The value group of $v$ on $K(\proj^1) = K(t)$ is generated by $v(\psi)$ and $v(\phi_n)$, and thus, by rephrasing Corollary~\ref{C:Multupstairs} in terms of valuation theory, 
By Lemma~\ref{LuValueGrp} the value group of 
 the extension of $v$ to $K(X)$ is generated by $v(\psi)$, $v(\phi_n)$, and $v(z)$.  The
 analogous results hold for $v'$.  Now, $v(\psi) = v'(\psi) = 1/N$, $(v(\phi_n), v'(\phi_n)) =
 (\lambda_n, \lambda'_n)$, and (extending $v$ and $v'$ to $\hat{\mc{O}}_{\mc{X},x}$ so that they are centered at the generic points of
$\widetilde{D}$ and $\widetilde{D}'$ respectively),
 \begin{equation}\label{Evv'z}(v(z), v'(z)) = (\frac{v(f)}{d}, \frac{v'(f)}{d}) =
 \frac{1}{d}(v(\phi_n^e\psi^s), v'(\phi_n^e\psi^s)) =
 (\frac{e}{d}\lambda_n + \frac{s}{Nd}, \frac{e}{d}\lambda_n' +
 \frac{s}{Nd}).\end{equation}  
 By Corollary~\ref{CregularNpath2}, $(\widetilde{\lambda}_n, \widetilde{\lambda}_n')$ and
$(1/\widetilde{N}, 1/\widetilde{N})$ generate the lattice generated by
$(v(\psi), v'(\psi))$, $(v(\phi_n), v'(\phi_n))$, and $(v(z),
v'(z))$.  By
Lemma~\ref{LuValueGrp}(ii), this means that the value groups of the extensions of $v$ and $v'$ on
$\hat{\mc{O}}_{\mc{X},x}$ are generated by $1/\widetilde{e}_v$ and
$1/\widetilde{e}_{v'}$, respectively.  In other words,
$\widetilde{e}_v$ (resp.\ $\widetilde{e}_{v'}$) is the multiplicity of
$\widetilde{D}$ (resp.\ $\widetilde{D}'$) on the special fiber of $\Spec
\hat{\mc{O}}_{\mc{X},x}$, proving (i).

The assumption that no horizontal part of $\divi(f)$ passes through $y$ guarantees that the divisor of $z$ is purely vertical. Since $\psi$ is a monomial in $1,\phi_1,\ldots,\phi_{n-1}$, the divisors of of $\psi$ and $\phi_n$ are also purely vertical by Lemma~\ref{Lnonspecialize}(i).

We turn to part (ii), beginning by calculating $[\hat{\mc{O}}_{\mc{X},x} : \hat{\mc{O}}_{\mc{Y},y}]$. On $\Spec \hat{\mc{O}}_{\mc{Y},y}$, we have $\divi(g) =
\divi(\phi_n^e)$ by Lemma~\ref{Lwriteasprincipal} and $\divi(h) =
s\divi(\psi)$ by the definition of $s$.  So $\divi(f) =
\divi(\phi_n^e\psi^s)$.
We observe for later that, since all units in $\hat{\mc{O}}_{\mc{Y},y}$
are $d$th-powers, $f$ is a $\gcd(d, e, s)$-th power
in $\hat{\mc{O}}_{\mc{Y},y}$.  Furthermore, if $a$ is maximal such
that $f$ is an $a$th power
in $\hat{\mc{O}}_{\mc{Y}, y}$, then $a \mid e$ since 
the horizontal part of $\divi(\phi_n)$ is irreducible, which means
that $\psi^s$ is an $a$th power, which means that $a \mid s$ since $\divi(\psi)$ is vertical and
indivisible as a divisor by the definition of $\psi$.  So $a \mid
\gcd(e, s)$ and thus $\gcd(d, a) \mid \gcd(d, e, s)$, which means that the fiber of $y$ in $\mc{X}$
consists of $\gcd(d, e, s)$ points, and thus
\begin{equation}\label{Edegree}
  [\hat{\mc{O}}_{\mc{X},x} : \hat{\mc{O}}_{\mc{Y},y}] = d/\gcd(d, e, s). 
\end{equation}

Recall that $e_v$ and $e_{v'}$ are the
multiplicities of $D$ and $D'$ in the special fiber of $\Spec
\hat{\mc{O}}_{\mc{Y},y}$.  
By Lemma~\ref{Lintersectionnumber}, we have $(D, D') =
N/((\lambda_n - \lambda_n')e_ve_{v'})$.  The ramification indices of
$\widetilde{D}/D$ and $\widetilde{D} '/D'$ are $\widetilde{e}_v/e_v$ and
$\widetilde{e}_{v'}/e_{v'}$, respectively.  By
Lemma~\ref{L:LiuLorAdapt} (noting that $k(w) = k(z) = k$ in the
language of the lemma), we have
\begin{equation}\label{EE}
\underbrace{\frac{d}{\gcd(d,e,s)}}_{[\hat{\mc{O}}_{\mc{X},x} :
        \hat{\mc{O}}_{\mc{Y},y}], \text{ see }\eqref{Edegree}} \underbrace{\frac{N}{(\lambda'_n - \lambda_n)e_ve_{v'}}}_{(D, D')} 
    = \frac {\widetilde{e}_v\widetilde{e}_{v'}}{e_ve_{v'}} (\widetilde{D}, \widetilde{D}').
\end{equation}
Now, $\lambda'_n - \lambda_n= (d/\gcd(d,e))(\widetilde{\lambda}_n' -
\widetilde{\lambda}_n)$. Plugging this into \eqref{EE} yields part (iii).
\end{proof}

\begin{remark}
Note the similarity between Lemma~\ref{Lintersectionnumber} and Lemma~\ref{Ldivisormultiplicities}(iv).
\end{remark}

\begin{prop}\label{Pstandardcrossingregular}
Suppose that no horizontal part of $\divi(f)$ passes through $y$, and $s,r,\widetilde{N}$ are as above.
If $x \in \mc{X}$ is a point above $y \in \mc{Y}$, then $x$ is
regular if and only if
$\widetilde{\lambda}_n' > \widetilde{\lambda}_n$ from
Lemma~\ref{Ldivisormultiplicities} above
is an $\widetilde{N}$-path.  Furthermore, in this case, the special fiber of
    $\mc{X}$ has normal crossings at $x$.
\end{prop}

\begin{proof}
%By Corollary~\ref{Csmoothoncomponent}, the valuations $v$ and $v'$ on
%$\hat{\mc{O}}_{\mc{Y}, y}$ extend uniquely to
%$\hat{\mc{O}}_{\mc{X},x}$, since $\Spec \hat{\mc{O}}_{\mc{X},x}$
%contains unique divisors $\widetilde{D}$ and $\widetilde{D}'$ above $D$ and $D'$
%respectively.
Let $\widetilde{D}$ and $\widetilde{D}'$ be reduced divisors on $\Spec
\hat{\mc{O}}_{\mc{X},x}$ as in
Lemma~\ref{Ldivisormultiplicities}.  By
Lemma~\ref{Ldivisormultiplicities}(i), they are
irreducible.
%reduced divisors
%corresponding to the parts of $\nu^{-1}(D)$
%and $\nu^{-1}(D')$ respectively passing through $x$.  By
%Corollary~\ref{Csmoothoncomponent}, both $\widetilde{D}$ and
%$\widetilde{D}'$ are irreducible.
Since $x$ being regular implies that $\widetilde{D}$ and
$\widetilde{D}'$ are principal in $\Spec \hat{\mc{O}}_{\mc{X},x}$, it suffices by
Lemma~\ref{Lbothprincipal} to show that $\widetilde{D}$ and
$\widetilde{D}'$ are principal on $\Spec \hat{\mc{O}}_{\mc{X},x}$ if and only if the $\widetilde{N}$-path criterion in
the proposition holds, and that in this case $(\widetilde{D}, \widetilde{D}') =
1$.

Consider the lattice $L$ in $\rats^2$ generated by
$(v(\psi), v'(\psi)) = (1/N, 1/N)$, $(v(\phi_n), v'(\phi_n)) =
(\lambda_n, \lambda_n')$, and
$$(v(z), v'(z)) = \frac{1}{d}(v(f), v'(f)) =
\frac{1}{d}(v(\phi_n^e\psi^s), v'(\phi_n^e\psi^s)) = ((e/d)\lambda_n + s/Nd,
(e/d)\lambda_n' + s/Nd),$$ where $v$ and $v'$ are extended to
$K(X)$ so that they are centered at the generic points of
$\widetilde{D}$ and $\widetilde{D}'$ respectively.
%\footnote{The extension might not be unique if
%$\widetilde{D}$ and $\widetilde{D}'$ are reducible, but the valuation
%of $z$ is well-determined regardless.  In fact, $\widetilde{D}$ and
%$\widetilde{D}'$ are \emph{a posteriori} irreducible, but we won't use this
%in the proof.}
By Corollary~\ref{CregularNpath3}, the $\widetilde{N}$-path criterion in the
    proposition holds if and only if the lattice $L$ is aligned with
    the coordinate axes.

We claim that $L$ is aligned with the coordinate axes if and only if $\widetilde{D}$ and
$\widetilde{D}'$ are principal on $\Spec \hat{\mc{O}}_{\mc{X},x}$.
To prove the claim, note that by
Lemma~\ref{LuValueGrp}(ii), the projection of $L$ to its
first (resp.\ second) coordinate is the value group of $v$ (resp.\
$v'$) on
$\hat{\mc{O}}_{\mc{X},x}$.  So $L$ being aligned with
    the coordinate axes implies that $\widetilde{D}$ and
    $\widetilde{D}'$ are locally principal at $x$.  On the other hand, if
    $\widetilde{D}$ and $\widetilde{D}'$ are locally principal at $x$, then
Lemma~\ref{Lgeneratingdivisorsdownstairs} shows that there are
monomials in $\phi_n$, $\psi$, and $z$ whose divisors cut out
$\widetilde{D}$ and $\widetilde{D}'$ locally, which means that $L$ is aligned
with the coordinate axes. 
%\padma{Cite new lemma about upstairs  multiplicity.} \andrew{I don't think this lemma is necessary.  $L$   containing $(a,0)$ and $(0, a')$ exactly means that there exist   elements of $\hat{\mc{O}}_{\mc{X},x}$ cutting out $\widetilde{D}$   and $\widetilde{D}'$ themselves, not multiples, exactly because $a$   and $a'$ generate the respective value groups.  In essence, the   lemma is just the statement that the value group is generated by   $v(\psi)$, $v(\phi_n)$, and $v(z)$.}

To complete the proof of the proposition, it remains to show that
$(\widetilde{D}, \widetilde{D}') = 1$ assuming the
$\widetilde{N}$-path criterion holds.  But $\widetilde{\lambda}_n$ and $\widetilde{\lambda}_n'$ being adjacent on
an $\widetilde{N}$-path means by definition that
$\widetilde{\lambda}_n' - \widetilde{\lambda}_n =
\widetilde{N}/\widetilde{e}_v\widetilde{e}_{v'}$.  By
Lemma~\ref{Ldivisormultiplicities}(ii), $(\widetilde{D}, \widetilde{D}') = 1$,
completing the proof.
\end{proof}

\begin{remark}
Observe that if $f$ is monic and $v_{f_i}^{\infty} \succ v'$ for all $i$,
then $h = 1$, $s = 0$ and the criterion reduces to $(\gcd(d,e)/d)\lambda_n' >
(\gcd(d,e)/d)\lambda_n$ being an $N$-path.
\end{remark}

\subsection{Finite cusps}\label{Sstandardendpoints}

Let $v = [v_0,\, v_1(\phi_1) = \lambda_1,\, \ldots,\,
v_{n-1}(\phi_{n-1}) = \lambda_{n-1},\, v_n(\phi_n) = \lambda_n]$ be a Mac
Lane valuation such that $v_{n-1}$ is minimally presented, but we
allow the possibility that $v_{n-1} = v_n$ and $\phi_n$ is a proper
key polynomial over $v_{n-1}$ (this occurs when $\lambda_n = v_{n-1}(\phi_n)$).  Let $y \in \mc{Y}$ be the intersection of $D_{\phi_n}$ with the
special fiber of $\mc{Y}$, and suppose that $y$ lies only on the
$v$-component of $\mc{Y}$.  By Lemma~\ref{Lstandardendpointunique}, $y$ is a finite cusp if $v$
is minimally presented and $e_v > e_{v_{n-1}}$, but the results of this section apply in a
slightly broader context that will be necessary for proving Theorem~\ref{TX0regular}.
Write $N = e_{v_{n-1}}$ (so $(1/N)\ints$ is the group generated by $1,
\lambda_1, \ldots, \lambda_{n-1}$), and write $\psi$ for a monomial in
$\phi_1, \ldots, \phi_{n-1}$ over $K$ such that $v(\psi) = 1/N$.  

\begin{lemma}\label{Lgeneratingdivisors2}
 Suppose that $f = \phi_n^a h$, where no horizontal part of $\divi(h)$
passes through $y$.
  Then, $\divi(h)$ is an integer multiple of $\divi(\psi)$, and the group of principal vertical divisors of $\Spec \hat{\mc{O}}_{\mc{X},x}$ is
 contained in the group generated by $\divi(z)$, $\divi(\nu^*\psi)$,
  and $\divi(\nu^*\phi_n)$. 
\end{lemma}

\begin{proof}
By Corollary~\ref{Csmoothoncomponent} applied to $v_n$ (or to
$v_{n-1}$ if $v_{n-1} = v_n$) with $g = \phi_n$, there is only one prime
vertical divisor of $\Spec \hat{\mc{O}}_{\mc{X}, x}$.  So that divisor is $\Sigma$-invariant, for $\Sigma$ as in
Lemma~\ref{Ldivgroupgenerators}.  By Lemma~\ref{Ldivgroupgenerators},
the group of principal vertical divisors of $\Spec
\hat{\mc{O}}_{\mc{X},x}$ is contained in the group generated by
$\divi(z)$ and $\divi(\nu^*\beta)$ for $\beta \in
\hat{\mc{O}}_{\mc{Y},y}$  Furthermore, since the horizontal part of
$\divi(z)$ is supported above $\divi(\phi_n)$, we have that the only
$\beta$ we need to consider are $\phi_n$ and those $\beta$ such that
$\divi(\beta)$ is vertical.  So it suffices to prove the first
assertion of the proposition.

We may assume $h$ is an irreducible polynomial.
Since $\divi(h)$ has no horizontal component containing $y$,
Proposition~\ref{Pparameterize}\ref{Psmallerdiskoid} applied to a root $\theta$ of
$h$ and a root
$\alpha$ of $\phi_n$ would show that $v_K(\phi_n(\theta)) \leq
\lambda_n$. Then Lemma~\ref{Ldominantterm}(ii) shows that if $h = \sum_{i} a_i\phi_n^i$
is the $\phi_n$-adic expansion of $\beta$, we have $v(h)
= v(a_0)$.  Since $\deg(a_0) < \deg(\phi_n)$, we have $v(a_0) 
 \in (1/N)\ints$, so $\divi(h)$ is a multiple of $\divi(\psi)$.  We are done.
\end{proof}

\begin{prop}\label{Pstandardendpointregular}
Suppose that $f = \phi_n^a h$, where no horizontal part of $\divi(h)$
passes through $y$.
\begin{enumerate}[\upshape (i)]
 \item We have $v(h) = s/N$ for some $s \in \ints$.
 \item Let $\widetilde{N} = N\gcd(d,a)/\gcd(d, a, s)$, with $s$ as in
   part (i).  If $x \in \mc{X}$ is a point above 
   $y \in \mc{Y}$, let $\widetilde{e}_v$ be the multiplicity of the special fiber of $\Spec
\hat{\mc{O}}_{\mc{X}, x}$. then $\mc{X}$ is
regular with normal crossings at $x$ if and only if $\widetilde{N} =
\widetilde{e}_v$.
\item The criterion of part (ii) is equivalent to 
$$\lambda_n \in (1/\widetilde{N})\ints \text{
  and } v(f) \in (d/\widetilde{N})\ints.$$ 
\end{enumerate}
\end{prop}

\begin{proof}
By Lemma~\ref{Lgeneratingdivisors2}, since $\divi(h)$ is vertical on
$\hat{\mc{O}}_{\mc{Y},y}$, we have $\divi(h) = s\divi(\psi)$ for some
$s \in \ints$, so $v(h) = s/N$.  This proves (i).

Now, we prove parts (ii) and (iii). If $D$ is the prime vertical divisor of
$\mc{Y}$ corresponding to $v$, then
by Corollary~\ref{Csmoothoncomponent} applied to $v_n$ (or to
$v_{n-1}$ if $v_{n-1} = v_n$), and with $g = \phi_n$ in that corollary, $\Spec \hat{\mc{O}}_{\mc{X},x}$
contains a unique prime divisor $\widetilde{D}$ above $D$ and $x$ is smooth on $\widetilde{D}$.
By Lemma~\ref{LregularUFD}, $\mc{X}$ is regular at $x$ if and only if
$\widetilde{D}$ is principal, and since $x$ is smooth on
$\widetilde{D}$, normal crossings is automatic.
%Note that the order function on $\widetilde{D}$ gives rise to
%the extension of the vaulation $v$, appropriately scaled.
%The value group of $v$ on $D$ is
%generated by $v(\psi)$ and $v(\phi_n)$, and thus the value group of
%$v$ on $\widetilde{D}$ is generated by $v(z)$, $v(\psi)$, and
%$v(\phi_n)$.
Let
$\widetilde{D}_{\phi_n}$ be the horizontal part of
$\divi(\nu^*\phi_n)$.  Recalling that $\widetilde{e}_v$ is the multiplicity of $\widetilde{D}$ in $\Spec
\hat{\mc{O}}_{\mc{X},x}$, we define
\begin{align*}
  D_1 := \divi(z) = \frac{1}{d}\divi(\nu^*f) &= 
                                        \frac{1}{d}\widetilde{e}_v\left(a\lambda_n + \frac{s}{N}  
  \right) \widetilde{D}  + \frac{a}{d} \widetilde{D}_{\phi_n} \\
  D_2 := \divi(\nu^*\psi) &= \widetilde{e}_v \frac{1}{N} \widetilde{D}  \\
  D_3 := \divi(\nu^*\phi_n) &= \widetilde{e}_v \lambda_n \widetilde{D}  + \widetilde{D}_{\phi_n} 
\end{align*}
By Lemma~\ref{Lgeneratingdivisors2}, the group $G$ of integer combinations of these divisors with support on
$\widetilde{D}$ is exactly the set of principal divisors supported on
$\widetilde{D}$.  So $\widetilde{D}$ is principal if and only if it
generates $G$.  Alternatively, $\widetilde{D}$ is principal if and only if the vertical parts of
$D_1$, $D_2$, and $D_3$ are in $G$ (the ``only if'' part is immediate
because the vertical parts of the $D_i$ are supported on $\widetilde{D}$,
and the ``if'' part follows because $\nu^*D$ lies in the group generated by
the vertical parts of $D_2$ and
$D_3$, so $\widetilde{D}$ lies in the group generated by the vertical
parts of $D_1$,
$D_2$, and $D_3$, see
Corollary~\ref{C:Multupstairs}).

Now, $G$ is generated by $$D_2 = \frac{\widetilde{e}_v}{N}\widetilde{D}
\qquad \text{and} \qquad \frac{dD_1
- aD_3}{\gcd(d,a)} = \frac{\widetilde{e}_vs}{N\gcd(d,a)} \widetilde{D}.$$  Pulling out a
factor of $\widetilde{e}_v/N$, and noting that the denominator of $s/\gcd(d,a)$ is
$\gcd(d,a)/\gcd(d, a, s)$, we have that $G$ is
generated by
$(\widetilde{e}_v/\widetilde{N}) \widetilde{D}$.  So $G$ is generated
by $\widetilde{D}$ if and only if $\widetilde{e}_v = \widetilde{N}$,
proving (ii).

Alternatively, the vertical part of $D_1$ is contained in $G$ if and
only if $v(f)/d \in (1/\widetilde{N}) \ints$, the vertical part of
$D_2$ is automatically in $G$, and the vertical part of $D_3$ is
contained in $G$ if and only if $\lambda_n \in
(1/\widetilde{N})\ints$.  This finishes the proof of part (iii).
\end{proof}

\begin{remark}\label{Rnobranchpoint}
In the situation of Proposition~\ref{Pstandardendpointregular}(iii) above,
if $f = h$ (so that $a = 0$), the condition $\lambda_n \in
(1/\widetilde{N})\ints$ automatically implies $v(f) \in
(d/\widetilde{N})\ints$.  This is because
$$v(f) = v(h) = \frac{s}{N} = \frac{sd}{\gcd(d, s) \widetilde{N}} \in \frac{d}{\widetilde{N}}\ints.$$
\end{remark}

Recall that the notion of \emph{geometric ramification} was defined in
Definition~\ref{Dramified}.
 By abuse of notation, if $\nu \colon \mc{X} \to \mc{Y}$ is a finite flat morphism of
arithmetic surfaces over $\Spec \mc{O}_K$, and if $y \in \mc{Y}$ lies
on a unique irreducible component
$\ol{W}$ of the special fiber of $\mc{Y}$, then we say $y$ is geometrically
ramified in $\mc{X} \to \mc{Y}$ if it is geometrically ramified in
$\nu^{-1}(\ol{W}) \to \ol{W}$.

\begin{prop}\label{Pbranchpointramindex}
Suppose that $f = \phi_n^a h$, where no horizontal part of $\divi(h)$
passes through $y$.  Let $s$ be such that $v(h)
= s/N$ as in Proposition~\ref{Pstandardendpointregular}(i).  Suppose
that each point $x \in \mc{X}$ above $y$ is regular.  Then the geometric
ramification index of $y$ in $\mc{X} \to \mc{Y}$ is $$\frac{de_v}{N\gcd(d,a)}.$$ In particular, $y$ is geometrically ramified whenever $e_v > N$.
\end{prop}

% The result below is for when $a = 0$ and the point above the
% standard endpoint is not necessarily regular.  I don't think we will
% need it.
%
%\item The geometric ramification index of any point above $y$ in the normalization
%  $\nu \colon \mc{X} \to \mc{Y}$ of
%$\mc{Y}$ in $K(X)$ is $\gcd(d, se_v/N)/\gcd(d,s)$.

\begin{proof}
Let $\ol{Z}$ be the $v$-component of $\mc{Y}$, and let $\ol{W}
= \nu^{-1}(\ol{Z})$.  The multiplicity of $\ol{Z}$ in the special
fiber of $\mc{Y}$ is $e_v$ and the multiplicity of $\ol{W}$ in the
special fiber of $\mc{X}$ is $\widetilde{e}_v = \widetilde{N}$ as in Proposition~\ref{Pstandardendpointregular}(ii).  
So the ramification index of $\ol{W}$ over $\ol{Z}$ is $\widetilde{N}/e_v$, which means that the induced morphism $\ol{W}^{\red} \to \ol{Z}^{\red}$
has degree $de_v/\widetilde{N}$.

On the other hand, $v(h) = sv(\psi)$, so $\divi(f) = s\divi(\psi) +
a\divi(\phi_n)$ in a formal neighborhood of $y$ in $\mc{Y}$.  By Lemma~\ref{Lallnthroots}, all units are
perfect $d$th powers in $\mc{O}_{\mc{Y},y}$, so we may assume $$f =
\phi_n^a\psi^s =
(\phi_n^{a/\gcd(d,a,s)}\psi^{s/\gcd(d,a,s)})^{\gcd(d,a,s)}.$$ Raising
$f$ to an appropriate prime-to-$d$th power, which does not affect the
cover, we may even assume
$$f =
(\phi_n^{\gcd(d,a)/\gcd(d,a,s)}\psi^{s'/\gcd(d,a,s)})^{\gcd(d,a,s)},$$
where $\gcd(d, s') = \gcd(d, s)$.
Since $\gcd(d,a)/\gcd(d,a,s)$ and $s'/\gcd(d,a,s)$ are relatively
prime, and neither $\phi_n$ nor $\psi$ is a non-trivial perfect power in
$\hat{\mc{O}}_{\mc{Y},y}$, we have that
$\phi_n^{\gcd(d,a)/\gcd(d,a,s)}\psi^{s'/\gcd(d,a,s)}$ is not a perfect
power either.
So $\mc{X}$ splits into $\gcd(d, a, s)$ connected components above a formal
neighborhood of $y$.
In particular, $\# \nu^{-1}(y) = \gcd(d, a, s)$.
We conclude that the geometric ramification index above $y$ is
$de_v/(\widetilde{N}\gcd(d,a,s))$, which equals $de_v/(N\gcd(d,a))$.
\end{proof}

\subsection{Standard $\infty$-specialization}\label{Sinfty}
If $V$ is a finite set of Mac Lane valuations with a unique minimal
valuation $v$, then Corollary~\ref{CAllspecializations}\ref{Cinftyspecialization} shows that $D_{\infty}$ meets the $V$-model $\mc{Y}$ of $\proj^1_K$
at a point $y \in \mc{Y}$ lying only on the $v$-component.  This meeting point is called the \emph{standard
  $\infty$-specialization} on $\mc{Y}$.  

Since everything in \S\ref{Sinfty} is local at the standard
$\infty$-specialization, we may as well suppose that $\mc{Y}$ is the $v$-model of $\proj^1_K$ for
$v = [v_0,\, \ldots,\, v(\phi_n) = \lambda_n]$.  Throughout
\S\ref{Sinfty}, \emph{we will assume that $n \leq 1$}.  In fact, if $v
= v_0$, we will write $v = [v_0,\, v_1(x) = 0]$, so that any $v$ we
consider can be written as $[v_0,\, v_1(\phi_1) = \lambda_1]$ for some
linear $\phi_1$.  As usual,
$\nu \colon \mc{X} \to \mc{Y}$ is the normalization of $\mc{Y}$ in
$K(\mc{X})$, where we recall that $K(\mc{X}) = K(t)[z]/(z^d - f(t))$
for a polynomial $f \in \mc{O}_K[x]$.  In
\S\ref{Sinfty}, we determine when a point $x$ (equivalently all points
$x$) of $\mc{X}$ above $y$ are
regular in the special case when the inductive length $n$ of $v$ is
$\leq 1$.

\begin{lemma}\label{Ltotallyarithmetic}
If $x$ is regular in
$\mc{X}$, then there exists $h \in \hat{\mc{O}}_{\mc{X}, x}$ such that
$h^{e_v} = \pi_K$.
\end{lemma}

\begin{proof}
We first claim that $\chara k \nmid e_v$.  Let $D$ be the prime divisor corresponding to the reduced special
fiber of $\Spec \hat{\mc{O}}_{\mc{Y}, y}$. 
  Since $\hat{\mc{O}}_{\mc{X}, x}$ is regular, Lemma~\ref{LregularUFD} shows
  that all height 1 ideals are principal.  So $\nu^*D$ is a principal,
  $\Sigma$-invariant divisor, where $\Sigma =
  \Aut(\hat{\mc{O}}_{\mc{X},x}/\hat{\mc{O}}_{\mc{Y},y})$.  By
  Lemma~\ref{Ldivgroupgenerators}, $\nu^*D$ is in the group generated
  by $\divi(z)$ and $H$, where $H$ is the group generated by $\nu^*(\beta)$ as $\beta$ ranges through
  $\hat{\mc{O}}_{\mc{Y},y}$.  Since $z^d \in \hat{\mc{O}}_{\mc{Y},y}$,
  we have that $\nu^*(dD) \in H$, and thus that $dD$ is a principal
  divisor of $\Spec \hat{\mc{O}}_{\mc{Y},y}$.  Since $e_v$ is the
  smallest positive integer such that $e_vD$ is a principal divisor of
  $\Spec \hat{\mc{O}}_{\mc{Y},y}$, we have that $e_v \mid d$.  By
  assumption, $\chara k \nmid d$, so $\chara k \nmid e_v$, proving the claim.

Now, let $h' \in
  \hat{\mc{O}}_{\mc{X}, x}$ be such that $\divi(h')$ is the principal
  divisor $\nu^*D$.   Since $\divi(\pi_K) = e_vD$ in $\Spec \hat{\mc{O}}_{\mc{Y}, y}$, we have
  $\divi((h')^{e_v}) = \nu^*(e_vD) = \nu^*\divi(\pi_K)$ in $\Spec \hat{\mc{O}}_{\mc{X}, x}$, which implies 
  $(h')^{e_v} = \pi_Ku$ for some $u \in
  \hat{\mc{O}}_{\mc{X},x}^{\times}$.  Since $\chara k \nmid e_v$,
  Lemma~\ref{Lallnthroots} shows that $u$ is an $e_v$-th power in
  $\hat{\mc{O}}_{\mc{X},x}$.  Letting $h = h' / \sqrt[e_v]{u}$ proves the lemma.
%  By Lemma~\ref{Lallnthroots}, $\pi_K$ is an $e_v$-th power
%  in $\hat{\mc{O}}_{\mc{X},x}$ as desired.
\end{proof}

\begin{lemma}\label{Lrootofpi}
Suppose the inductive length of $v$ is $\leq 1$.  Let $L/K$ be a
totally ramified field
extension of degree $e_v$ with ring of integers $\mc{O}_L$.  Then
$\mc{O}_{\mc{Y}, y} \otimes_{\mc{O}_K} \mc{O}_L$ is smooth as an
$\mc{O}_L$-algebra (and thus regular).
\end{lemma}

\begin{proof}
By assumption, $v = [v_0,\, v_1(\phi) = c/e_v]$ for some integer $c$
and linear polynomial $\phi$.  The ring $\mc{O}_{\mc{Y},y}$ consists of those elements of $K(t)$
whose pole divisors do not pass through $y$, that is, all rational
functions $h \in K(t)$ with $v(h) \geq 0$ and for which $D_{\alpha}$ does
not meet $y$ for any pole $\alpha$ of $h$.  Since $y$ is the
standard $\infty$-specialization, Proposition~\ref{PAllspecializations}\ref{Cphiizeroes} 
shows that this is equivalent to $v(h) \geq 0$ and $v_K(\phi(\alpha))
\geq c/e_v$ for all poles $\alpha$ of $h$.

Let $w$ be the unique extension of $v$ to
$L(t)$, renormalized so that $w(\pi_L) = 1$ (so $w = e_vv$ when
restricted to $K(t)$).  Now, $w = [v_0,\, v_1(\phi) = c]$ on $L(t)$.
 Just as above, $A := \mc{O}_{\mc{Y},y} \otimes_{\mc{O}_K} \mc{O}_L$ consists of those
rational functions $h$ in $L(t)$ such that $w(h) \geq 0$ and
$v_L(\phi(\alpha)) \geq c$ for all poles $\alpha$ of $h$.  That is, $A$ is the local ring of the
standard $\infty$-specialization on the $w$-model of $\proj^1_L$.
Making the change of variables $u = \phi/\pi_L^c$, we see that $w$ is
equivalent to the Gauss valuation on the variable $u$, which means the $w$-model of
$\proj^1_L$ is isomorphic to $\proj^1_{\mc{O}_L}$.  So all its local
rings are regular and smooth as $\mc{O}_L$-algebras.
\end{proof}

\begin{lemma}\label{Linftyfactorization}
Suppose the inductive length of $v$ is $\leq 1$.  Write the irreducible factorization of $f$ in $\mc{O}_K[t]$ as $\pi_K^af_1^{a_1} \cdots
f_r^{a_r}$ with all $f_i$ having unit leading coefficient.
Order the factors $f_i$ so that there exists $s$ with $1 \leq s \leq r$ such that $v
\not \prec v_{f_i}^{\infty}$ for $i \leq s$ and $v \prec
v_{f_i}^{\infty}$ for $i > s$.  In
$\hat{\mc{O}}_{\mc{Y}, y}$, up to multiplication by $d$th powers, the irreducible factorization of $f$ is
$$\pi_K^af_1^{a_1} \cdots f_s^{a_s} \phi_1^e,$$ where $e = a_{s+1}\deg(f_{s+1}) +
\cdots + a_r\deg(f_r)$ and $1 \leq a_i \leq d$ for all $i$.
\end{lemma}

\begin{proof}
 Clearly there is no problem requiring $1 \leq a_i \leq d$ for $i \leq
 s$.  Now, write $v = [v_0,\, v_1(\phi_1) =  \lambda_1]$, with
 $\lambda_1 = 0$ if $v$ has inductive length $0$.  Consider $f_i$ for $i > s$.  By Corollary~\ref{Cdominantterm},
 $v(f_i) = \ell v(\phi_1)$, where $\ell = \deg(f_i)/\deg(\phi_1) =
 \deg(f_i)$.  So $\divi(f_i)$ and $\divi(\phi_1^{\ell})$ have the same
vertical part in $\Spec \hat{\mc{O}}_{\mc{Y},y}$.   Also, the divisors of $f_i$ and $\phi_1$ have the
same negative horizontal part, namely $-\ell D_{\infty}$.  Lastly, the divisors of $f_i$ and
$\phi_1$ have no positive horizontal part in $\Spec \hat{\mc{O}}_{\mc{X},x}$, by
Proposition~\ref{PAllspecializations}\ref{Cphiizeroes} in the case of $\phi_1$ and by combining
Lemma~\ref{Lgfollows} and 
Proposition~\ref{Pparameterize}\ref{Psmallerdiskoid} in the case of $f_i$.  So
$f_i^{a_i}$ is the same as $\phi_1^{\ell a_i}$ up to multiplication by
units.  Since all units are $d$th powers by Lemma~\ref{Lallnthroots}, this shows that
$$f_{s+1}^{a_{s+1}} \cdots f_r^{a_r} \sim \phi_1^e,$$ where $\sim$
means equality up to multiplication by $d$th powers in
$\hat{\mc{O}}_{\mc{Y},y}$.

It remains to show that $f_i$ is irreducible in
$\hat{\mc{O}}_{\mc{Y},y}$ for $i \leq s$.  In this case, combining 
Lemma~\ref{Lgfollows} and Proposition~\ref{Pparameterize}\ref{Psmallerdiskoid} shows that
the positive horizontal part of $\divi(f_i)$ passes through $y$, so it
is a prime divisor in $\Spec \hat{\mc{O}}_{\mc{Y},y}$.  This proves
the irreducibility.
\end{proof}

\begin{lemma}\label{Llinearoreisenstein}
Suppose the inductive length of $v$ is $\leq 1$, so $v = [v_0,\,
v_1(\phi_1) = \lambda_1]$.  Let $\alpha \in
\ol{K}$ such that
$D_{\alpha}$ meets the standard $\infty$-specialization on $\mc{Y}$.
\begin{enumerate}[\upshape (i)]
  \item If $e_v = 1$, then $D_{\alpha}$ is regular on $\mc{Y}$ if and only if $\alpha \in
    K$ or $v_K(\phi_1(\alpha)) = \lambda_1 - 1/\deg(\alpha)$.
  \item If $e_v > 1$, let $L = K[\sqrt[e_v]{\pi_K}]$, with valuation
    ring $\mc{O}_L$.  If the minimal polynomial of
    $\alpha$ over $L$ is in fact defined over $K$, then $D_{\alpha}$
    is regular over $\mc{Y} \otimes_{\mc{O}_K} \mc{O}_L$ if and only
    if $\alpha \in K$.
\end{enumerate}
\end{lemma}

\begin{proof}
If $e_v = 1$, then $\lambda_1 \in \ints$, so under the change of variables $u = \phi_1(t)/\pi_K^{\lambda_1}$, we see that $D_{\alpha}$ (in
terms of $t$) is $D_{\phi_1(\alpha)/\pi_K^{\lambda_1}}$ (in terms of $u$).
 So renaming $u$ as $t$
again, we may assume $\phi_1(t) = t$ and $\lambda_1 = 0$, that is, $v = v_0$.  Thus we
may assume we are on the $v_0$-model $\proj^1_{\mc{O}_K}$ of
$\proj^1_K$.  Now, the maximal ideal $\mf{m}$ of the local ring of the
$\infty$-specialization on $\proj^1_{\mc{O}_K}$ is generated by $t^{-1}$
and $\pi_K$.  Since $D_{\alpha}$ meets the $\infty$-specialization,
Proposition~\ref{PAllspecializations}\ref{Clowvalspecialization} shows that $v_K(\alpha) < 0$.
If $g(t)$ is the monic minimal polynomial of $\alpha^{-1}$, then
since $v_K(\alpha^{-1}) > 0$, all non-leading coefficients of $g(t)$
have positive valuation.  Thus $\divi(g(t^{-1}))$ has no vertical part
on $\proj^1_{\mc{O}_K}$, and we conclude that $D_{\alpha} =
\divi(g(t^{-1}))$.  So $D_{\alpha}$ is regular if and only if
$g(t^{-1}) \notin \mf{m}^2$, which is equivalent to $g$ being linear
or Eisenstein.  This is in turn equivalent to $\alpha \in K$ or
$v_K(\alpha) = -1/\deg(\alpha)$, proving (i).

If $e_v > 1$, letting $w$ be the extension of $v$ to $\mc{O}_L(t)$, we
have that $w = [v_0,\, v_1(\phi_1) = e_v\lambda_1]$, with $e_v
\lambda_1 \in \ints$.  As in the previous paragraph, we may assume $w$
is the Gauss valuation on $L(t)$ and that $\mc{Y} \otimes_{\mc{O}_K}
\mc{O}_L$ is $\proj^1_{\mc{O}_L}$.  The maximal ideal at the point
above $y$ on $\proj^1_{\mc{O}_L}$ is generated by $t^{-1}$ and a
uniformizer $\pi_L$ of $L$.  As in the previous paragraph,
$D_{\alpha}$ is regular if and only if the minimal polynomial $g$ of
$\alpha^{-1}$ is linear or Eisenstein over $L$.  But since $g$ is
defined over $K$, it is not Eisenstein over $L$.  This proves part (ii).
\end{proof}

\begin{lemma}\label{Linftyequivalence}
Suppose the inductive length of $v$ is $\leq 1$.   Write the irreducible factorization of $f$ in $\hat{\mc{O}}_{\mc{Y}, y}$ as $\pi_K^af_1^{a_1} \cdots
f_s^{a_s} \phi_1^e$ with $v \not \prec v_{f_i}^{\infty}$ for all $i$ as in Lemma~\ref{Linftyfactorization}.  Let $\beta =
\gcd(d, a, a_1, \ldots, a_s, e)$. 
If $x \in \mc{X}$ is a point above the standard
$\infty$-specialization $y$, then the following two conditions are equivalent:
\begin{enumerate}[\upshape (a)]
    \item $e_v \mid \gcd(d, a_1, \ldots, a_s, e) / \beta$.
     \item $\hat{\mc{O}}_{\mc{X},x}$ contains an $e_v$-th root of $\pi_K$.
     \end{enumerate}
\end{lemma}

\begin{proof}
First, observe that $\hat{\mc{O}}_{\mc{X},x}$ is given by normalizing
$\hat{\mc{O}}_{\mc{Y},y}$ in the function field given
by $$\frac{\Frac(\hat{\mc{O}}_{\mc{Y},y})[z]}{(z^{d/\beta} - \pi_K^{a/\beta}
  f_1^{a_1/\beta} \cdots f_s^{a_s/\beta} \phi_1^{e/\beta})}.$$  By replacing $a$,
$d$, the $a_i$, and $e$ by their quotients by $\beta$, we may assume that
$\beta = 1$.

Now, if condition (a) holds, then the field extension of
$\Frac \hat{\mc{O}}_{\mc{Y},y}$ given by taking an $e_v$-th root of
$f$ is the same as that given by taking an $e_v$-th root of $\pi_K^a$,
which, since $\beta = 1$, is the same as that given by extracting an
$e_v$-th root of $\pi_K$.  Also, since $e_v \mid d$, this field
extension is contained in $\Frac \hat{\mc{O}}_{\mc{X},x}$.  Since
$\hat{\mc{O}}_{\mc{X},x}$ is normal, it contains an $e_v$-th root of
$\pi_K$, proving condition (b).

On the other hand, suppose (b) holds, so $\hat{\mc{O}}_{\mc{X},x}$ contains an $e_v$-th
root of $\pi_K$, which
we call $\pi_L$.  Since
$\hat{\mc{O}}_{\mc{X},x} / \hat{\mc{O}}_{\mc{Y},y}$ is a
$\ints/d$-extension, the extension
$A / \hat{\mc{O}}_{\mc{Y},y}$, where $A = \hat{\mc{O}}_{\mc{Y},y}[\pi_L]$,
is the unique $\ints/e_v$-subextension of
$\hat{\mc{O}}_{\mc{X},x}/\hat{\mc{O}}_{\mc{Y},y}$.  So $e_v \mid d$,
and $A$
is isomorphic to the normalization of $\hat{\mc{O}}_{\mc{Y},y}$ in the fraction
field extension given by taking an $e_v$-th root of $f$, which by Kummer
theory, in turn implies that some prime-to-$e_v$-th power of $\pi_K$ equals $f$ up to
multiplication by $e_v$-th powers in $\hat{\mc{O}}_{\mc{Y},y}$.  This
shows that $e_v \mid a_i$ for all $i$, and $e_v \mid e$, and thus
condition (a) holds since $\beta=1$.  This completes the proof.
   \end{proof}
   
The following proposition is the main result of \S\ref{Sinfty}, and
its proof uses the lemmas stated above.

\begin{prop}\label{PinftySNC}
Maintain the notation and assumptions of
Lemma~\ref{Linftyequivalence}. Then $\mc{X}$
is regular with normal crossings at $x$ if and only if condition (i),
as well as one of conditions (ii), (iii), or (iv) below holds:
\begin{enumerate}[\upshape (i)]
 \item $e_v \mid \gcd(d, a_1, \ldots, a_s, e) / \beta$ (this is
   condition (a) of Lemma~\ref{Linftyequivalence}.)
 \item $s = 0$ (i.e., up to $d$-th
  powers, $f = \pi_K^a \phi_1^e$).
\item $s = 1$, $f_1$ is linear, and $d/\gcd(d, a_1)$
  is relatively prime to $d/\gcd(d, e_vv(f))$.
\item $s = 1$ with $e_v = 1$, $d = 2\beta$ and $2\beta \mid v(f)$, $f_1$
  quadratic, and $v_K(\phi_1(\alpha_1)) = \lambda_1 -
  1/2$, where $\alpha_1$ is any root of $f_1$.
\end{enumerate}
If conditions (i) and (ii) hold, then $x$ is furthermore smooth on the
reduced special fiber of $\mc{X}$.
\end{prop}

\begin{proof}
As in Lemma~\ref{Linftyequivalence}, replacing $a$,
$d$, the $a_i$, and $e$ by their quotients by $\beta$ (which replaces
$f$ by $f^{1/\beta}$ and thus does not change the quantities in part (iii)), we may assume that
$\beta = 1$.

Next, note
that if $x$ is regular, then Lemma~\ref{Ltotallyarithmetic} implies condition (b) of
Lemma~\ref{Linftyequivalence}.  By Lemma~\ref{Linftyequivalence}, this implies
condition (i).  To finish the proof, we will show,
assuming condition (i),
that $x$ being regular with normal crossings is equivalent to one of conditions (ii), (iii),
or (iv) (and that under condition (ii), $x$ is non-nodal).

So assume condition (i).  By Lemma~\ref{Linftyequivalence},
$\hat{\mc{O}}_{\mc{X},x}$ contains an $e_v$-th root of $\pi_K$, say $\pi_L$.
By Lemma~\ref{Lrootofpi}, $A  := \hat{\mc{O}}_{\mc{Y},y}[\pi_L]$ is in fact regular and smooth as
  an $\mc{O}_L$-algebra where $\mc{O}_L := \mc{O}_K[\pi_L]$.  Now, $\hat{\mc{O}}_{\mc{X},x}$ is the
normalization of $A$ in the field
$\Frac(A)[z]/(z^{d/e_v} - f^{1/e_v})$, where
\begin{equation}\label{Eupstairs}
  f^{1/e_v} := \pi_L^af_1^{a_1/e_v} \cdots
  f_s^{a_s/e_v}\phi_1^{e/e_v}.
\end{equation}

Let us examine the ramification divisor $B$ of the degree $d/e_v$ morphism
$\Spec \hat{\mc{O}}_{\mc{X},x}
\to \Spec A$, beginning with the horizontal part. Since $d/e_v \mid
\deg(f^{1/e_v})$, the negative part of $\divi(f^{1/e_v})$ does not contribute
to horizontal ramification.  So if condition (ii) holds, there is no horizontal
ramification, and Proposition~\ref{P:RegandNormalize}(i) shows that $x$
is regular and non-nodal in $\mc{X}$.

On the other hand, if condition (ii) fails,
then Proposition~\ref{PAllspecializations}\ref{Clowvalspecialization} shows that $\divi(f_i)$ appears with nonzero multiplicity in $\divi(f^{1/_{e_v}})$ in $\Div(\Spec(A))$. Furthermore, the multiplicity of each $\divi(f_i)$ in $\divi(f^{1/_{e_v}})$ is not divisible by $d/e_v$ in $\Div(\Spec(A))$, and thus $\divi(f_i)$ is in $B$.  In this case, Proposition~\ref{P:RegandNormalize}(ii)
shows that $x$ is regular with normal crossings only if the horizontal
part of $B$ is
irreducible, which implies $s = 1$.  The horizontal part of $B$ has
ramification index
$$e_{\rm horiz} := \frac{d/e_v}{\gcd(d/e_v, a_1/e_v)} =
\frac{d}{\gcd(d, a_1)}$$ in this case.  Assuming $s = 1$, it remains to show that $x$ is
regular with normal crossings if and only if condition (iii) or (iv) holds.

Let $w$ be the extension of $v$ to $A$, thought of as a Mac Lane
valuation on $L(t)$ with $L = \Frac \mc{O}_L$ (i.e., so that $w(\pi_L)
= 1$).  Note that $e_w = 1$ by
construction, since $w$ is unramified over $v$ so $e_w = (1/e_v)e_v$. Now, the ramification
index of $\Spec \hat{\mc{O}}_{\mc{X},x} \to \Spec A$ along the special
fiber is $$e_{\rm vert} := \frac{d/e_v}{\gcd(d/e_v, e_w w(f^{1/e_v}))} =
\frac{d}{\gcd(d, e_w w(f))} = \frac{d}{\gcd(d, w(f))} =
\frac{d}{\gcd(d, e_v v(f))}.$$
If $e_{\rm vert} > 1$, then $B$ has a vertical part, so by
Proposition~\ref{P:RegandNormalize}(ii)(b), $\mc{X}$ is regular with
normal crossings at $x$ if and only if $f_1$ is linear and $e_{\rm
  vert}$ is relatively prime to $e_{\rm horiz}$.  This is true if and
only if condition
(iii) holds.

On the other hand, suppose $e_{\rm vert} = 1$, which means $d \mid e_v
v(f)$ and $B$ has only a
horizontal part.  By Proposition~\ref{P:RegandNormalize}(ii), $B$ must
be irreducible and regular for
$\mc{X}$ to be regular with normal crossings at $x$.  This requires
first that $f_1$ is
irreducible over $\mc{O}_L$, which means that the minimal polynomial
of any root $\alpha_1$ of $f_1$ over $L$ is just $f_1$.  In
particular, $B = D_{\alpha_1}$.

If $e_v > 1$, then Lemma~\ref{Llinearoreisenstein}(ii) shows that $D_{\alpha_1}$ is
regular on $\Spec A$ if and only if $f_1$ is linear, which is
equivalent to condition (iii) holding.  By
Proposition~\ref{P:RegandNormalize}(ii)(a), this is in fact equivalent
to $\mc{X}$ being regular with normal crossings at $x$.

If $e_v = 1$, then Lemma~\ref{Llinearoreisenstein}(i) shows that $D_{\alpha_1}$ is regular on $\Spec A$ if and only if $f_1$
is linear or $v_K(\phi_1(\alpha_1)) = \lambda_1 - 1/\deg(f_1)$.  Now, 
$f_1$ is linear if and only if condition (iii) holds, and this again
is equivalent to $\mc{X}$ being regular with normal crossings at $x$
as in the previous paragraph.  On the other
hand, if $\deg(f_1) > 1$, then
Proposition~\ref{P:RegandNormalize}(ii)(a) shows that $\mc{X}$ is
regular with normal crossings at $x$ if and only if $D_{\alpha_1}$ is
regular, $\deg(f_1) = 2$, and $d/e_v = d = 2$ and $d \mid e_vv(f) = v(f)$.  This is exactly condition
(iv), completing the proof.
%
%To prove the last statement in the proposition, we note that if
%conditions (i) and (ii) hold, then
%Proposition~\ref{P:RegandNormalize}(iii) shows that $\mc{X}$ has normal
%crossings at $x$, covering case (a) of the normal crossings criterion.  If conditions (i) and (iii) hold,
%then Corollary~\ref{Clinear} applied to $\Spec \hat{\mc{O}}_{\mc{X},x}
%\to \Spec A$ as a morphism of local arithmetic surfaces over
%$\mc{O}_L$ shows that $\mc{X}$ has normal crossings
%at $x$ if and only if $f_1$ is linear, or $f_1$ is 
%quadratic over $\mc{O}_L$ (and thus over $\mc{O}_K$) and $d/e_v = 2$.  Since $e_v = 1$ in this
%case according to condition (iii), this
%covers cases (b) and (c) of the normal crossings criterion.
%The proof is complete.
\end{proof}

\begin{corollary}\label{CinftySNC}
If $x$ is regular in $\mc{X}$, then the geometric ramification index
above $y$ in $\mc{X} \to \mc{Y}$ is divisible by $e_v$.  The divisibility is
strict if and only if $s = 1$ in
Proposition~\ref{PinftySNC} above, that is, if there exists $i$ with
$v \not \prec v_{f_i}^{\infty}$.
\end{corollary}

\begin{proof}
An $e_v$-th root of $\pi_K$ in $\hat{\mc{O}}_{\mc{X},x}$ is
  guaranteed by Lemma~\ref{Ltotallyarithmetic}, so
  let $A = \hat{\mc{O}}_{\mc{Y},y}[\sqrt[e_v]{\pi_K}] \subseteq
  \hat{\mc{O}}_{\mc{X},x}$.  Since $\Spec A \to \Spec
  \hat{\mc{O}}_{\mc{Y},y}$ is a Kummer cover given by extracting an
  $e_v$-th root of $\pi_K$, and $e_v \mid e_vv(\pi_K)$, the cover is
  unramified along the special fiber.  On the other hand, $A$ is a
  local ring, so $\Spec A$ contains only one point above $y$.  So the
  geometric ramification index of $\Spec A \to \Spec
  \hat{\mc{O}}_{\mc{Y},y}$ above $y$ is $e_v$.

 Now consider $\Spec \hat{\mc{O}}_{\mc{X},x} \to \Spec A$.  Then $s =
 1$ if and only if this morphism has non-trivial horizontal ramification
 divisor (because $\Spec A \to \Spec \hat{\mc{O}}_{\mc{Y},y}$ clearly
 does not have horizontal ramification).  By Lemma~\ref{Lrootofpi}, $\Spec A$ is regular and has
 smooth special fiber as an $\mc{O}_{K[\sqrt[e_v]{\pi_K}]}$-scheme.
 By Corollary~\ref{Cregularbranchindex}, the geometric ramification
 index $c$ of the point above $y$ in $\Spec \hat{\mc{O}}_{\mc{X},x} \to
 \Spec A$ is greater than $1$ if and only if $s = 1$.  Thus the geometric ramification index
 above $y$ in $\mc{X} \to \mc{Y}$ is $ce_v$, proving the corollary.
\end{proof}

\subsection{$\infty$-crossings}\label{Sinftycrossing}
Let $V$ be a finite set of Mac Lane valuations with exactly \emph{two}
minimal valuations $v$ and $v'$. Let $\mc{Y}$ be the $V$-model of $\P^1_K$ and let $y \in \mc{Y}$ be the intersection of the $v$ and $v'$ components in $\mc{Y}$.  Assume further that $v = [v_0,\, v_1(t - c)
= \mu]$ and $v' = [v'_0 := v_0, v'_1(t - c') = \mu']$, for some $c, c'
\in \mc{O}_K$ with $v_K(c - c') = 0$ and $\mu, \mu' > 0$.  
  We call $y$ the \emph{$\infty$-crossing} on $\mc{Y}$, since $D_{\infty}$ meets the special fiber of the
$V$-model $\mc{Y}$ of $\proj^1_K$ at the intersection point $y$ of the $v$- and $v'$-components by Corollary~\ref{CAllspecializations}\ref{Cinftyspectwominimal}.

%Let $D_v$, $D_{v'}$ be the vertical prime divisors on $\mc{Y}$
%corresponding to $v$ and $v'$, respectively.
Assume that we can write
$f = \pi_K^a jj'$ for monic $j$ and $j'$ in $\mc{O}_K[t]$ (here $j'$
does \emph{not} mean the derivative of $j$), with every irreducible factor $\psi$ of $j$ satisfying $v \prec
v_{\psi}^{\infty}$ and every irreducible factor $\psi'$ of $j'$
satisfying $v' \prec v_{\psi'}^{\infty}$.
%This assumption implies that the only horizontal part
%of $\divi(f)$ that passes through $y$ is supported on $D_{\infty}$.
Assume further that $d \mid \deg(f)$, and write $\delta$ (resp.\ $\delta'$) for $\deg(j)$ (resp.\ $\deg(j')$).  

\begin{lemma}\label{Linftytostandardcrossing}
  Let $\alpha \in \ints$, and consider the change of variables $u = \pi^{\alpha}(t-c')/(t-c)$.
  Then, up to multiplying by $d$th powers in $K(u)$, we can write $f(t)$ as a product of polynomials $g(u)h(u)$ in $\mc{O}_K[u]$ where
  \begin{itemize}
    \item The leading coefficient of $g(u)$ is in $\mc{O}_K^{\times}$, every zero $\theta$ of $g(u)$ satisfies $v_K(\theta) \geq
      \alpha + \mu'$, and $\deg(g(u)) = \delta'$.
    \item Every zero $\theta$ of $h(u)$ satisfies $v_K(\theta) \leq
      \alpha - \mu$, and $\deg(h(u)) = \delta$.
    \item The constant term of $h(u)$ has valuation $a + \delta\alpha$.
  \end{itemize}
  \end{lemma}

  \begin{proof}
    By Lemma~\ref{Lgfollows}, each zero $\gamma$ (resp.\ $\gamma'$) of
    $j(t)$ (resp.\ $j'(t)$) satisfies
    $v_K(\gamma - c) \geq \mu$ (resp.\ $v_K(\gamma - c') \geq
    \mu'$).
Let $\tilde{g} \in K(u)$ be such that $\tilde{g}(u) = j'(t)$.
    Then each zero $\theta$ of $\tilde{g}(u)$ is
    $\pi_K^{\alpha}(\gamma-c')/(\gamma-c)$ for some zero $\gamma$
    of $j'(t)$, and thus satisfies $v_K(\theta) \geq \alpha +
    \mu'$.  Furthermore, since $j'$ has a single pole of order $\delta'$ at $t=\infty$, it follows that $\tilde{g}(u)$ has a single pole of order
    $\delta'$ at $u = \pi_K^{\alpha}$.  Likewise, letting $\tilde{h} \in K(u)$
    be such that $\tilde{h}(u) = \pi_K^aj(t)$, we have that each
    zero $\theta$ of $\tilde{h}(u)$ satisfies $v_K(\theta) \leq \alpha
    - \mu$, and that $\tilde{h}(u)$ has a single pole of order $\delta$
    at $u = \pi_K^{\alpha}$.  Let $g(u) \colonequals
    \tilde{g}(u)(u-\pi_K^{\alpha})^{\delta'}$ and $h(u) \colonequals \tilde{h}(u)(u -
    \pi_K^{\alpha})^{\delta}$.  Then $g(u)$ and $h(u)$ are polynomials of the same degrees as
    $j'(t)$ and $j(t)$ respectively, and the zeroes of
    $g(u)$ and $h(u)$ are as required in the lemma. Since $d \mid \delta + \delta'$ and $f(t) = \tilde{g}(u) \tilde{h}(u)$ by assumption, we have that
    $g(u)h(u)$ equals $f(t)$ up to multiplication by $d$th powers.

    It remains to show that the leading coefficient of $g(u)$ and the
    constant term of $h(u)$ are as in the lemma. If $\gamma_1, \ldots,
    \gamma_{\delta'}$ are the roots of $j'(t)$ (with multiplicity) in $\ol{K}$, then
    one calculates that
    $$g(u) = \prod_{i=1}^{\delta'} ((c - \gamma_i)u +
    \pi_K^{\alpha}(\gamma_i - c')).$$  Since all $\gamma_i$ satisfy
    $v_K(\gamma_i - c') \geq \mu' > 0$, we have $v_K(\gamma_i - c)
    = 0$, which proves that the leading coefficient of $g(u)$ is a
    unit.  Similarly, if $\epsilon_1, \ldots, \epsilon_{\delta}$ are the
    roots of $j(t)$ (with multiplicity) in $\ol{K}$, then
    $$h(u) = \pi_K^a \prod_{i=1}^{\delta} ((c - \epsilon_i)u +
    \pi_K^{\alpha}(\epsilon_i - c')),$$ and $v_K(\epsilon_i - c') = 0$
    for all $i$, so the constant coefficient of $h(u)$ has valuation
    $a + \delta\alpha$.
  \end{proof}

\begin{prop}\label{Pinftycrossingregular}
Let $\nu \colon \mc{X} \to \mc{Y}$ be the normalization of $\mc{Y}$ in $K(X)$. 
Let $r$ be an integer such that $r\delta'/\gcd(d, \delta') \equiv 1 \pmod{d/\gcd(d,
  \delta')}$. 
If $x \in \mc{X}$ is a point above $y \in \mc{Y}$, then $x$ is
  regular if and only if
$$\frac{\gcd(d, \delta')}{d}\mu' + \frac{ra}{d} >
-\frac{\gcd(d,\delta')}{d}\mu + \frac{ra}{d}$$ is a $\widetilde{N}$-path, where $\widetilde{N} =
\gcd(d, \delta') / \gcd(d, a, \delta')$.   Furthermore, in this case, the special fiber of
    $\mc{X}$ has normal crossings at $x$.
  \end{prop}

\begin{proof}
Pick $\alpha \in \nats$ such that $\alpha > \mu$, and make the change
of variable $u = \pi_K^{\alpha}(t-c')/(t-c)$.  By
Proposition~\ref{Pchangeofvariable}, when written in terms of $u$, we have
$v = [v_0, v_1(u) = \alpha - \mu]$ and $v' = [v_0, v'_1(u) = \alpha +
\mu']$, so the point $y$ becomes a standard crossing.  Write $f
= g(u)h(u)$ as in Lemma~\ref{Linftytostandardcrossing}.  Note that all
roots $\theta$ of $g(u)$ satisfy $v_K(\theta) \geq \alpha + \mu'$,
whereas all roots $\theta$ of $h(u)$ satisfy $v_K(\theta) \leq \alpha
- \mu$, so $g$ and $h$ play the same roles as in
\S\ref{Sstandardcrossings} (see the discussion before
Lemma~\ref{Ldivisormultiplicities}).  Furthermore, no
horizontal part of $\divi(f)$ passes through $y$ by
Proposition~\ref{Pparameterize}\ref{Pannulus}.  In the
language of Proposition~\ref{Pstandardcrossingregular}, we have $d =
d$, $N = 1$, $e = \deg(g(u)) = \delta'$, $\lambda' = \alpha + \mu'$, and
$\lambda = \alpha - \mu$.  Also, we have $s = v(h(u))$,
which by Lemma~\ref{Ldominantterm}(ii) equals $v_K(a_0)$, where $a_0$ is the constant coefficient of
$h(u)$.  So $s = a + \delta \alpha$ by
Lemma~\ref{Linftytostandardcrossing}.  So the criterion for $x$ being
regular with normal crossings in
Proposition~\ref{Pstandardcrossingregular} becomes
\begin{equation}\label{Epath}
  \frac{\gcd(d,\delta')}{d} (\alpha + \mu') + \frac{r(a + \delta \alpha)}{d}
  > \frac{\gcd(d, \delta')}{d} (\alpha - \mu) + \frac{r(a + \delta \alpha)}{d}
\end{equation}
being a $\gcd(d, \delta')/\gcd(d, \delta', a + \delta\alpha)$-path, where $r \delta'
\equiv \gcd(d, \delta') \pmod{d}$ as in the proposition.  But since $d \mid (\delta' +
\delta)$, we have $\gcd(d, \delta') \mid \delta$, so $\gcd(d, \delta', a + \delta
\alpha) = \gcd(d, \delta', a)$, so $x$ is regular with normal crossings if
and only if \eqref{Epath} is an $\widetilde{N}$-path.  Since $\delta
\equiv -\delta' \pmod{d}$, we have $r \delta \equiv -\gcd(d, \delta') \pmod{d}$, so
\eqref{Epath} simplifies to
$$  \frac{\gcd(d, \delta')}{d}\mu' + \frac{ra}{d} + n
> -\frac{\gcd(d, \delta')}{d}\mu + \frac{ra}{d} + n,$$ where $n =
\alpha(\gcd(d, \delta') + r \delta)/d \in \ints$.  But it is clear from
Definition~\ref{DNpath} that adding the same integer to each entry in a
decreasing sequence does not affect whether or not it is an
$\widetilde{N}$-path, so we can ignore the $n$, which gives the
criterion from the proposition.
\end{proof}

\begin{remark}\label{Rneverwithmonic}
Note that if $f$ is \emph{monic}, then $a = 0$ and the criterion in
Proposition~\ref{Pinftycrossingregular} never holds, since $m > n$
can never be an $\widetilde{N}$-path if $m$ is positive and $n$ is negative.
\end{remark}

%\begin{corollary}\label{Cinftycrossingregular}
%  The model $\mc{X}$ is regular with normal crossings at $x$ if and
%  only if the condition of Proposition~\ref{Pinftycrossingregular} holds.
%\end{corollary}

%\begin{proof}
%This follows immediately from Lemma~\ref{Lbothprincipal} and Proposition~\ref{Pinftycrossingregular}.
%\end{proof}

\section{Construction of regular normal crossings models of cyclic covers}\label{Sconstruction}

%\begin{itemize}
%\item First build a regular SNC-model
%  \begin{itemize}
%    \item Start with the (pseudo-valuation) $v_{f_i}^{\infty}$'s, as well as their predecessors.\andrew{Done.}
%    \item Pass to inf-closure. \andrew{Done.}
%    \item Put in the correct paths between the $v_{f_i}$'s and their
%      predecessors. \andrew{Done.  These are the ``links''.}
%    \item Put in the correct ``dead-end'' paths from the $v_{f_i}$'s
%      \andrew{Done.  These are the ``tails'' and ``branch point tails''.}
%    \item Show that all standard crossings, standard endpoints, and
%      all other points are regular. \andrew{Done.}
%    \end{itemize}
%  \item Now contract to make the model minimal
%    \begin{itemize}
%    \item Show nothing can be contracted in the middle. \andrew{Done.}
%    \item Show when $v_{f_i}$'s can be contracted. \andrew{Done.}
%    \item Show when $v_0$ can be contracted. \andrew{Done.}
%    \end{itemize}
%  \end{itemize}

Let $\nu \colon X \to Y \cong \proj^1_K$ be a $\ints/d$-cover, and
assume $\chara k
\nmid d$.  In this section, we will construct a normal model $\mc{Y}_{\reg}$ of $Y$ such
that the normalization $\mc{X}_{\reg}$ of $\mc{Y}_{\reg}$ in $K(X)$ is the minimal regular
normal crossings model of $X$.
%After some preliminary reductions in
%\S\ref{Sreductions}, our construction proceeds in two steps.  In \S\ref{Sfirstmodel}, we construct a
%normal model $\mc{Y}_{\reg}$ of $Y$ whose normalization in $K(X)$ is a
%regular model $\mc{X}_{\reg}$ of $X$ with normal crossings.
The model $\mc{X}_{\reg}$
often \emph{is} the minimal regular model with normal crossings,
%---
%for instance, if $\nu$ splits completely above $\infty$\footnote{e.g.,
%  if $\nu$ is given by a Kummer equation $z^d = f$ with $\deg d \mid
%  f$ and $f \in \mc{O}_K[t]$ monic.} and $d$ is odd this is always
%  the case ---
but
sometimes $\mc{X}_{\reg}$ has components on the special fiber that can be
contracted.  
%\padma{Move this to Section~9. One of the reductions isn't necessary here.}  In \S\ref{Ssecondmodel}, we explicitly calculate those
%components of the special fiber of $\mc{Y}_{\reg}$ such that contracting
%them yields a model $\mc{Y}_{\min}$ whose normalization in
%$K(X)$ is the minimal regular model $\mc{X}_{\min}$ of $X$ with normal crossings.
%
Before we begin the construction we introduce some terminology that will be
useful throughout \S\ref{Sconstruction} and \S\ref{Ssecondmodel}.

\begin{defn}\label{Dregularbase}\hfill
%Fix a coordinate on $\proj^1_K$ so that finite sets of Mac Lane valuations correspond
%to normal models of $\proj^1_K$ as in \S\ref{Smaclanemodels}.
  \begin{enumerate}[\upshape (i)]
\item A nonempty finite set $V$ of geometric valuations is a \emph{regular normal
crossings base} for $X \to \proj^1_K$ if the normalization of the $V$-model in $K(X)$ is a regular model of $X$ with normal
crossings.
\item Suppose $V$ is a regular normal crossings base.  A valuation $v
\in V$ is \emph{removable} from $V$ if $V \setminus \{v\}$ remains
a regular normal crossings base.
\end{enumerate}
\end{defn}

%Recall from Algorithm~\ref{AY0} that we have sets of valuations $V_1
%\subseteq V_2 \subseteq V_3 \subseteq V_4 \subseteq V_{\reg}$, where
%the model $\mc{Y}_{\reg}$ of $\proj^1_K$ corresponding to $V_{\reg}$
%has the property that the normalization $\mc{X}_{\reg}$ of $\mc{Y}_{\reg}$ in $K(X)$
%is regular with normal crossings.  That is, $V_{\reg}$ is a regular
%normal crossings base.

\subsection{A preliminary reduction}\label{Sreductions}

Recall that $\nu \colon X \to Y \cong \proj^1_K$ is a $\ints/d$-cover
with $\chara k
\nmid d$.  Since $t$ is a fixed coordinate on $\proj^1_K$, Kummer
theory shows that $\nu$ is given birationally by the equation $z^d =
f(t)$.  By changing $t$-coordinates on $\P^1_K$ using an element of
$\GL_2(K)$, we may assume that no branch point of $\nu$ specializes to
$\infty$ on the special fiber of the standard model
$\proj^1_{\mc{O}_K}$ of $\proj^1_K$.  That is, after possibly
multiplying $f$ by a power of $\pi_K^d$, we may assume that $f
\in \mc{O}_K[t]$ with all roots of $f$ integral over $\mc{O}_K$, and
(since there is no branch point at $\infty$), that $d \mid
\deg(f)$.  Also, if $\deg(f) \leq 2$, then $X$ has genus $0$,
and it is trivial to find a regular model of $X$, so assume $\deg(f)
\geq 3$.

\subsection{A regular model for $X$}\label{Sfirstmodel}
Let $Y = \proj^1_K$ with coordinate $t$, and
let $X
\to Y = \proj^1_K$ to be the morphism of smooth projective $K$-curves
corresponding to the inclusion $K(t) \hookrightarrow K(t)[z]/(z^d -
f)$ with $\chara k \nmid d$, where, as in \S\ref{Sreductions}, we may assume that $f \in \mc{O}_K[t]$ is a polynomial of degree $\geq 3$ such that all roots of $f$ are integral
over $\mc{O}_K$, that $d
\mid \deg f$, and such that there does not exist $a \in \mc{O}_K$
with $v_K(\theta - a) \geq 1$ for all roots $\theta$ of $f$.  In this subsection, we will construct a normal model $\mc{Y}$ of $Y$ such
that the normalization of $\mc{Y}$ in $K(X)$ is the minimal regular
normal crossings model of $X$.

Write the irreducible factorization of $f$ as $f =
\pi_K^af_1^{a_1}\cdots f_q^{a_q}$.  We will define the model
$\mc{Y}_{\reg}$ by giving the corresponding finite set $V_{\reg}$ of Mac Lane
valuations.  Before we build $V_{\reg}$, we define certain chains of Mac Lane valuations called ``links'', ``branch point tails'', and ``tails''.

\begin{defn}\label{Dlink}
Suppose $v = [v_0,\, v_1(\phi_1) = \lambda_1,\, \ldots,\,
v_{n-1}(\phi_{n-1}) = \lambda_{n-1},\, v_n(\phi_n) = \lambda_n]$ and $v'
= [v_0,\, v_1(\phi_1) = \lambda_1,\, \ldots,\,
v_n(\phi_{n-1}) = \lambda_{n-1},\, v_n'(\phi_n) = \lambda_n']$ are two
Mac Lane valuations with $\lambda_n' > \lambda_n$.  Let $N =
e_{v_{n-1}}$.  Here $v'$ is minimally presented, but we allow the
possibility that $v = v_{n-1}$, that is, $\lambda_n =
v_{n-1}(\phi_n)$.  Assume no $D_{f_i}$ meets the intersection of the $v$-
and $v'$-components on the $\{v, v'\}$-model of $\proj^1_K$.  We
define the \emph{link} $L_{v, v'}$ as follows:

%Assume, for each $i$, that $v_{f_i}^{\infty} \succ v'$ if and only if
%$v_{f_i}^{\infty} \succ v$.
Write $f = gh$, where $g$ is the product of the $f_i^{a_i}$ such that
$v_{f_i}^{\infty} \succ v'$.  Let $e = \deg(g)/\deg(\phi_n)$ and let $s$ be
such that $v(h) = s/N$ (both $e$ and $s$ are integers by the
discussion immediately preceeding Lemma~\ref{Ldivisormultiplicities}). Let
$$\widetilde{N} = N\frac{\gcd(d,e)}{\gcd(d, e, s)}.$$  Lastly, note that the residue of $e/\gcd(d,e)$ modulo $d/\gcd(d,e)$ is a
unit, so let $r$ be any integer such that $re/\gcd(d,e) \equiv 1 \pmod{d/\gcd(d,e)}$.
Write
$$\widetilde{\lambda}_n = \frac{\gcd(d,e)}{d}\lambda_n + \frac{rs}{Nd}, \qquad
\widetilde{\lambda}'_n = \frac{\gcd(d,e)}{d}\lambda_n' +
\frac{rs}{Nd}.$$
 A \emph{link} $L_{v, v'}$ is the set of Mac Lane valuations $[v_0,\, v_1(\phi_1) = \lambda_1,\, \ldots,\,
v_n(\phi_{n-1}) = \lambda_{n-1},\, v_n(\phi_n) = \lambda]$, as
$\lambda$ ranges over the set of values such that
$$\frac{\gcd(d,e)}{d}\lambda + \frac{rs}{Nd}$$ forms the shortest
$\widetilde{N}$-path from $\widetilde{\lambda}'_n$ to
$\widetilde{\lambda}_n$, including the endpoints.
\end{defn}

  \begin{defn}\label{Dtail}\hfill
    \begin{enumerate}[\upshape (i)]
\item If $v = [v_0,\, \ldots,\, v_n(\phi_n) = \lambda_n]$, then the \emph{tail} $T_{v}$
  is the link $L_{v, v'}$, where $v'
  = [v_0,\, \ldots,\, v_n(\phi_{n-1}) = \lambda_{n-1},\, v_n(\phi_n) =
  \lambda'_n]$ with $\lambda_n' \geq \lambda_n$ minimal such that
  $\lambda_n' \in (1/\widetilde{N})\ints$ (here, if $v' = v$, we
  simply take $T_v = \{v\}$).
\item Suppose $V$ is a set of Mac Lane valuations including $v_{f_i}$
  for each irreducible non-constant factor $f_i$ of $f$.  The \emph{branch
  point tail} $B_{V, f_i}$ is the link $L_{v, v'}$, where $v
\in V$ is maximal such that $v \prec v_{f_i}^{\infty}$,
written as $$v = [v_0,\,
\ldots,\, v_{n-1}(\phi_{n-1}) = \lambda_{n-1},\, v_{n}(f_i)  =
\lambda_{n}]\footnote{Note that $v \succeq v_{f_i}$, and if $v = v_{n-1} = v_{f_i}$, then
  $\lambda_n = (\deg(f_i)/\deg(\phi_{n-1}))\lambda_{n-1}$.},$$ and $$v' = [v_0,\,
\ldots,\, v_n(\phi_{n-1}) = \lambda_{n-1},\, v_{n}(f_i) =
\lambda_{n}'],$$
where $\lambda_n' \geq \lambda_n$ is minimal such that $\lambda_n' \in
(1/\widetilde{N})\ints$ and $v'(f) = s/N + a_i \lambda_n' \in
(d/\widetilde{N})\ints$. Again, if $v' = v$, we set $B_{V,f_i} = \{v\}$.
\end{enumerate}
\end{defn}

\begin{remark}\label{Rlinkhasends}
Note that $L_{v, v'}$ includes $v$ and $v'$, and that $T_{v}$
includes $v$.
\end{remark}

\begin{remark}
Both $d$ and $f$ are implicit in the definition of links, tails, and
branch point tails, but we suppress them to lighten the notation.
\end{remark}

The algorithm below builds a regular normal crossings base for $X \to
\proj^1_K$.  The idea is to start with a tree of sorts, where the
leaves are exactly the $v_{f_i}$ (this is the content of Steps 1 and
2).  The normalization of the corresponding model of $\proj^1_K$ in
$K(X)$ may have singularities located at standard crossings,
finite cusps, and specializations of branch points from the
generic fiber.  The next steps append totally ordered sequences of
valuations (the ``links'', ``tails'', and ``branch point tails'' mentioned above) to resolve
these singularities.

\begin{algo}[cf.\ {\cite[Algorithm~3.12]{KW}}]\label{AY0}
  \hfill
  \begin{enumerate}[(1)] 
   \item Begin with the set $V_1$ of all $v_{f_i}^{\infty}$ and all of
     their predecessors (note that this includes all the $v_{f_i}$).
    \item Let $V_2$ be the \emph{inf-closure} of the set $V_1$. 
    \item (\textbf{Resolve singularities above standard crossings}) Let $S \subseteq V_2^2$ be the set of pairs $(v, w)$ of adjacent
      valuations $v \prec w$ in $V_2$.  By
      Lemma~\ref{Ladjacentstandard}, the $v$- and $w$- components form
      a standard crossing in the $V_2$-model of $\proj^1_K$.   
      Then $V_3$ is obtained from $V_2$ by replacing each subset $\{v,
      w\} \subseteq V_2$ for $(v, w)\in S$ by the link
      $L_{v, w}$.  
%      Note that we obviously have $v_{f_i}^{\infty} \succ w$ if and only if $v_{f_i}^{\infty} \succ v$.
    \item (\textbf{Resolve singularities above finite cusps})
      Let $T \subseteq V_3$ be the set of all valuations $v \in V_3$ such that
      the $v$-component of the $V_3$-model of $\proj^1_K$ has a finite
      cusp.  Then $V_4$ obtained from $V_3$ by replacing each $v \in T$
      with the tail $T_{v}$. 
    \item (\textbf{Resolve singularities above branch point specializations}) For each $i$, let $w_i$ be the
      maximal valuation in $V_4$ bounded above by $v_{f_i}^{\infty}$.  Then $V_5$ is obtained from $V_4$ by
      replacing each $w_i$ with the branch point tail $B_{V_4,
        f_i}$.
    \item Lastly, we let $V_{\reg} \subseteq V_5$ be the set of
      valuations in $V_5$ (that is, we remove all of the infinite pseudovaluations).
    \end{enumerate}
\end{algo}

\begin{example}\label{Ebasic}
  Consider the cover given by $z^5 = (t-1)^2(t^3 - \pi_K^2)$.
Write $f_1 = t-1$ and $f_2 = t^3 - \pi_K^2$.
Then $v_{f_1}^{\infty} = [v_0,\, v_1(t-1) = \infty]$ and
$v_{f_2}^{\infty} = [v_0,\, v_1(t)
= 2/3,\, v_2(f) = \infty]$.   So $V_1$ consists of $v_{f_1}^{\infty}$,
$v_{f_2}^{\infty}$, and its predecessors $v_0$ and
$v_{2/3} := [v_0,\, v_1(t) = 2/3]$.  This set is already inf-closed, so $$V_1 =
V_2 = \{v_{f_1}^{\infty}, v_{f_2}^{\infty}, v_0, v_{2/3}\}.$$

The only adjacent pair of valuations in $V_2$ is $(v_0, v_{2/3})$, so to form $V_3$, we replace this pair with the link
$L_{v_0,v_{2/3}}$.  We have $g = f$ and $h = 1$, so $N = 1$, $e = 3$, $s =
0$, $d = 5$, $\widetilde{N} = 1$, and $r = 2$.  Thus we adjoin $v_\lambda
:= [v_0,\,
v_1(t) = \lambda]$, where $\lambda$ ranges over those numbers such
that $\lambda/5$ forms the shortest $1$-path between $0$ and
$2/15$. This $1$-path is $2/15 > 1/8 > 0$, so $V_3 = V_2 \cup
\{v_{5/8}\}$. That is, $$V_3 = \{v_{f_1}^{\infty}, v_{f_2}^{\infty}, v_0, v_{5/8}, v_{2/3}\}.$$

To form $V_4$, observe that by Corollary~\ref{Cmaximalcusp}, the only valuation in $V_3$ with a finite cusp is $v_{2/3}$.  So we replace this valuation with the
tail $T_{v_{2/3}}$.  For this tail, we have $h = f$ and $g = 1$, so $N
= 1$, $e = 0$, $s = 2$, $d = 5$, $\widetilde{N} = 5$, and $r = 0$.
By definition, $T_{v_{2/3}} = L_{v_{2/3}, v_{4/5}}$, where $v_{4/5} := [v_0,\,
v_1(t)  = 4/5]$.  Thus we adjoin $v_{\lambda} := [v_0,\, v_1(t) =
\lambda]$, where $\lambda$ ranges over the shortest $5$-path from
$4/5$ to $2/3$.  This $5$-path is $4/5 > 7/10 > 2/3$, so $$V_4 =
\{v_{f_1}^{\infty}, v_{f_2}^{\infty}, v_0, v_{5/8},  v_{2/3}, v_{7/10}, v_{4/5}\}.$$

To form $V_5$, we append branch point tails $B_{V_4, f_i}$ for $i \in \{1, 2\}$.  For $i=1$, we have (in
the language of Definition~\ref{Dtail}(ii)) that $g = (t-1)^2$ and $h =
(t^3-\pi_K^2)$, so $N = 1$, $e = 2$, $d = 5$, and thus $\tilde{N} =
1$.  So $B_{V_4, f_1} = L_{v_0, v_0} = \{v_0\}$. 
For $i = 2$, observe that the valuation in $V_4$ that is maximal
among those bounded above by $v_{f_2}^{\infty}$ is $v_{2/3}$.  So we
replace this valuation with the branch point tail $B_{V_4, {f_2}}$.  For
this tail, we have $N = 3$ (since we think of $v_{2/3}$ as $[v_0,\,
v_1(t) = 2/3,\, v_2(f_2) = 2]$), and $g = f_2$ and $h = (t-1)^2$.  So $e = 1$,
$s = 0$, $d = 5$, $\widetilde{N} = 3$, and $r = 1$.  Then $B_{V_4, f_2}
= L_{v_{2/3} = w_2, w_{10/3}}$, where for $\lambda \in \rats$, we define
$w_{\lambda} := [v_0 =: w_0,\, w_1(t) = 2/3,\, w_2(f_2) =
\lambda]$. Thus we adjoin $w_{\lambda}$ where $\lambda$ ranges over
those numbers such that $\lambda/5$ forms the shortest $3$-path from
$(10/3)/5 = 2/3$ to $2/5$.  This $3$-path is $2/3 > 1/2 > 4/9 > 5/12 >
2/5$, so
$$V_5 = \{v_{f_1}^{\infty}, v_{f_2}^{\infty}, v_0, v_{5/8}, v_{2/3} = w_2, v_{7/10}, v_{4/5},
w_{10/3}, w_{5/2}, w_{20/9}, w_{25/12}\},$$
and $V_{\reg} = V_5 \setminus \{v_{f_1}^{\infty}, v_{f_2}^{\infty}\}$.

By calculating $v(f)$ for each $v \in V_{\reg}$, we see that only
for $v = v_{2/3}$ is $v(f)$ not divisible by $5$ in the value group.
So if $\mc{X}_{\reg}$ is the normalization of $\mc{Y}_{\reg}$ in
$K(X)$, then $v_{2/3}$ is the only generically ramified component in $\mc{X}_{\reg} \to
\mc{Y}_{\reg}$, and thus there is a unique component lying above $v_{2/3}$ in $\mc{X}_{\mathrm{reg}}$. There is a unique component lying above
$v_0$ and $w_{10/3}$, as they contain specializations of branch
points. By an inductive argument using the fact that a tame ramified branched cover of $\P^1_k$ has at least two distinct branch points, one can also show that the intersection of any two irreducible components is also part of the branch locus. (Note that this does not violate purity of the branch locus since $\mc{Y}_{\mathrm{reg}}$ is not regular!) 
%the cover of special fibers to the preimages of each of the irreducible components of $V_{\mathrm{reg}}$ (which are each isomorphic to $\mathbb{P}^1_k)$) and using the fact that a tame ramified cover of $\mathbb{P}^1_k$ cannot be branched at exactly one point, we see that the intersection point of any two of the components downstairs is also part of the branch locus.  
It follows that there is exactly one irreducible component
of the special fiber 
$\ol{X}_{\reg}$ of $\mc{X}_{\reg}$ above each irreducible component of the special
fiber of $\mc{Y}_{\reg}$.  The dual graph of 
$\ol{X}_{\reg}$ is depicted in Figure~\ref{FE8}.  The
self intersection number of each irreducible component of
$\ol{X}_{\reg}$ is $-2$ (other than the one corresponding to $v_0$, which is $-8$).
So $\mc{X}_{\reg}$ is actually the \emph{minimal} regular model of
$X$.  This will be reconfirmed in Example~\ref{E2}.  
\end{example}

\begin{figure}
  \begin{center}
    \setlength{\unitlength}{1mm}
\begin{picture}(100,40)

\put(17,20){\circle{2}}
\put(29,20){\circle*{2}}
\put(41,20){\circle*{2}}
\put(51,10){\circle*{2}}
\put(63,10){\circle*{2}}
\put(51,30){\circle*{2}}
\put(63,30){\circle*{2}}
\put(75,30){\circle*{2}}
\put(87,30){\circle*{2}}

\put(18,20){\line(1,0){10}}
\put(30,20){\line(1,0){10}}
\put(41,20){\line(1,1){10}}
\put(41,20){\line(1,-1){10}}
\put(52,10){\line(1,0){10}}
\put(52,30){\line(1,0){10}}
\put(64,30){\line(1,0){10}}
\put(76,30){\line(1,0){10}}

\put(15,23){$v_0$}
\put(26,23){$v_{5/8}$}
\put(36,23){$v_{2/3}$}
\put(46,33){$w_{25/12}$}
\put(58,33){$w_{20/9}$}
\put(71,33){$w_{5/2}$}
\put(82,33){$w_{10/3}$}
\put(49,13){$v_{7/10}$}
\put(61,13){$v_{4/5}$}

\put(16,14){$1$}
\put(28,14){$8$}
\put(38,14){$15$}
\put(49,24){$12$}
\put(62,24){$9$}
\put(74,24){$6$}
\put(86,24){$3$}
\put(49,4){$10$}
\put(62,4){$5$}

\end{picture}

\end{center}
\caption{The dual graph of $\mc{X}_{\reg}$ in Example~\ref{Ebasic}.
  The label below each vertex is the corresponding multiplicity in the special
  fiber and the label above each vertex is the valuation
  corresponding to the image of the component in $\mc{Y}_{\reg}$.} \label{FE8}
\end{figure}
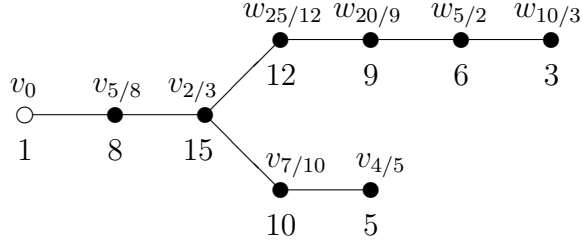

\begin{remark}
In Example~\ref{Ebasic}, suppose $K = k((s))$ and $\pi_K = s$.  If one
takes the normalization $\mc{X}$ of $\proj^1_{O_K}$ in $K(X)$, then
$\mc{O}_{X, x} \cong k[[z, t, s]]/(z^5 - t^3 + s^2)$, where $x$ is the point
above the specialization of $t = 0$ in $\proj^1_{\mc{O}_K}$ (here we
replace $z$ with $z(t-1)^{-2/5}$).  This is
the famous Du Val $E_8$-singularity, and one verifies that the
(non-$v_0$-part of the) diagram in Figure~\ref{FE8} is exactly the
Dynkin diagram for $E_8$, with the correct Cartan matrix (all
self-intersections are $-2$).
\end{remark}

\begin{lemma}\label{Linfclosed2}
  The sets $V_2$, $V_3$, $V_4$, $V_5$, and $V_{\reg}$ from
  Algorithm~\ref{AY0} above are all inf-closed.
\end{lemma}

\begin{proof}
First, $V_2$ is inf-closed by definition, and it is easy to see
from the construction that $V_3$ is as well, since links are
totally ordered.
%Furthermore, any two adjacent valuations in $V_3$ are contained in some $L_{v,w,d}$.

In Step (4), if $v \in T$ has inductive length $n$ and $w \in V_3$
satisfies $w \succ v$, then the
$n$th predecessor $w_n$ of $w$ (which is
contained in $V_3$ and satisfies $w_n \succeq v$), must be $v$.
Since any $v' \in T_v$ has inductive length $n$ as well, $\inf(v', w) = \inf(v',
w_n) = \inf(v', v) = v$.  Since
$T_{v}$ is totally ordered, and thus inf-closed, 
Lemma~\ref{Linfclosed} shows that $V_3 \cup T_{v}$ is
inf-closed, and repeating this process shows that $V_4$ is inf-closed.
%Furthermore, any two adjacent
%valuations in $V_4$ are contained in some $L_{v, w}$ or $T_{v}$.

In Step (5), for each $w_i$, if $w \in V_4$ satisfies $w \succ w_i$, then
$\inf(v', w) = w_i$ for all $v' \in B_{V_4, f_i}$ by the maximality
of $w_i$ with respect to boundedness by $v_{f_i}^{\infty}$.  Since
$B_{V_4, f_i}$ is totally ordered and thus inf-closed,
successive applications of Lemma~\ref{Linfclosed} show that $V_{\reg}
= V_4 \cup \left(\bigcup_i B_{V_4, f_i}\right)$ is inf-closed.
%All new
%pairs of adjacent valuations added in Step (5) lie in a $B_{V_4, f_i,
%  d}$, so any two adjacent valuations in $V_{\reg}$ are contained in some
%$L_{v, w, d}$, $T_{v, d}$, or $B_{V_4, f_i, d}$.

Lastly, $V_{\reg}$ is inf-closed because it is obtained from $V_5$ by
eliminating maximal elements.
\end{proof}

\begin{lemma}\label{Lpredclosed}
The set $V_{\reg}$ has the property that if $v \in V_{\reg}$, then all
predecessors of $V$ are also in $V_{\reg}$.
\end{lemma}

\begin{proof}
By Lemma~\ref{Lallpredincl}, the property in the lemma holds for
$V_2$.  It is not hard to verify from the definitions that adjoining
links, tails, and branch point tails does
not affect the property, thus it holds for $V_5$ as well.  Obviously,
removing infinite pseudovaluations does not affect the property, since they
cannot be predecessors of any other pseudovaluation, so the property holds for $V_{\reg}$.
\end{proof}

\begin{lemma}\label{LY0structure}
Let $\mc{Y}_{\reg}$ be the normal model of $\proj^1_K$ corresponding to the
set $V_{\reg}$ of Mac Lane valuations constructed in
Algorithm~\ref{AY0}, and let $\ol{Y}_{\reg}$ be its special fiber.
\begin{enumerate}[\upshape (i)]
\item The poset $V_{\reg}$ is a rooted tree with root $v_0$.
\item Every closed point of $\mc{Y}_{\reg}$ that lies on more than one component
of the special fiber $\ol{Y}_{\reg}$ lies on exactly two components, and is a standard
crossing (Definition~\ref{Dstandardcrossing}(i)).  Furthermore, the
valuations corresponding to the two components are both contained in a
single $L_{v, w}$, $T_{v}$, or $B_{V_4 , f_i}$ as in Steps (3), (4), or (5) of Algorithm~\ref{AY0}.
\item Every non-regular closed point of $\mc{Y}_{\reg}$ that lies on
  exactly one component of $\ol{Y}_{\reg}$ and is not the specialization of
  a branch point of $X \to Y$ is a finite cusp
  (Definition~\ref{Dstandardcrossing}(ii)) and the component
  corresponds to the maximal valuation of some $T_{v}$.
\item The horizontal divisor $D_{f_i}$ on 
  $\mc{Y}_{\reg}$ intersects $\ol{Y}_{\reg}$ on a single irreducible component corresponding to the maximal valuation
  of $B_{V_4, f_i}$.
  %(this is a standard endpoint so long as
  %$B_{V_4, f_i} \neq \{v_{f_i}\}$).
  %If $B_{V_4, f_i, d} =
  %v_{f_i}$, then $D_{f_i}$ specializes to a regular point of
  %$v_{f_i}$.
\item If $ i \neq j$, the horizontal divisors $D_{f_i}$ and $D_{f_j}$
  do not meet on $\mc{Y}_{\reg}$.
  \end{enumerate}
\end{lemma}

\begin{proof}
  By Lemma~\ref{Linfclosed2}, $V_{\reg}$ is inf-closed, so \cite[Corollary~2.28]{KW} shows that $V_{\reg}$, when thought of
as (the graph of) a partially ordered set, is a rooted tree. This proves (i).

By Proposition~\ref{P35}, the dual graph of $\ol{Y}_{\reg}$ is in fact the rooted
tree corresponding to $V_{\reg}$, and in particular a pair of intersecting
components of $\ol{Y}_{\reg}$ corresponds to a pair of adjacent valuations in $V_{\reg}$. To prove part (ii), we first note that
  %$V_2$ is inf-closed by definition, and it is easy to see
  %from the construction that $V_3$ is as well.  Furthermore,
  any two adjacent valuations in $V_3$ are contained in some
  $L_{v,w}$, any new pair of adjacent valuations in $V_4$ is
  contained in some $T_{v}$, and any new pair of adjacent valuations
  in $V_5$ (and thus $V_{\reg}$) is contained in some $B_{V_4, f_i}$.
Since all pairs of adjacent valuations in $V_{\reg}$ are contained in an
$L_{v,w}$, $T_{v}$, or $B_{V_4, f_i}$, and since all pairs of
adjacent valuations in an $L_{v,w}$, $T_{v}$, or $B_{V_4, f_i}$
form standard crossings by construction, part (ii) follows.  

To prove part (iv), we note that the maximal valuation $w$ in $B_{V_4,
  f_i}$ is exactly the maximal one among all valuations in $V_{\reg}$ bounded above by
$v_{f_i}^{\infty}$.  In particular, we have $v_{f_i} \preceq w \prec
v_{f_i}^{\infty}$.  By Proposition~\ref{Pbranchspecialization}, $D_{f_i}$ specializes only to the
$w$-component of $\mc{Y}_{\reg}$.

%Then $w$ can be written as
%$[v_{f_i}, w(f_i) = \lambda]$,
%
%
%
%If $B_{V_4, f_i, d} \neq
%\{v_{f_i}\}$, then $w \neq v_{f_i}$, and by definition the point where
%$D_{f_i}$ intersects the special fiber is a standard endpoint of $w$.
%If $B_{V_4, f_i,d} = \{v_{f_i}\}$, then $w = v_{f_i}$ and the proof of
%(iii) is done.

We now prove part (iii).  By \cite[Lemma~7.3]{ObusWewers}, a
non-regular closed point $y$ of $\mc{Y}_{\reg}$ that lies on one
irreducible component of $\ol{Y}_{\reg}$ is either a finite cusp or the specialization of $t = \infty$.  Since $v_0$ is the
unique minimal valuation in $V_{\reg}$, the point $t = \infty$ specializes
to the component of $\ol{Y}_{\reg}$ corresponding to $v_0$, and the
specialization is thus
regular by \cite[Lemma~7.3(ii)]{ObusWewers}, taking $\lambda_1 = 0$ in
that lemma.  So $y$ is a finite cusp.

Suppose $y$ lies on the $w$-component of $\mc{Y}_{\reg}$ for some valuation $w$.
%By Lemma~\ref{Lendpointmaximal}, $w$ is maximal in $V_{\reg}$ for its
%inductive length.
The construction of
Algorithm~\ref{AY0} starting from step (4) shows that $w$ is maximal either in a $T_{v}$
or a $B_{V_4, f_i}$, and that the only way $w$ is not maximal in a
$T_{v}$ is if $w$ is maximal in some $B_{V_4, f_i}$ with $B_{V_4,
  f_i} \neq \{v_{f_i}\}$.  But in this case, $y$ meets $D_{f_i}$ by
Lemma~\ref{Lstandardendpointunique}, so $y$ is the
specialization of a branch point, contradicting the assumption in part
(iii).  This proves (iii).

Lastly, since $v_{f_i}^{\infty}$ and $v_{f_j}^{\infty}$ are
non-comparable in the partial order, they are not neighbors, so
Proposition~\ref{P35} shows they do not meet on $\mc{Y}_{\reg}$
(recall that $V_{\reg}$ is inf-closed).  This proves part (v).
\end{proof}

\begin{theorem}\label{TX0regular}
Let $\mc{Y}_{\reg}$ be the normal model of $\proj^1_K$ corresponding to the
set $V_{\reg}$ of Mac Lane valuations constructed in Algorithm~\ref{AY0},
and let $\nu \colon \mc{X}_{\reg} \to \mc{Y}_{\reg}$ be the normalization of $\mc{Y}_{\reg}$ in $K(X)$.  Then
$\mc{X}_{\reg}$ is a regular model of $X$ with normal crossings.  In
other words, $V_{\reg}$ is a regular normal crossings base. In fact,
$\mc{X}_{\reg}$ even has \emph{strict} normal crossings (that is, all the
irreducible components of the reduced special fiber are smooth). 
\end{theorem}

\begin{proof}
We go systematically through all closed points $y \in \mc{Y}_{\reg}$ and
show that each point $x \in \nu^{-1}(y)$ is regular in $\mc{X}_{\reg}$ with
normal crossings, and furthermore that if $y$ lies on only one irreducible
component of the special fiber of $\mc{Y}_{\reg}$, then $x$ is a smooth point of the reduced special fiber.

If $y$ is the intersection of some $D_{f_i}$ with $\ol{Y}_{\reg}$, then by
Lemma~\ref{LY0structure}(iv), $y$ specializes only to the
$v$-component of $\mc{Y}_{\reg}$, where $v = [v_{f_i}, v(f_i) =
\lambda]$ is the maximal valuation in the
branch point tail $B_{V_4,
  f_i}$ (we allow the possibility that $\lambda =
v_{f_i}(f_i)$, thus making the presentation of $v$ non-minimal).  In
particular, $\lambda \in (1/\widetilde{N})\ints$ and $s/N + a_i
\lambda \in (d/\widetilde{N})\ints$, where $s$, $N$, and
$\widetilde{N}$ are defined as in the link corresponding to $B_{V_4,
  f_i}$ as in Definition~\ref{Dtail}(ii).  Then all points $x \in \nu^{-1}(y)$ are regular
with normal crossings by
Proposition~\ref{Pstandardendpointregular}(iii) ($f_i$ and $\lambda$ here play the
roles of $\phi_n$ and $\lambda_n$ in that proposition).  Since the
horizontal part of $\divi_0(f)$ is $\sum a_i D_{f_i}$, condition (b) of
Corollary~\ref{Csmoothoncomponent} applies to $x$, and hence any such $x$ is a smooth point of the reduced special fiber. 

For the remainder of
the proof, assume that $y$ is not the specialization of a branch point
of $X \to Y$.
If $y$ lies on more than one irreducible component of $\ol{Y}_{\reg}$, then
by Lemma~\ref{LY0structure}(ii), $y$ is a standard crossing
corresponding to two adjacent valuations in
some $L_{v, w}$, $B_{V_4, f_i}$, or $T_{v}$.  By
Proposition~\ref{Pstandardcrossingregular}, any $x \in
\nu^{-1}(y)$ is regular in $\mc{X}_{\reg}$ with normal crossings.
%Proposition~\ref{Pstandardcrossingregular} also shows in the third case that any
%$x \in \nu^{-1}(y)$ is also regular in $\mc{X}_{\reg}$ with normal
%crossings, once one notes that one can take $e = r = 0$ and $s =
%Nv(f)$ in Proposition~\ref{Pstandardcrossingregular}.

If $y$ is a non-regular point lying on one irreducible
component of $\ol{Y}_{\reg}$, then by Lemma~\ref{LY0structure}(iii), $y$ is a
finite cusp on the $w$-component of $\ol{Y}_{\reg}$, where $w$ is
maximal in some $T_v$.  Specifically, $w = [v_0,\, \ldots,\,
w_n(\phi_n) = \lambda_n]$ such that $\lambda_n \in
(1/\widetilde{N})\ints$, where $\widetilde{N}$ is defined as for the
link corresponding to $T_v$ in
Definition~\ref{Dtail}(i).  By
Proposition~\ref{Pstandardendpointregular}(iii) (with $a = 0$ in that
proposition) combined with Remark~\ref{Rnobranchpoint}, all $x \in
\nu^{-1}(y)$ are regular in $\mc{X}_{\reg}$ with normal crossings.  Condition (b) of
Corollary~\ref{Csmoothoncomponent} applies to $x$, and hence any such $x$ is a smooth point of the reduced special fiber.
%By
%Corollary~\ref{Csmoothoncomponent}, $\mc{X}_{\reg}$ has normal crossings at
%each $x$ as well. \andrew{Why does this corollary apply here?}

Lastly, suppose $y$ lies on only one irreducible
component $\ol{Z}$ of the special fiber $\ol{Y}_{\reg}$, is regular in $\mc{Y}_{\reg}$, and is not the
specialization of a branch point of $X \to Y$.  The reduced induced
subscheme of $\ol{Z}$ is isomorphic to $\proj^1_k$ by \cite[Lemma
7.1]{ObusWewers}, so in particular, $\mc{Y}_{\reg}$ has normal
crossings at $y$.  Since regularity can be checked after completion by \cite[Proposition~11.24]{AM}, all $x \in \nu^{-1}(y)$ are
regular in $\mc{X}_{\reg}$ with normal crossings by
Proposition~\ref{P:RegandNormalize}(i). Condition (b) of
Corollary~\ref{Csmoothoncomponent} applies to $x$, and hence any such $x$ is a smooth point of the reduced special fiber.
\end{proof}

\section{The minimal regular model with normal crossings}\label{Ssecondmodel}
Throughout \S\ref{Ssecondmodel}, we let $\mc{Y}_{\reg}$ be the
$V_{\reg}$-model of $\proj^1_K$, where $V_{\reg}$ is constructed in
Algorithm~\ref{AY0}, and we let $\nu \colon \mc{X}_{\reg} \to \mc{Y}_{\reg}$ be its
normalization in $K(X)$.  By
Theorem~\ref{TX0regular}, $\mc{X}_{\reg}$ is a regular normal crossings model of
$\mc{X}$.  In the language of
Definition~\ref{Dregularbase}, 
$V_{\reg}$ is a regular normal crossings base.  In this section, we will describe which irreducible
components of $\mc{X}_{\reg}$ need to be contracted to obtain the
minimal regular normal crossings model.  
 Equivalently, we will show
which valuations in $V_{\reg}$ are removable
(Definition~\ref{Dregularbase}).  After an important
preliminary lemma in \S\ref{Sgeneralities}, we will show that
all such removable valuations are either maximal valuations in $V_{\reg}$ (\S\ref{Smaximal})
or minimal valuations in $V_{\reg}$ (\S\ref{Sminimal}).  The main
result is Theorem~\ref{Tminnormalcrossingsbase}.

\subsubsection{A weak minimality condition on $f$}
As in \S\ref{Sconstruction}, we assume $f = \pi_k^af_1^{a_1} \cdots f_r^{a_r}$ is an irreducible factorization of
$f$ with all $f_i$ monic in $\mc{O}_K[t]$, and $X \to Y = \proj^1_K$
is a $\ints/d$-cover of smooth projective curves given birationally by $z^d = f$.
Recall in \S\ref{Sreductions}, we showed we may assume that $d \mid \deg (f)$ and $\deg(f) \geq 3$. We now add another assumption in \S\ref{Ssecondmodel} without loss of generality.
Namely, suppose there exists $a \in \mc{O}_K$ such that each root
$\theta$ of $f$ satisfies $v_K(\theta - a) \geq 1$.  Then, letting $b
= \lfloor \min_{\theta} v_K(\theta - a) \rfloor$ and replacing $t$
with $a + \pi_K^bt$ guarantees that there no longer exists $a \in \mc{O}_K$
as above,
%\footnote{This should be thought of as a weak ``minimality'' condition on $f$.} 
while still preserving the fact that all roots of $f$ are
integral over $\mc{O}_K$.  So we assume no $a$ exists as above.

\subsection{Generalities}\label{Sgeneralities}

We begin with a discussion of regular normal crossings bases
associated to \emph{minimal} regular models of $X$.

%\andrew{Moved this here from \S\ref{Sconstruction}}
\begin{prop}\label{Pcontractdownstairs}\hfill
  \begin{enumerate}[\upshape (i)]
    \item\label{P:CanQt} There exists a regular normal crossings base $V_{\min}$ for $X \to
      \proj^1$ such that the corresponding model $\mc{X}_{\min}$ of
      $X$ is the minimal regular model with normal crossings.
    \item\label{P:ContractGalOrbit}  If $V_{\reg}$ is a regular normal crossings base, then there is a chain $V_{\reg} =: V_0 \supsetneq V_1 \supsetneq \cdots
      \supsetneq V_n := V_{\min}$ where, for $0 \leq i < n$, there
      exists $v_i \in V_i$ such that $v_i$ is removable from $V_i$ and
      $V_{i+1} = V_i \setminus \{v_i\}$.
   \end{enumerate}
\end{prop}

\begin{proof}
  %\padma{New} In what follows, we will freely use the identification between normal models of $\mathbb{P}^1_K$ and collections of MacLane valuations provided by \cite[Corollary 3.18]{Ruth}, and in particular we will call a normal model of $\mathbb{P}^1_K$ a regular normal crossings base if the corresponding set of MacLane valuations is.
Part~\ref{P:CanQt} follows from
Proposition~\ref{P:ExtendingCoversToCanonicalModels}. We now prove
part~\ref{P:ContractGalOrbit}. If $\mc{X}$ is the normalization of a
normal model $\mc{Y}$ of $\P^1_K$, then the action of the Galois group
$G$ of the cover $X \rightarrow \P^1_K$ extends to $\mc{X}$, and
$\mc{Y} = \mc{X}/G$. 
%Observe that if $\mc{X}$ is the normalization of a normal model $\mc{Y}$ of $\P^1_K$, and if $\mc{X}'$ is a model obtained by contracting the entire Galois orbit of a curve $E$ on $\mc{X}$, then $\mc{X}'$ is the normalization of a new normal model $\mc{Y}'$ of $\P^1_K$, which is obtained from $\mc{Y}$ by contracting the image of $E$. 
Say that a $-1$ curve on $\mc{X}$ is special if contracting $E$ on $\mc{X}$ produces a new regular normal crossings model of $X$. 
%Therefore, for proving part~\ref{P:ContractGalOrbit}, by Remark~\ref{R:SNCRealm}, it suffices to 
We first show that if $\mc{X}$ is a regular normal crossings model
obtained from a regular normal crossings base and if $E$ is a special $-1$ curve on $\mc{X}$, then contracting the entire $G$-orbit of $E$ produces a new regular normal crossings model $\mc{X}'$ of $X$.

Since the $G$-action preserves intersection numbers, it follows
that if $E$ is a special $-1$ curve, so is every curve in its $G$-orbit. If the curves in the $G$-orbit of $E$ are pairwise disjoint,
then since being normal crossings is a local property, it follows that
the entire $G$-orbit of $E$ can be contracted to produce a normal
crossings regular model of $X$. We now argue that two curves in the
$G$-orbit of $E$ cannot intersect. Assume that there are two
intersecting special $-1$ curves $E_1,E_2$ in the $G$-orbit of
$E$. Let the common image of $E_1,E_2$ in $\mc{Y}$ be the
component $\Gamma$. We only need to consider the case where the
special fiber of $\mc{Y}$ has at least $2$ components, since $\mc{Y}$
is already minimal otherwise. Let
$\Gamma'$ be a component of the special fiber of $\mc{Y}$ that
intersects $\Gamma$, and let $F'$ be an irreudicble component of the preimage of $\Gamma'$ in
$\mc{X}$ that intersects $E_1$. Since $E_2$ and $F'$ are both
neighbors of $E_1$, the sum of the multiplicities of the components
intersecting $E_1$ is strictly larger than the multiplicity of
$E_2$. But since $E_1$ is a $-1$ curve, this sum is also supposed to
equal the multiplicity of $E_1$.
This contradicts the fact that the multiplicities of $E_1$ and $E_2$ in the special fiber are equal (by virtue of being in the same $G$-orbit).

% We need to show that both $E_1$ and $E_2$ can be contracted to produce a regular normal crossings model of $X$. If $\Gamma_1$ is a terminal component, then contracting $E_1$ keeps $E_2$ a special $-1$ curve    \padma{I'm confused again about how we finished the argument in this case. Sorry for poor notetaking... I also seem to remember we said we can't have both $\Gamma_1$ and $\Gamma_2$ both be non-terminal, because otherwise the image of $z$ becomes singular or something but I don't see why that's a problem anymore... what was a $-1$ curve before contraction need not be a $-1$ curve anymore after we contract, right? or can that not happen?} 

Finally, let $\mc{Y}_0$ be the model of $\proj^1_K$ corresponding to the
regular normal crossings base $V_{\reg} := V_0$, and let $\mc{X}_0$ be its
normalization in $K(X)$.  Suppose $E$ is a special $-1$
curve on $\mc{X}_0$ and $\Gamma$ is its image in $\mc{Y}_0$.
Let $\mc{X}_0 \to \mc{X}_1$ be the contraction of the entire $G$-orbit
of $E$ in $\mc{X}_0$, and let $\mc{Y}_0 \to \mc{Y}_1$ be the
contraction of $\Gamma$ in $\mc{Y}_0$.  Then $\mc{X}_1$ is the normalization of
$\mc{Y}_1$ in $K(X)$, and as we have seen, $\mc{X}_1$ is regular with
normal crossings.  If $v_0 \in V_0$ is the valuation corresponding to
$\Gamma$, then this shows that $v_0$ is removable from $V_0$.
We now iterate this
procedure and use Remark~\ref{R:SNCRealm}
to finish the proof.
\end{proof}

\begin{remark}\label{Rspecialcontraction}
In particular, the proof above shows that there exists a special
$-1$-curve $E$ on $\mc{X}$ lying above the $v$-component of $V$ if and
only if $v$
is removable from $V$.
\end{remark}

We say that $V$ is
a \emph{minimal} regular normal crossings base if $V$ is a regular
normal crossings base with no removable valuations.  In light of
Proposition~\ref{Pcontractdownstairs}, there is a unique minimal
regular normal crossings base $V_{\min}$ and the normalization of a model of
$\proj^1_K$ corresponding to a minimal regular normal crossings base
is the minimal regular model of $X$ with normal crossings.

The following lemma is useful for  
showing that certain valuations in a regular normal crossings base are not removable. This will allow us to show that after possibly removing certain maximal valuations and certain minimal valuations in $V_{\mathrm{reg}}$, there are no further removable valuations. 
%\begin{lemma}\label{Lthreecomponents}
%Suppose that $V$ is a regular normal crossings base and there exist four
%distinct valuations $v, w_1, w_2, w_3 \in V$ such that $w_3 \prec v = \inf
%(w_1, w_2)$.  Then $v$ is not removable from $V$.
%\end{lemma}
%
%\begin{proof}%
%We see that $v$ has three distinct adjacent valuations $r_1,
%r_2, r_3$ such that $v \prec r_i \prec w$ $(i \in \{1, 2\}$ and $v
%\succ r_3 \succ w_3$.
%So contracting the
%irreducible component corresponding to $v$ in the model $\mc{Y}_0$ of
%$\proj^1_K$ corresponding
%to $V_0$ results a model $\mc{Y}_0'$ where at least three irreducible components of the
%special fiber (those corresponding to $r_1$, $r_2$, and $r_3$) meet at one point, which means the same is true on
%the normalization $\mc{X}_0'$ of $\mc{Y}_0$ in $K(X)$.  In other
%words, $V_0 \setminus \{v\}$ is not a regular normal crossings base.
%\end{proof}

\begin{lemma}\label{Lramificationnotremovable}
  Suppose $V$ is a regular normal crossings base, and let $v \in V$.
  Let $\mc{Y}$ be the $V$-model of $\proj^1$, and let $\mc{X} \to
  \mc{Y}$ be the normalization of $\mc{Y}$ in $K(X)$.  Let $\ol{Y}$ be
  the special fiber of $\mc{Y}$.
  Suppose that any one of the following is true:
  \begin{enumerate}[\upshape (i)]
  \item The $v$-component of $\ol{Y}$ intersects at least three other
    irreducible components of $\ol{Y}$.
  \item The $v$-component of $\ol{Y}$ intersects two other irreducible components
    of $\ol{Y}$ and
  there exists a point lying only on the $v$-component of $\ol{Y}$
  that is geometrically ramified in $\mc{X} \to \mc{Y}$.
 \item The $v$-component of $\ol{Y}$ intersects one other irreducible
   component of $\ol{Y}$ and
  there exist two points lying only on the $v$-component of $\mc{Y}$ that are geometrically ramified in $\mc{X} \to \mc{Y}$,
  with at least one of the geometric ramification indices
  strictly greater than $2$.
  \item There exist three points lying on the $v$-component of $\mc{Y}$
  that are geometrically ramified in $\mc{X} \to \mc{Y}$.
  \end{enumerate}
Then $v$ is not removable from $V$.
\end{lemma}

\begin{proof}
Let $\ol{Z}_v$ be the $v$-component of the $V$-model $\mc{Y}$ of $\proj^1_K$. Let
$\ol{W}$ be an irreducible component of the special fiber of $\mc{X}$
above $\ol{Z}_v$.

In case (i), contracting $\ol{Z}_v$ results in
a model where at least three irreducible components of the special
fiber meet at one point, which means the same is true when contracting
$\ol{W}$, which violates normal crossings.  

Now, note that if the cyclic cover $\ol{W}^{\red} \to \ol{Z}_v^{\red}$ is ramified above
at least three points, then the arithmetic genus of
$\ol{W}^{\red}$ is positive.  This means that contracting
$\ol{W}$ results in $\mc{X}$ no longer being regular, so $v$ is
not removable from $V$.  This takes care of case (iv), and allows us
to assume in cases (ii) and (iii) that at least one of the points of
$\ol{Z}_v$ intersecting another component of the special fiber of $\mc{Y}$ is not geometrically ramified.

So in case (ii), let $y$ and $y'$ be the points where $\ol{Z}_v$
intersects the rest of $\ol{Y}$, and let $y''$ be a geometrically ramified point lying only on $\ol{Z}_v$.  Assume, say, that $y$ is not geometrically
ramified.  This means that $$\# (\nu^{-1}(y) \cap \ol{W}) > \#
(\nu^{-1}(y'') \cap \ol{W}).$$ In particular, $\# (\nu^{-1}(y) \cap
\ol{W}) \geq 2$, which means that contracting $\ol{W}$ results in at least three local irreducible components of the special fiber of
$\mc{X}$ (at least two intersecting $\ol{W}$ above $y$ and one
intersecting $\ol{W}$ above $y'$) meeting at a point.  Thus the resulting
model does not have normal crossings, which means that $v$ is not
removable from $V$.

In case (iii), let $y$ be the point of $\ol{Z}_v$ intersecting the rest
of $\ol{Y}$, and assume $y$ is not
geometrically ramified.  By assumption, the degree of $\ol{W}^{\red}
\to \ol{Z}_v^{\red}$ is at least 3.  So $\# (\nu^{-1}(y) \cap
\ol{W}) \geq 3$, which means that contracting $\ol{W}$ results in at least three local irreducible components of the special fiber of
$\mc{X}$ meeting at a point.  As in the previous paragraph, 
$v$ is not removable from $V$.  
\end{proof}

%\begin{lemma}\label{Lregularbranchpointspecialization}
%  Let $V$ be a set of Mac Lane valuations including $v_{f_i}$ for some
%  $i$ such that $V$ contains no valuation $w$ satisfying $v_{f_i}
%  \prec w \prec v_{f_i}^{\infty}$, and let $\mc{Y}$
%  be the $V$-model of $\proj^1$.  Assume further that the horizontal
%  part of $\divi_0(f_i)$ does not meet the horizontal part of
%  $\divi_0(f_j)$ on $\mc{Y}$ for any $j \neq i$.  Furthermore, suppose
%  the normalization $\mc{X}$ of $\mc{Y}$ in $K(X)$ is regular.
%
%  The horizontal part of $\divi_0(f_i)$ intersects the special
%      fiber of $\mc{Y}$ at a regular point of $\mc{Y}$ lying on the $v$-component.
%  \end{lemma}

%\begin{proof}
%  Let $y \in \mc{Y}$ be the point on the special fiber of $\mc{Y}$
%  which the horizontal part $D_{f_i}$ of $\divi_0(f_i)$ intersects.
%  Since $\mc{Y}$ is regular at $y$ and $\mc{X}$ is regular above $y$,
%  \andrew{Lemma from section 3} shows that the branch locus of $\mc{X}
%  \to \mc{Y}$ is regular at $y$.  Since this branch locus contains
%  $D_{f_i}$ which intersects the $v$-component of $\mc{Y}$, the branch
%  locus cannot also contain the $v$-component.  So $v$ is unramified
%  in $\mc{X} \to \mc{Y}$, which means $\tilde{e}_v = e_v$.   
%\end{proof}

We also state a partial converse to
Lemma~\ref{Lramificationnotremovable}(iii) after recalling Castelnuovo's contractibility criterion.

\begin{lemma}\label{LCastelnuovo}
 Let $\mc{X}$ is a regular normal crossings arithmetic surface. Let
 $\Gamma$ be a multiplicity $m$ component of the special fiber, and
 let $\mc{X} \rightarrow \mc{X}'$ be the contraction of $\Gamma$. If
 $\Gamma$ is not isomorphic to $\P^1_k$, then $\mc{X}'$ is not
 regular. 
 %\padma{Are we allowing nodal $\P^1_k$ in our special fiber?}
 %\andrew{Yes.  That's why we don't say \emph{strict} normal crossings} 
 Furthermore, if $\Gamma$ intersects exactly two (resp.\ one) other
 components of the special fiber having multiplicities $m_1,m_2$
 (resp.\ $m'$), then $\mc{X}'$ is also regular normal crossings if and
 only if $\Gamma \cong \P^1_k$ and $m = m_1+m_2$ (resp.\ $m = m'$).
\end{lemma}
\begin{proof}
By \cite[Proposition~9.1.21]{LiuBook}, the self-intersection number of
$\Gamma$ is $-(m_1 + m_2)/m$ (resp.\ $-m'/m$).  By Castelnuovo's criterion, $\mc{X}'$
is regular if and only if $\Gamma \cong \proj^1_k$ and $m = m_1 +
m_2$ (resp.\ $m = m'$).  In this case $\mc{X}'$ is normal crossings as well by \cite[Lemma~9.3.35]{LiuBook}.
\end{proof}

\begin{lemma}\label{Lremovable}
Maintain the notation of Lemma~\ref{Lramificationnotremovable}.
Suppose the $v$-component of $\mc{Y}$ intersects exactly one other
irreducible component (say the $w$-component) of $\mc{Y}$. Suppose further that there are exactly two points lying only on the $v$-component of $\mc{Y}$ that are geometrically
ramified, that these geometric ramification indices both equal $2$, 
and that the points above the geometrically ramified points are smooth points of the reduced special fiber.
%\padma{adjust to include the case where the upstairs component is contractible and has just one neighbor? }\andrew{I think the new version is what we want}
Then
\begin{enumerate}[\upshape (i)]
\item If $\ol{X}_v$ is an irreducible component of $\ol{X}$ above the
  $v$-component, then $\ol{X}_v$ meets the rest of $\ol{X}$ at exactly two points. 
\item The $v$-component is removable from $V$ if and only if
  $\tilde{e}_v = 2\tilde{e}_w$, where
  \begin{equation}\label{Etildes}
  \tilde{e}_v \colonequals
\frac{e_vd}{\gcd(d, e_vv(f))} \quad \text{ and } \quad \tilde{e}_w \colonequals
\frac{e_wd}{\gcd(d, e_ww(f))}.
\end{equation}
\item If the $v$-component is removable from $V$, then the
  $w$-component is not removable from $V \setminus \{v\}$.
\end{enumerate}
 \end{lemma}

\begin{proof}
The curve $\ol{X}_v^{\red}$ is smooth at the point where it meets the components
 above the $w$-component, since non-smoothness of $\ol{X}_v^{\red}$
 here would 
 %mean there would be at least $3$ local irreducible components on the special fiber of $\mc{X}$ at this point,
 contradict the assumption that $\mc{X}$ has normal
 crossings.  An unramified cover of a (local) smooth curve
 is smooth, so the only places where $\ol{X}_v^{\red}$ could be
 non-smooth are above the
 geometrically ramified points lying only on the $v$-component.  By assumption $\ol{X}_v^{\red}$ is
 smooth at these points, so
 $\ol{X}_v^{\red}$ is smooth.

 Let $\ol{Y}_v$ be the $v$-component of $\mc{Y}$. We first argue that
 the point where $\overline{Y}_v$ meets the rest of $\overline{\mc{Y}}
 \equalscolon \overline{Y}$ is not geometrically ramified.
By assumption, $\ol{X}_v^{\red} \to
 \ol{Y}_v^{\red} \cong \proj^1_k$ is a cyclic cover with at most $3$
 branch points, two of which have geometric ramification index
 $2$.  
 %\sout{two branch points of index  $2$, and is unbranched outside of these two points except possibly  above the point where $\ol{Y}_v$ meets the rest of $\ol{Y}$.}  
Since $\ol{X}_v^{\red}$ is smooth, the quotient cover $\ol{X}_v^{\red}/(\ints/2) \to \proj^1_k$
 is a tame cover of smooth
 projective curves branched at at most one point, which implies, e.g.,
 by the Riemann--Hurwitz formula, that
 it is an isomorphism.  So
 $\ol{X}_v^{\red} \to \ol{Y}_v^{\red}$ is a $\ints/2$-cover, and again by the
 Riemann--Hurwitz formula and the fact that the genus of
 $\ol{X}_v^{\red}$ is an integer, such a cover cannot be branched at 3 points.
Thus $\ol{X}_v^{\red} \to \ol{Y}_v^{\red}$ is a $\ints/2$-cover of genus
 zero curves, which means that $\ol{X}_v$ meets
 the rest of $\ol{X}$ at two points, proving part (i).

 Now, the multiplicities of the irreducible
components of the special fiber $\ol{X}$ of $\mc{X}$ above the $v$- and
$w$-components are $\tilde{e}_v$ and $\tilde{e}_w$ respectively. By
Lemma~\ref{LCastelnuovo}, $\ol{X}_v$ can be contracted while preserving regularity
 with normal crossings if $\tilde{e}_v = 2\tilde{e}_w$.  By
 Remark~\ref{Rspecialcontraction}, this is equivalent to $v$ being 
 removable from $V$, proving part (ii).

Let $\ol{X}_w$ be an irreducible component of the special fiber of
 $\mc{X}$ above the $w$-component meeting $\ol{X}_v$.  By part (i),
 $\ol{X}_v$ intersects the rest of $\ol{X}$ at two points. So after contracting all the components above the $v$-component, either the image of
 $\ol{X}_w$ either intersects itself, in which case it is not
 contractible by Lemma~\ref{LCastelnuovo}, or it intersects another component lying above the
 $w$-component.  Such a component has the same multiplicity as
 $\ol{X}_w$ in the special fiber, and $\ol{X}_w$ also intersects some other
 component \emph{not} lying above the $w$-component.  By
 Lemma~\ref{LCastelnuovo}, contracting $\ol{X}_w$ does not give a
 regular normal crossings model. By Remark~\ref{Rspecialcontraction},
 $w$ is not removable from $V \setminus \{v\}$, proving part (iii).
\end{proof}

\subsection{Contractions of maximal components}\label{Smaximal}

Let $V_1 \subseteq V_2 \subseteq V_3 \subseteq V_4 \subseteq V_5
\supseteq V_{\reg}$ be as in Algorithm~\ref{AY0}.  The main result of
\S\ref{Smaximal} is Proposition~\ref{Pv1max}, which describes exactly
which valuations are removable from $V_{\reg} \setminus \{v_0\}$.

Recall from Definitions~\ref{Dlink}, \ref{Dtail} that the set of valuations in a link/tail/branch-point tail is totally ordered.
\begin{lemma}\label{Lmidofchain}
 If $v \in V_{\mathrm{reg}}$ is a non-maximal and non-minimal component of a link, or tail, or a branch point tail, then $v$ is not removable.
\end{lemma}
\begin{proof} Let $C$ be the totally-ordered set of valuations corresponding to a link/tail/branch-point tail containing $v$. Since $v$ is non-maximal and non-minimal, by Definition~\ref{Dlink}, $v$ has exactly two neighbors $v_1,v_2$ which are also in $C$. 
%Note that $v_1,v_2$ are not neighbors in the $\widetilde{N}$-path corresponding to $L$. 
Furthermore, if $\mc{Y}'$ is the $V_{\mathrm{reg}} \setminus \{v\}$-model, then the irreducible components corresponding to $v_1,v_2$ intersect at a point $y$ in $\mc{Y}'$, and $y$ is non-regular on $\mc{Y}'$ by the $\widetilde{N}$-path criterion of 
Proposition~\ref{Pstandardcrossingregular} and Definition~\ref{Dlink}. In other words, $v$ is not removable. 
\end{proof}

\begin{prop}\label{Ponlyspecialcontract}
If $v \in V_{\reg} \setminus V_2$, then $v$ is not removable from $V_{\reg}$.
\end{prop}

\begin{proof}
%Throughout, let $\mc{Y}_{\reg}$ be the $V_{\reg}$-model of $\proj^1_K$, and let
%$\mc{X}_{\reg}$ be the normalization of $\mc{Y}_{\reg}$ in $K(X)$.

%First observe that $V_{\reg} \setminus
%\{v\}$ is not a regular normal crossings base for any valuation $v$
%inserted in step (2) of Algorithm~\ref{AY0}, since such a valuation is
%the inf of two other valuations in $V_0$ and thus
%Lemma~\ref{Lthreecomponents} applies.

We must show that $V_{\reg} \setminus \{v\}$ is not a regular normal
crossings base for any valuation $v$ in $V_3 \setminus V_2$, $V_4
\setminus V_3$, or $V_5 \setminus V_4$.   
If $v$ in $V_3 \setminus V_2$, then by Remark~\ref{Rlinkhasends}, $v$ is a non-maximal and non-minimal element of a link and Lemma~\ref{Lmidofchain} shows that $v$ is not removable. If $v$ is in $V_4
\setminus V_3$, or $V_5 \setminus V_4$, then $v$ is a non-minimal element of a tail or a branch point tail respectively, and once again Lemma~\ref{Lmidofchain} shows $v$ is not removable if $v$ is also non-maximal. So without loss of generality, assume that $v$ is a maximal element of a branch point tail $B$ or tail $T_w$.

\begin{comment}
$e_v \mid e_w$ by Lemma~\ref{Lv2onlymax}, and $e_w \mid \tilde{e}_w$
%Then by Definition~\ref{Dlink}, $L$ is a totally-ordered chain of valuations in which $v$ has exactly two neighbors $v_1,v_2$ which are also in $L$. 
%Note that $v_1,v_2$ are not neighbors in the $\widetilde{N}$-path corresponding to $L$. 
%Furthermore, if $\mc{Y}' \colonequals V_{\mathrm{reg}} \setminus \{v\}$-model, then the irreducible components corresponding to $v_1,v_2$ intersect at a point $y$ in $\mc{Y}'$, and $y$ is non-regular on $\mc{Y}'$ by the $\widetilde{N}$-path criterion of Proposition~\ref{Pstandardcrossingregular} and the definition of a link. In other words, $v$ is not removable. 

%\padma{needs remembering what the key thing about being a link is, so
%spelt out instead}
\sout{
,then $L \setminus \{v\}$ is no longer a
link, as follows directly from the ``shortest path''
criterion of Definition~\ref{Dlink}.  This means that the two
valuations in $L$ adjacent to $v$ no longer form an
$\widetilde{N}$-path, which by
Proposition~\ref{Pstandardcrossingregular} shows that a singularity
appears above the point in the model corresponding to $V_{\reg} \setminus
\{v\}$ where the components corresponding to these two valuations
cross.  That is, $V_{\reg} \setminus \{v\}$ is not a regular normal
crossings base.}

\end{comment}

%Next suppose $v \in V_5 \setminus V_4$ is a maximal element of some branch point tail $B$ as in step (5) of Algorithm~\ref{AY0}. 
%If $v$ is non-maximal in $B$, the same argument as in the above paragraph shows that $v$ is not removable. 
First suppose $v$ is maximal in a branch point tail $B$ as in step (5) of Algorithm~\ref{AY0}. Then the roots of some $f_i$
  specialize to the $v$-component of the special fiber of
  $\mc{Y}_{\reg}$, and $v$ satisfies the condition of
  Proposition~\ref{Pstandardendpointregular}(iii) (here $f_i$ plays the
  role of $\phi_n$ in Proposition~\ref{Pstandardendpointregular}).  If we replace $\mc{Y}_{\reg}$ with the model
  $\mc{Y}'$ of $\proj^1_K$ corresponding to $V_{\reg} \setminus \{v\}$, then the roots of
  $f_i$ specialize to the $v'$-component where $v'$ is the
  adjacent valuation $v$, which by Definition~\ref{Dtail}(ii) of a branch
  point tail no longer
  satisfies the criterion of Proposition~\ref{Pstandardendpointregular}(iii).
  So the points above the specialization of the roots of $f_i$ to
  the special fiber of $\mc{Y} '$ are not regular, and thus $V_{\reg} \setminus
  \{v\}$ is not a regular normal crossings base.

Now suppose $v$ is a maximal element of some tail $T_w$ as in step (4) of Algorithm~\ref{AY0}, with $w = [w_0 = v_0,\,
\ldots,\, w_n(\phi_n) = \lambda_n]$. 
%The same argument as before shows that $v$ is non-removable if $v$ is non-maximal in $T$.  
Since $v$ is maximal, we
again consider the adjacent valuation $v'$ as in the previous
paragraph, which by Definition~\ref{Dtail}(i) of a tail no longer
satifies the criterion of 
Proposition~\ref{Pstandardendpointregular}(iii) (with $a = 0$ in that
proposition).  So by Proposition~\ref{Pstandardendpointregular}(iii), the points above the
intersection of $D_{\phi_n}$ with the $v'$-component are not regular.
\end{proof}

\begin{lemma}\label{Lfincuspram}
 If the $v$-component of $\mc{Y}_{\reg}$ has a finite cusp, then the finite cusp is geometrically ramified in $X \to \proj^1_K$. Furthermore, if some $D_{f_j}$ specializes to the finite cusp, then the geometric ramification index is strictly bigger than $2$. 
\end{lemma}
\begin{proof}
 By
Lemma~\ref{Lstandardendpointunique}, we have $e_v > e_{v_{n-1}} = N$ where $n$ is the
inductive length of $v$, and since $N \mid e_v$, we have $e_v/N \geq 2$.
%(here we use that $v$ has a neighbor in the
%direction of $\infty$ and $y$ does not lie on the intersection of
%$\ol{Z}_v$ and the component corresponding to this neighbor to rule out
%the case of \cite[Lemma 7.3(ii)]{ObusWewers}).
By
Proposition~\ref{Pbranchpointramindex} (with $f_i$ playing the role of
$\phi_n$ in that proposition), the geometric ramification index at $y$ is $\geq e_v/N$, and the inequality is strict if some $D_{f_j}$ specializes to the cusp, as desired. 
%By Proposition~\ref{Pbranchpointramindex} (with $f_j$ playing the role of $\phi_n$ in that proposition), the geometric ramification at $y$ is strictly greater than $e_v/N$, which in turn is at least   $2$ since $y$ is a finite cusp by Lemma~\ref{Lstandardendpointunique}.
\end{proof}

\begin{lemma}\label{Lspfigeoram}
 Assume that $D_{f_i}$ intersects the $v$-component $\ol{Z}_v$ of $\mc{Y}_{\reg}$ at a closed point $y$. Then $y$ lies
only on $\ol{Z}_v$, and if $v \neq v_0$, then $y$ is geometrically ramified in $X \to \proj^1_K$.
\end{lemma}
\begin{proof}
% First, suppose that some $D_{f_i}$ intersects the special fiber of $\mc{Y}_{\reg}$ at a point $y$ on the $v$-component $\ol{Z}_v$.  
By Lemma~\ref{LY0structure}(iv), $y$ lies
only on $\ol{Z}_v$, and $v$ is in some $B_{V_4, f_i}$.  Furthermore,
by Lemma~\ref{LY0structure}(v), $D_{f_i}$ does not intersect any other
$D_{f_j}$ on $\mc{Y}_{\reg}$.  If $y$ is regular on $\mc{Y}_{\reg}$, then
Corollary~\ref{Cregularbranchindex} shows that $y$ is geometrically
ramified in the cover $X \to \proj^1_K$ as desired. By
Corollary~\ref{CAllspecializations}\ref{Cinftyspecialization}, $y$ does not meet $D_{\infty}$
since {$v_0$} is the unique minimal valuation in $V_{\reg}$ and $v_0 \neq v$. Therefore, if $y$ is not
regular on $\mc{Y}_{\reg}$, then $y$ is a finite cusp and we can apply Lemma~\ref{Lfincuspram} and we are done. \qedhere
%(here we use that $v$ has a neighbor in the
%direction of $\infty$ and $y$ does not lie on the intersection of
%$\ol{Z}_v$ and the component corresponding to this neighbor to rule out
%the case of \cite[Lemma 7.3(ii)]{ObusWewers}).
%and by Lemma~\ref{Lstandardendpointunique}, we have $e_v > e_{v_{n-1}} = N$ where $n$ is the inductive length of $v$. By Proposition~\ref{Pbranchpointramindex} (with $f_i$ playing the role of $\phi_n$ in that proposition), $y$ is geometrically ramified in $X \to \proj^1_K$ as desired. 
\end{proof}

\begin{lemma}\label{Lv2nonmax}
If $v_0 \neq v \in V_2$ has at most
two neighbors, then the
$v$-component $\ol{Z}_v$ of $\mc{Y}_{\reg}$ contains a geometrically ramified
point in $X \to \proj^1_K$ that lies on no other irreducible component
of the special fiber of $\mc{Y}_{\reg}$.
\end{lemma}

\begin{proof}
By Lemma~\ref{Lspfigeoram}, it suffices to prove the lemma assuming that no branch point of $X \to \proj^1_K$
specializes to $\ol{Z}_v$.  Since $v_0 \neq v$ and $V_{\mathrm{reg}}$ is a rooted tree with root $v_0$, it follows that $v$ has a unique neighbor $w \prec v$. First suppose $v \in V_2 \setminus V_1$. Then $v =
\inf(v', v'')$ for $v', v'' \in V_1$. Since $v \prec v', v \prec v''$ and $w \prec v$ and $v$ is assumed to have
at most two neighbors in $V_{\mathrm{reg}}$,
at least one of $v'$ or $v''$ must equal $v_{f_i}^{\infty}$  for some
$i$, and furthermore, no valuation in $V_{\reg}$ can lie between $v$
and $v_{f_i}^{\infty}$.  By Proposition~\ref{Pbranchspecialization},
$D_{f_i}$ intersects $\ol{Z}_v$, a
contradiction. So we may assume $v \in V_1 \setminus \{v_0\}$, that is, $v$ is a predecessor of
some $v_{f_i}^{\infty}$.  
%We claim that $v$ has a finite cusp.  
%To see this, since $V_{\mathm{reg}}$ is a rooted tree, and $v \neq v_0$, observe that $v$ has a unique neighbor $w$ with $w \prec v$. 
%Since $v$ is not maximal, there is a neighbor $v \prec w'$. 
Since $v \prec v_{f_i}^{\infty}$ and no branch point of $X \to \proj^1_K$
specializes to $\ol{Z}_v$, by Proposition~\ref{Pbranchspecialization}, there is a valuation $w'$ such that  $v \prec w' \preceq
v_{f_i}$. Since $v$ has at most two neighbors, and $w \prec v$, it follows that such a  $w'$ is unique.  
%(note that if there is no such neighbor, then $v = v_{f_i}$ and $D_{f_i}$ specializes to $v$ \padma{justify? (since $v$ is non-maximal)?}).  
By Proposition~\ref{Pmaximalforlength}, the
inductive length of $w'$ is greater than that of $v$.  So by
Corollary~\ref{Cstandardendpointguaranteed}, $\mc{Y}_{\reg}$ has a finite cusp on the $v$-component. Now apply Lemma~\ref{Lfincuspram}. \qedhere
%Since $v \neq v_0$, by Lemma~\ref{Lvnvprime} applied to $v$, we have $e_v > e_{v_{n-1}}$ where $n$ is the inductive length of $v$. Since $e_v > e_{v_{n-1}}$, Proposition~\ref{Pbranchpointramindex} (with $a = 0$ in that proposition) shows that the finite cusp on the $v$-component is geometrically ramified in $X \to \proj^1_K$ as desired. 
\end{proof}

\begin{corollary}\label{Cv2nonmax}
If $v_0 \neq v \in V_2$ is non-maximal in $V_{\reg}$, then $v$ is not
removable from $V_{\reg}$.
\end{corollary}

\begin{proof}
Since $v_0$ is the unique minimal valuation in $V_{\reg}$, the
valuation $v$ is neither maximal nor minimal, so it has at least two
neighbors.  By Lemma~\ref{Lramificationnotremovable}(i), we may assume
that $v$ has exactly two neighbors.  By Lemma~\ref{Lv2nonmax}, the $v$-component $\ol{Z}$ of $\mc{Y}_{\reg}$ contains a geometrically
ramified point in $X \to \proj^1_K$ that lies on no other irreducible
component of the special fiber of $\mc{Y}_{\reg}$.
Applying Lemma~\ref{Lramificationnotremovable}(ii) proves the corollary.
\end{proof}

\begin{lemma}\label{Lv2onlymax}
Suppose $v \in V_2 \setminus V_1$ is maximal in $V_{\reg}$.
\begin{enumerate}[\upshape (i)]
 \item Then $v = \inf(v_{f_i}^{\infty},
  v_{f_j}^{\infty})$ for $f_i \neq f_j$ monic irreducible factors of
  $f$ and the horizontal branch components $D_{f_i}$ and $D_{f_j}$ specialize to distinct regular points of the $v$-component of $\mc{Y}_{\mathrm{reg}}$. 
  \item\label{Lv2evew} Furthermore, $v_{f_i} = v_{f_j}$, and if $w \prec v$  is $v$'s
    neighbor in the rooted tree $V_{\mathrm{reg}}$, then $v_{f_i} \prec w$  and $e_v = e_{v_{f_i}} = e_{v_{f_j}} \mid e_w$.
  \item\label{Lv2georam} The specializations of $D_{f_i}$ and $D_{f_j}$ are geometrically ramified points of $\mc{X}_{\reg} \rightarrow \mc{Y}_{\reg}$.
\end{enumerate}
\end{lemma}
\begin{proof}
 Since $\inf(v_1,v_1') \prec v_1$ for any pair of elements $v_1,v_1'$ in $V_1$, if $v \in V_2 \setminus V_1$ is maximal in $V_{\reg}$, the only possibility is that $v = \inf(v_{f_i}^{\infty},
  v_{f_j}^{\infty})$ for $f_i \neq f_j$ monic irreducible factors of
  $f$. Since the set of valuations bounded above by $ v_{f_i}^{\infty}$ is totally ordered and both $v,v_{f_i}$ belong to this set, either $v \prec v_{f_i}$ or $v_{f_i} \prec v$ (likewise with $i$ replaced by $j$). Since $v$ is maximal, we conclude that $v_{f_i}, v_{f_j} \prec v$. 
  %By \cite[Proposition 2.25]{KW}, $v_{f_i}$ and $v_{f_j}$ are
  %comparable, so assume without loss of generality that $v_{f_i}
  %\preceq v_{f_j}$.
  Now, $v_{f_j} \prec v \prec v_{f_j}^{\infty}$,
  so
  \begin{equation}\label{Efj}
    v = [v_{f_j}, v(f_j) = \lambda]
  \end{equation}
  for some $\lambda$.   By
  symmetry, we can also write
  \begin{equation}\label{Efi}
   v = [v_{f_i}, v(f_i) = \lambda'].
  \end{equation}
 Since $v_{f_i}$ and $v_{f_j}$ are both the immediate
 predecessor of $v$, we have $v_{f_i} = v_{f_j}$.  
By Proposition~\ref{Pbranchspecialization},
$D_{f_i}$ and $D_{f_j}$ both meet the $v$-component
$\ol{Z}_v$ of $\mc{Y}_{\reg}$.  By Lemma~\ref{LY0structure}(v),
the divisors $D_{f_i}$ and
$D_{f_j}$ meet the $v$-component at distinct points. We now prove  $e_v = e_{v_{f_i}} = e_{v_{f_j}}$. 
If not, then $e_v > e_{v_{f_i}} = e_{v_{f_j}}$ and
Corollary~\ref{Cmaximalcusp} shows that both $D_{f_i}$ and $D_{f_j}$
meet the unique finite cusp on the $v$-component of $\mc{Y}_{\reg}$, which is a contradiction. 

 %Since $V_{\mathrm{reg}}$ is a rooted tree with root $v_0$ and $v \neq v_0$, there is a unique $w \in V_{\mathrm{reg}}$ such that $w \prec v$. 
 Let $\ol{Z}_w$
  be the $w$-component of $\mc{Y}_{\reg}$.  Then $v_{f_i} = 
  v_{f_j}$ is a predecessor of $w$, so $e_v = e_{v_{f_i}} =
  e_{v_{f_j}} \mid e_w$. By Lemma~\ref{Lstandardendpointunique}, it
  follows that the specializations of $D_{f_i}$ and $D_{f_j}$ are
  regular points of the $v$-component. This proves (i) and
  (ii). Part~(i) and Corollary~\ref{Cregularbranchindex} prove (iii).
  %Finally, Corollary~\ref{Cregularbranchindex} shows that the points where $D_{f_i}$ and $D_{f_j}$ meet the special fiber have nontrivial geometric ramification in $\mc{X}_{\reg} \to   \mc{Y}_{\reg}$.
\end{proof}

\begin{lemma}\label{Lv2max}
If $v \in V_2 \setminus V_1$ is maximal in $V_{\reg}$, then $v$ is not removable
from $V_{\reg}$.  
\end{lemma}

\begin{proof}
Assume that $v$ is removable from $V_{\reg}$. Let $w$ be the unique
predecessor of $v$ in the rooted tree $V_{\reg}$ and let $\ol{Z}_w$ be
the corresponding irreducible component.   The points where $D_{f_i}$
and $D_{f_j}$ meet the special fiber are geometrically ramified by
Lemma~\ref{Lv2onlymax}(iii).  By
  Lemma~\ref{Lramificationnotremovable}(iii), the geometric
  ramification indices at these points are both $2$.  
 Let $\tilde{e}_v$ (resp.\ $\tilde{e}_w$) be the multiplicity of the
  irreducible components of the special fiber of the normalization
  $\mc{X}$ of
  $\mc{Y}_{\reg}$ in $K(X)$ above $\ol{Z}_v$ (resp.\ $\ol{Z}_w$).  By
  Lemma~\ref{Lremovable}, $v$ is removable from $V_{\reg}$ only if
  $\tilde{e}_v = 2\tilde{e}_w$. Now, $e_v \mid e_w$ by
  Lemma~\ref{Lv2onlymax}\ref{Lv2evew}, and $e_w \mid \tilde{e}_w$, so $e_v \mid \tilde{e}_w$.  
  On the other hand, Corollary~\ref{Cregularbranchindex} shows that
  $\tilde{e}_v/e_v$ is odd, which contradicts $\tilde{e}_v  = 2\tilde{e}_w$.  Thus $v$ is not removable from $V_{\reg}$. \qedhere  
%{\sout{By Lemma~\ref{Lv2onlymax} and Corollary~\ref{Cregularbranchindex}, the points where $D_{f_i}$ and $D_{f_j}$ meet the special fiber have nontrivial geometric ramification in $\mc{X} \to   \mc{Y}_{\reg}$.}} 
% \andrew{\sout{
%   By the removability of $v$, Lemma~\ref{Lv2onlymax}\ref{Lv2georam} and Lemma~\ref{Lramificationnotremovable}(iv), the intersection point $\ol{Z}_v \cap \ol{Z}_w$ is \emph{not} geometrically ramified in
%   $\mc{X} \to \mc{Y}_{\reg}$.  In particular, each component above
%   $\ol{Z}_v$ has at least two, and thus exactly
%    two intersection points with the preimage of $\ol{Z}_w$ in
%   $\mc{X}$ (more intersection points would result in not having normal
%   crossings when $\ol{Z}_v$ is contracted).  At the same time,
%   Lemma~\ref{Lramificationnotremovable}(iii) shows that the geometric
%   ramification indices at the points where $D_{f_i}$ and $D_{f_j}$
%   meet the special fiber are both $2$.  
%  
%   Let $\tilde{e}_v$ (resp.\ $\tilde{e}_w$) be the multiplicity of the
%   irreducible components of the special fiber of the normalization
%   $\mc{X}$ of
%   $\mc{Y}_{\reg}$ in $K(X)$ above $\ol{Z}_v$ (resp.\ $\ol{Z}_w$).
%   In order for an irreducible component above $\ol{Z}_v$ to have self-intersection number $-1$, we would need
%   $\tilde{e}_v = 2\tilde{e}_w$ \padma{by Castelnuovo -- add ref}.  By
%   Corollary~\ref{Cregularbranchindex}, $\tilde{e}_v/e_v$ is odd.
%   Since $e_v \mid e_w$ by Lemma~\ref{Lv2onlymax}(ii), and $e_w \mid \tilde{e}_w$, this contradicts
%   $\tilde{e}_v  = 2\tilde{e}_w$ .  Thus $v$ is not removable
%   from $V_{\reg}$.}}
  \end{proof}

  \begin{lemma}\label{Lv1max} 
   Suppose $v \in V_1 \setminus \{v_0\}$ is maximal in $V_{\mathrm{reg}}$. 
   \begin{enumerate}[\upshape (i)]
    \item\label{Lvmaxvfi} Then $v=v_{f_i}$ for some $f_i$ dividing $f$, the $v$-component has a finite cusp on $\mc{Y}_{\reg}$, and $D_{f_i}$ meets $\mc{Y}_{\reg}$ at a regular closed point of the $v$-component. 
    \item\label{Lv1maxgeoram} The specialization of $D_{f_i}$ and the finite cusp of the $v$-component are two distinct geometrically ramified points in $X \rightarrow \P^1_K$. 
    \item\label{Lv1maxnonrem} If $D_{f_j}$ intersects the $v$-component for some $j \neq i$, then $v$ is not removable.    
   \end{enumerate}
\end{lemma}
\begin{proof}
Since every element of $V_1$ is a predecessor of $v_{f_i}$ for some $i$, the maximality of $v$ implies that $v = v_{f_i}$ for some monic irreducible $f_i \mid f$.
Since $v \neq v_0$, Lemma~\ref{Lvnvprime} implies that
$e_v/e_{v_{n-1}} > 1$, where $n$ is the inductive length of $v$. Since
$v$ is maximal and $e_v > e_{v_{n-1}}$, the $v$-component $\ol{Z}_v$
of $\mc{Y}_{\reg}$ has a finite cusp by
Corollary~\ref{Cmaximalcusp}. Since $v = v_{f_i}$ is maximal in
$V_{\reg}$, the divisor $D_{f_i}$ meets the $v$-component by Proposition~\ref{Pbranchspecialization}. Furthermore, by Lemma~\ref{Lnonspecialize}(iii) applied to $f_i$ (which is a key polynomial over $v$) and Lemma~\ref{Lstandardendpointunique}, $D_{f_i}$ does
  not meet the finite cusp on $\ol{Z}_v$. This proves (i). Combining (i) with Lemma~\ref{Lfincuspram} and Lemma~\ref{Lspfigeoram} proves (ii).
  
  %Combining (i) with Lemma~\ref{Lfincuspram} and Lemma~\ref{Lspfigeoram}, we get that the specialization of $f_i$ and the finite cusp on the $v$-component are two distinct geometrically ramified points of the $v$-component, proving (ii).
  %in $X \rightarrow \proj^1_K$.

It remains to show that if $D_{f_j}$ for $j \neq i$ meets the $v$-component $\ol{Z}_v$, then $v$ is not removable. By
  Proposition~\ref{P35} applied to the non-comparable (and thus
  non-adjacent) valuations $v_{f_i}^{\infty}$ and
  $v_{f_j}^{\infty}$, the divisors $D_{f_i}$ and $D_{f_j}$ do not meet
  on $\mc{Y}_{\reg}$. If $D_{f_j}$ specializes to a regular point of
  the $v$-component (necessarily distinct from the unique finite cusp
  and the specialization of $f_i$), then the $v$-component has at
  least $3$ distint geometrically ramified points by (ii), and hence
  $v$ is not removable by Lemma~\ref{Lramificationnotremovable}(iv).
  If $D_{f_j}$ specializes to a non-regular point, then since $v \neq
  v_0$, this
  non-regular point is the finite cusp of the $v$-component by Corollary~\ref{CAllspecializations}\ref{Cinftyspecialization}. Lemma~\ref{Lfincuspram} shows that the geometric ramification index at the finite cusp is strictly larger than $2$, and thus $v$ is not removable by Lemma~\ref{Lramificationnotremovable}(iii). 
\end{proof}

We are finally ready to characterize the removable valuations in
$V_{\reg}$ (other than $v_0$) in Definition~\ref{Dremovability} and prove Proposition~\ref{Pv1max}. 

\begin{defn}\label{Dremovability}
Let $f = \pi_k^af_1^{a_1} \cdots f_r^{a_r}$ be an irreducible factorization of
$f$ as in this section.   Let $d \in \nats$ with $\chara k \nmid d$.  Let $v = [v_0,\, \ldots,\,
v_n(\phi_n) = \lambda_n]$, and write $N$ for
$e_{v_{n-1}}$.  We say that $v$ satisfies the
\emph{removability criterion} with respect to $f$ and $d$ if $v \neq v_0$, it is
maximal in $V_{\reg}$ and the following all hold:
\begin{enumerate}
\item $v = v_{f_i}$ for a unique $1 \leq i \leq r$, 
\item for this $i$, we have $a_i \equiv d/2 \pmod{d}$,
\item $e_v/N = 2$,
\item $e_w/N = \gcd(d, e_ww(f))/\gcd(d, e_vv(f))$, where $w \prec v$
  is the unique neighbor of $v$ in the rooted tree
  $V_{\reg}$. 
\end{enumerate}

\end{defn}

\begin{prop}\label{Pneccondrem}
Suppose valuation $v \neq v_0$ is removable from $V_{\reg}$. Then 
$v$ satisfies
the removability criterion of Definition~\ref{Dremovability} with
respect to $f$ and $d$.  
\end{prop}

\begin{proof}
%First, assume that $v$ is removable from $V_{\reg}$.
 %{\sout{Let $\mc{Y}_{\reg}$ be the $V_{\reg}$-model of $\proj^1_K$.  Recall that $a_i$ is the exact power of $f_i$ dividing $f$.}}
By Proposition~\ref{Ponlyspecialcontract}, Corollary~\ref{Cv2nonmax},
Lemma~\ref{Lv2max} and Lemma~\ref{Lv1max}\ref{Lvmaxvfi}, $v$ being
removable implies that $v=v_{f_i}$ for some $f_i$ dividing $f$
and that $v$ is maximal in $V_{\reg}$. If $v=v_{f_j}$ for some $j \neq i$, since $v = v_{f_j}$ is maximal in $V_{\reg}$, it follows that $D_{f_j}$ also specializes to the $v$-component by Proposition~\ref{Pbranchspecialization}.  Part~(a) now follows from Lemma~\ref{Lv1max}\ref{Lv1maxnonrem}.

Lemma~\ref{Lv1max}\ref{Lv1maxgeoram} and Lemma~\ref{Lramificationnotremovable}(iii) show that
the geometric ramification indices at the specialization of $f_i$ and the finite cusp are both $2$, and there are no
other geometrically ramified points of $\ol{Z}_v$. By
Corollary~\ref{Cregularbranchindex} and
Proposition~\ref{Pbranchpointramindex} (with $a = 0$ in that
proposition since no horizontal branch divisor meets the finite cusp), this is only possible if
$\gcd(d, a_i) = d/2$ and $e_v/N =
2$.  This verifies parts (b) and (c) of the
removability criterion. Part (d) follows from $e_v = 2N$ and Lemma~\ref{Lremovable}(ii). \qedhere

% Lastly, we observe that if $\tilde{e}_v$ (resp.\ $\tilde{e}_w$) is the multiplicity of the
% irreducible components of the special fiber of $\mc{X}$ above the $v$-
% (resp.\ $w$-) component of $\mc{Y}_{\reg}$, then
% \begin{equation}\label{Etildes}
%   \tilde{e}_v =
% \frac{e_vd}{\gcd(d, e_vv(f))} = \frac{2Nd}{\gcd(d, e_vv(f))} \text{ and } \tilde{e}_w =
% \frac{e_wd}{\gcd(d, e_ww(f))}.
% \end{equation}

%, $v$ being removable implies $\tilde{e}_v = 2\tilde{e}_w$, which is equivalent to part (d).
\end{proof}

\begin{prop}\label{Psatisfiesrem}
 If a valuation $v$ satisfies the removability criterion of Definition~\ref{Dremovability} with
respect to $f$ and $d$, then it is removable from
$V_{\mathrm{\reg}}$ and the unique neighbor of $v$ in the rooted tree
$V_{\mathrm{\reg}}$ is not removable from $V_{\reg} \setminus \{v\}$.
%\padma{Need the formulas for multiplicities of upstairs components outside the proof of the previous.}
\end{prop}
\begin{proof}
 %Now, we show that if $v$ verifies the removability criterion, then $v$ is removable.  
Let $v=v_{f_i}$ as in part~(a) of the removability criterion. By Lemma~\ref{Lnonspecialize}(iii) applied to $f_i$ (which is a key polynomial over
  $v$) and Lemma~\ref{Lstandardendpointunique}, $D_{f_i}$ does
  not meet the finite cusp and in particular specializes to a regular point on $\ol{Z}_v$.  
By part (b) of the removability criterion and Corollary~\ref{Cregularbranchindex}, the geometric
ramification index of $\mc{X} \to \mc{Y}_{\reg}$ at the specialization of $D_{f_i}$ is $2$. 
By part (a) and the
maximality of $v$, no $D_{f_j}$ other than $D_{f_i}$ intersects
$\ol{Z}_v$ either. By part (c) of the removability criterion and
Proposition~\ref{Pbranchpointramindex}, the geometric ramification
index at the finite cusp is $2$ as well (note that $a = 0$ in
Proposition~\ref{Pbranchpointramindex}  since the zeroes of $f$ do not
specialize to the finite cusp on $\ol{Z}_v$).

 We now claim that no other closed point on the $v$-component besides these two points is geometrically ramified. Indeed, since any such closed point does not lie on a horizontal component of the branch divisor, the claim follows from purity of the branch locus applied to
$\mc{X}/(\ints/e) \to \mc{Y}_{\reg}$ where $e = \tilde{e}_v/e_v$ is
the ramification index of $\ol{Z}_v$ in $\mc{X} \to \mc{Y}_{\reg}$.  Since all irreducible components of the reduced special fiber of $\mc{X}$ are smooth by Theorem~\ref{TX0regular}, the geometrically ramified points are smooth points of the components that they are on, and by
combining parts (c) and (d) of the removability criterion with
Lemma~\ref{Lremovable}(ii) we get that $v$ is
removable from $V_{\reg}$.
Lemma~\ref{Lremovable}(iii) shows that the unique neighbor of $v$ in $V_{\reg}$ is
not removable from $V_{\reg} \setminus \{v\}$.
\end{proof}

Let $S$ be the set of valuations satisfying the removability criterion
of Definition~\ref{Dremovability}. From now on, let $V_{\reg}' \colonequals V_{\reg} \setminus S$, and let $\nu' \colon \mc{X}'_{\reg} \to \mc{Y}_{\reg}'$ be the cover coming from contracting all the $v$-components $\ol{Z}_v$ for $v \in S$ and all the irreducible components lying above them in $\mc{X}_{\reg}$.

\begin{remark}\label{Rnooddcontract}
If $d$ is odd, then part (b) of the removability criterion of
Definition~\ref{Dremovability} does not hold, so $V_{\reg}' = V_{\reg}$.
\end{remark}

\begin{prop}\label{Pv1max}
 $V_{\reg}'$ is a regular normal crossings base, or equivalently, $\mc{X}'_{\reg}$ is regular. If $v_0 \neq v \in V'_{\reg}$, then $v$ is not removable from $V_{\reg}'$.
\end{prop}
\begin{proof}
The valuations in $S$ are maximal valuations in $V_{\mathrm{reg}}$ by
Definition~\ref{Dremovability}.
%Proposition~\ref{Ponlyspecialcontract}, Corollary~\ref{Cv2nonmax},
%and Lemma~\ref{Lv2max}.
No two maximal valuations in $V_{\mathrm{reg}}$ can be adjacent, so
the irreducible components corresponding to the valuations in $S$ are
pairwise disjoint by Proposition~\ref{P35}. Combining this with
Proposition~\ref{Psatisfiesrem} and Lemma~\ref{LCastelnuovo}, we get that the irreducible components corresponding to valuations in $S$ can be simultaneously contracted from $\mc{Y}_{\reg}$, or equivalently, that $V_{\reg}'$ is a regular normal crossings base.

%, where $S$ is the collection of valuations satisfying the removability criterion .
%(not including $v_0$) from $V_{\reg}$.  

%We now prove there are no further removable valuations other than $v_0$ after we remove all the valuations in $S$. 
If $v_0 \neq w \in V_{\mathrm{reg}}'$
and $w$ is adjacent to a valuation $v \in S$, then $w$ is not
removable from $V_{\reg}'$ by Proposition~\ref{Psatisfiesrem}. If $v_0 \neq w \in V_{\mathrm{reg}}'$ is not adjacent to a valuation in
$S$, then Lemma~\ref{LCastelnuovo} shows that it is not removable from
$V_{\mathrm{reg}}'$, because it is not removable from $V_{\mathrm{reg}}$ by
Proposition~\ref{Pneccondrem}, and the
neighboring valuations are unchanged from those in $V_{\mathrm{reg}}$. 
This completes the proof.
%$\ol{Z}_v$ and all components above it.
%, and let $z \in \mc{Y}_{\reg}'$
%be the image of $\ol{Z}_v$ in the contraction $\mc{Y}_{\reg} \to
%\mc{Y}_{\reg}'$.
\end{proof}

\begin{lemma}\label{Lhorizspec} 
%Let $\mc{Y}_{\reg}'$ be the $V_{\reg}'$-model of $\proj^1_K$.  Then 
The poset $V_{\reg}'$ is a rooted tree with root $v_0$ and 
each $D_{f_i}$ meets a single component of the special fiber of
$\mc{Y}_{\reg}'$.  
\end{lemma}
\begin{proof}
The analogous statement is true for $V_{\reg}$ and the $V_{\reg}$-model $\mc{Y}_{\reg}$ by
Lemma~\ref{LY0structure}(i), (iv). It remains true for
$V_{\reg}'$ and $\mc{Y}_{\reg}'$ since $\mc{Y}_{\reg}'$ comes from $\mc{Y}_{\reg}$ by
contracting maximal components not equal to the $v_0$ component, and thus every point of the special fiber
of $\mc{Y}_{\reg}$ lying on exactly one irreducible component still
does after applying the contraction map $\mc{Y}_{\reg} \to \mc{Y}_{\reg}'$.
\end{proof}

\begin{lemma}\label{Lminlength}
If $v$ is adjacent to $v_0$ in $V_{\reg}'$, then the inductive
length of $v$ is $1$.
\end{lemma}

\begin{proof}
  This is true for $V_{\reg}$ by Lemma~\ref{Lpredclosed} and
  Corollary~\ref{Cminlength1}.  Since $V_{\reg}'$ is constructed from
  $V_{\reg}$ by removing maximal elements, the lemma is true for
  $V_{\reg}'$ as well.
\end{proof}

\subsection{Contraction of minimal components}\label{Sminimal}

By Proposition~\ref{Pv1max}, the only valuation that is possibly removable from $V_{\reg}'$
is $v_0$.  In \S\ref{Sminimal}, we determine when $v_0$ is removable from
$V_{\reg}'$, as well as if, after removing $v_0$, more valuations
become removable.

\begin{lemma}\label{Lminimalv0}
Suppose $V$ is a regular normal crossings base for $X \to \proj^1_K$,
and that $V$ has a unique minimal valuation $v$ with inductive length
$\leq 1$.  Suppose further that
$v$ has at least two neighbors in $V$.  If $v$ is removable
from $V$, then $v$ has exactly two neighbors and $e_v = 1$. 
\end{lemma}

\begin{proof}
If $v$ has at least three neighbors, then by
Lemma~\ref{Lramificationnotremovable}(i) it is not removable.  So
assume $v$ has two neighbors.  If $\mc{Y}$ is the $V$-model of
$\proj^1_K$, then $D_{\infty}$ specializes only to the $v$-component
of $\mc{Y}$ by Corollary~\ref{CAllspecializations}\ref{Cinftyspecialization}.  Let $\mc{X}$ be the normalization of $\mc{Y}$ in
$K(X)$.  By Corollary~\ref{CinftySNC}, the standard
$\infty$-specialization to the $v$-component of $\mc{Y}$ is geometrically ramified
in $\mc{X} \to \mc{Y}$ of index divisible by $e_v$.  If $e_v > 1$, then
Lemma~\ref{Lramificationnotremovable}(ii) shows that $v$ is not removable.
\end{proof}

\begin{corollary}\label{Cthreenotremovable}
  If $v_0$ has at least three neighbors in $V_{\reg}'$, then
  $V_{\reg}'$ is the minimal normal crossings base for $X \to \proj^1_K$.
\end{corollary}

\begin{proof}
By Proposition~\ref{Pv1max}, the only valuation that is possibly
removable from $V_{\reg}'$ is the unique minimal valuation $v_0$.  By Lemma~\ref{Lminimalv0}, $v_0$ is in fact not removable.
\end{proof}
 
For the remainder of \S\ref{Sminimal}, it will be helpful to
   define a subset $S$ of $V_{\reg}'$ as follows:

\begin{defn}\label{DS}
 The set $S \subseteq V_{\reg}'$ consists of those
 valuations $v$ with inductive length $\leq 1$ satisfying condition (i), and either condition (ii),
 (iii), or (iv) of Proposition~\ref{PinftySNC}.
\end{defn}

\begin{remark}
 Note that $v_0$ satisfies condition
 (i) of Proposition~\ref{PinftySNC}, and our
 preliminary assumptions in \S\ref{Sreductions} and
 Lemma~\ref{Linftyfactorization} show that $v_0$ satisfies condition
 (ii) as well.  So $v_0 \in S$, and our main dichotomy will be between
 the cases $S = \{v_0\}$
 (Lemma~\ref{L01}, Proposition~\ref{Pcontractiontoinftycrossing}) and $S \supsetneq \{v_0\}$
 (Proposition~\ref{Pminimalcontraction}).
\end{remark}

\begin{lemma}\label{L01}
  Suppose $S = \{v_0\}$ as in Definition~\ref{DS}.  If $v_0$ has at
  most $1$ neighbor in $V_{\reg}'$, then $V_{\reg}'$ is the minimal normal crossings base for $X \to \proj^1_K$.
\end{lemma}

\begin{proof}
  If $v_0$ has no neighbors, then $V_{\reg}' = \{v_0\}$ and
  $v_0$ is not removable.  So suppose $v_0$ has $1$ neighbor in $V_{\reg}'$, say
  $w$.  Then $w$ has inductive length $1$ by
  Lemma~\ref{Lminlength}.  If $v_0$ is removable, then $w$ is the unique minimal
  valuation of $V := V_{\reg}' \setminus \{v_0\}$, and $V$ is a
  regular normal crossings base.  By
  Corollary~\ref{CAllspecializations}(i), the $V$-model has a
  standard $\infty$-specialization on the $w$-component.  In particular, all points above the
  standard $\infty$-specialization are
  regular with normal crossings.  By Proposition~\ref{PinftySNC}, $w$
  satisfies condition (i), as well as one of conditions (ii), (iii),
  or (iv) of that proposition.  So $w \in S$, which contradicts $S = \{v_0\}$.
  Thus $\{v_0\}$ is not removable from $V_{\reg}'$.  By
  Proposition~\ref{Pv1max}, $V_{\reg}'$ is the minimal regular normal
  crossings base.
\end{proof}

Taking into account Lemma~\ref{L01} and Corollary~\ref{Cthreenotremovable}, if $S =
\{v_0\}$, then the only case in which $v_0$ can be removable from
$V_{\reg}'$ is when $v_0$ has exactly $2$ neighbors.

\begin{lemma}\label{Lifvoremovable} Suppose $v_0$ is removable from $V'_{\reg}$ and has exactly two neighbors $w,w'$. Let $y,y'$ be the closed points where the $v_0$-component intersects the two neighboring components.
\begin{enumerate}[\upshape (i)]
%\item The points $y,y'$  are geometrically ramified in $\mc{X}'_{\reg} \rightarrow \mc{Y}'_{\reg}$.
%have exactly one preimage in the normalization $\mc{X}'_{\reg}$, and in particular 
\item None of the $D_{f_i}$ specialize only to the $v_0$-component in $\mc{Y}_{\reg}'$. In particular, every $D_{f_i}$ specializes to a $v$-component, where either $w \preceq v$ or $w' \preceq v$, or equivalently either $w \preceq v_{f_i}^{\infty}$ or $w' \preceq v_{f_i}^{\infty}$.
%if $f_i$ is an irreducible factor of $f$, then either $w \prec v_{f_i}^{\infty}$ or $w' \prec v_{f_i}^{\infty}$. In particular,
 \item $w$ and $w'$ are of the form $w = [v_0, v_1(t - c) = \mu]
$ and $w' = [v_0, v_1'(t - c') = \mu']$, where $v_K(c - c') = 0$ and $\mu$ and $\mu'$ satisfy the condition of
  Proposition~\ref{Pinftycrossingregular} (where $a$ and $\delta'$ from
  Proposition~\ref{Pinftycrossingregular} are defined at the beginning
  of \S\ref{Sinftycrossing}).
 %The valuations $w$ and $w'$ have length $1$, and  $f$ admits a factorization $f = \pi^a j j'$ satisfying the criterion as in Proposition~\ref{Pinftycrossingregular}.
\end{enumerate}
%Furthermore, there are polynomials $j,j'$ such that $f = \pi^a j j'$ where ...
%Retain the notation of Proposition~\ref{Pcontractiontoinftycrossing}. Suppose a $(v,v')$ does not exist. Then
 \end{lemma}

 \begin{proof} Since any point on the $v_0$-component that is not $y$ or $y'$ is automatically regular by \cite[Lemma 7.3(iii)]{ObusWewers}, if $D_{f_i}$ specializes only to the $v_0$ component, then this point is regular, and hence geometrically ramified by Corollary~\ref{Cregularbranchindex}. Therefore $v_0$ is not removable by Lemma~\ref{Lramificationnotremovable}(ii).
Since $v_0$ is the unique minimal valuation of $V'_{\reg}$, it follows that $w$ and $w'$ are the minimal valuations of $V'_{\reg} \setminus \{v_0\}$, and every valuation $v$ in $V'_{\reg}$ satisfies either $w \preceq v$ or $w' \preceq v$.  This proves (i).

By Lemma~\ref{Lminlength}, $w$ and $w'$ have inductive length
$1$.  By Lemma~\ref{Lfdegree}(i), they are of the form $w = [v_0, v_1(t - c) = \mu]$ and $w' = [v_0,
v_1'(t - c') = \mu']$ with $\mu, \mu' > 0$.
Since $V_{\reg}'$ is inf-closed, $\inf(w, w') = v_0$.  In particular,
$w$ and $w'$ are non-comparable.  Since $w(t - c') = w((t - c) + c -
c') = \min(\mu, v_K(c - c'))$ and similarly $w'(t - c) = \min(\mu',
v_K(c-c'))$, one computes $\inf(w, w') = [v_0,\, v_1(t - c) =
\min(\mu, \mu', v_K(c - c'))]$.  The fact that $\inf(w, w') = v_0$ implies that $v_K(c - c')
= 0$.  Combined with part~(i), we get that $f$ admits a factorization $f = \pi^a j j'$ as in the beginning of \S\ref{Sinftycrossing}.

Since $v_0$ is removable, the preimages in the normalization of the intersection of the $w$- and $w'$-components in the $V'_{\reg} \setminus \{v_0\}$-model are regular, which implies that $\mu$ and $\mu'$ satisfy the condition of
  Proposition~\ref{Pinftycrossingregular}.
\end{proof}

For the proposition below, we define a partial ordering on
\emph{ordered pairs} of Mac Lane pseudovaluations by $(v, v') \preceq
(w,w')$ if and only if $v \preceq w$ and $v' \preceq w'$.

\begin{prop}\label{Pcontractiontoinftycrossing}
Suppose that $S = \{v_0\}$ as in Definition~\ref{DS}.  Suppose further that $v_0$ has exactly two neighbors $w$ and $w'$ in $V_{\reg}'$.  Let $(v, v')$
be a maximal ordered pair in $V_{\reg}'$ such that
\begin{enumerate}[\upshape (i)]
\item $(w,w') \preceq (v, v')$,
\item $v$ and $v'$ are of the form $v = [v_0, v_1(t - c) = \mu]
$ and $v' = [v_0, v_1'(t - c') = \mu']$, where $v_K(c - c') = 0$, and
for all $i$, 
either $v \prec v_{f_i}^{\infty}$ or $v' \prec v_{f_i}^{\infty}$.
\item $\mu$ and $\mu'$ satisfy the condition of
  Proposition~\ref{Pinftycrossingregular} (where $a$ and $\delta'$ from
  Proposition~\ref{Pinftycrossingregular} are defined at the beginning
  of \S\ref{Sinftycrossing}).
  
  \end{enumerate}
If $V_{\min}$ is the set of all valuations $\nu$ in $V_{\reg}'$ such
that $\nu \succeq v$ or $\nu \succeq v'$, then $V_{\min}$ is
the minimal regular normal crossings base.

If no ordered pair $(v, v')$ as above exists, then $V_{\min} := V_{\reg}'$ is the
minimal normal crossings base.
\end{prop}

Before we prove this proposition, we prove a lemma about the structure of $V_{\min}$.
\begin{lemma}\label{Lvvstructure}
  Retain the notation of Proposition~\ref{Pcontractiontoinftycrossing}
  and the assumption that $S = \{v_0\}$. Suppose a $(v,v')$ as in the proposition exists. Let $\mc{Y}_{\min}$ be the $V_{\min}$-model.
If $v$ (resp.\ $v'$) is removable from $V_{\min}$, then $v$ (resp.\ $v'$) has a unique
   neighbor in $V_{\min}$.
 \end{lemma}

 \begin{proof}
Clearly it suffices to prove the lemma for $v$. If $v$ is removable
form $V_{\min}$, 
then $v'$ must have a neighbor in $V_{\min}$, because if not,
$V_{\min} \setminus \{v'\}$ would be a regular normal crossings base
with unique minimal valuation $v$.  By
Corollary~\ref{CAllspecializations}(i), the $v$-component of the
corresponding model would contain the standard
$\infty$-specialization, and Proposition~\ref{PinftySNC} would show
that $v$ satisfies condition (i) and one of conditions (ii), (iii), or
(iv) of that Proposition.  Thus we would have $v \in S$, which
contradicts the assumption that $S = \{v_0\}$.  So $v'$ has a neighbor
$v''$ in $V_{\min}$.  Furthermore, $v''$ is the unique such neighbor
of $v'$ by Lemma~\ref{Lramificationnotremovable}(i)).
\end{proof}

\begin{proof}[Proof of Proposition~\ref{Pcontractiontoinftycrossing}]
Let $\mc{Y}_{\min}$ be the
$V_{\min}$-model.  Suppose an ordered pair $(v,v')$ as in the
proposition exists.  Since $v(t-c') = v(t-c+c-c') = v_K(c-c') = 0$
(and similarly $v'(t-c) = 0$), we have $\inf(v, v') = v_0$, so $v$ and $v'$ are not comparable, and hence by construction are the two
minimal elements of $V_{\min}$. In particular, $v_0 \notin V_{\min}$. Furthermore, the $v$-component and
$v'$-component of $\mc{Y}_{\min}$ meet at the $\infty$-crossing $z$ in
$\mc{Y}_{\min}$ by
Corollary~\ref{CAllspecializations}\ref{Cinftyspectwominimal}, and by
the same corollary, the contraction morphism $V_{\reg}' \to V_{\min}$ is an isomorphism away from the preimage of $z$.

If $\mc{X}_{\min}$ is the normalization of
$\mc{Y}_{\min}$ in $K(X)$, then all points of $\mc{X}_{\min}$ above
$z$ are regular with normal crossings by
Proposition~\ref{Pinftycrossingregular}.  All points of
$\mc{Y}_{\min} \setminus \{z\}$ have
neighborhoods isomorphic to neighborhoods of
$\mc{Y}_{\reg}'$, and thus all points of $\mc{X}_{\min}$ lying above
$\mc{Y}_{\min} \setminus \{z\}$ are regular, and the special fiber has
normal crossings.  So $\mc{X}_{\min}$ is a regular normal crossings
model.  This is clearly also true when no $(v, v')$ exists.

It remains to show that $\mc{X}_{\min}$ is the \emph{minimal} regular
model with normal crossings.  Proposition~\ref{Pv1max} shows that no valuation in $V_{\reg}'$ is removable
other than possibly $v_0$. Suppose no $(v, v')$ exists.  Then $v_0$ is not removable by Lemma~\ref{Lifvoremovable}, and thus $V_{\reg}'$ has no removable valuations, proving $V_{\min} = V_{\reg}'$. So assume that $v_0$ is removable and let $(v,v')$ be as in the proposition (whose existence is guaranteed by Lemma~\ref{Lifvoremovable}).
%Then $v_0$ cannot be removable, since if it were, then Proposition~\ref{Pinftycrossingregular} would show that $(w, w')$ satisfies the conditions of the proposition, a contradiction. \padma{Need to show that $(w,w')$ satisfy the assumptions in the setup as in \S\ref{Sinftycrossing} -- need to spell out that $w$ and $w'$ have length $1$, that $f$ has the right type of factorization and that the invariants satsify the necessary path criterion.}

Now, if $w \in V_{\min} \setminus \{v, v'\}$, then $z$ is not in the
$w$-component of $\mc{Y}_{\min}$, which means that the contraction
$\mc{Y}_{\reg}' \to \mc{Y}_{\min}$ is an isomorphism on the preimage
of the
$w$-component.  Since the $w$-component is not removable from
$V_{\reg}'$ by Proposition~\ref{Pv1max}, it is thus not removable from $V_{\min}$.  So the only valuations that can
possibly be removable from $V_{\min}$ are $v$ and $v'$.

Suppose without loss of generality that $v'$ is removable from
$V_{\min}$.  By Lemma~\ref{Lvvstructure}, $v'$ has a unique neighbor
$v''$ in $V_{\min}$.  By definition of $v'$, the ordered pair $(v, v'')$
does not satisfy the criteria of the proposition.  By construction,
$(v, v'')$ satisfies (i).  If
$(v, v'')$ does not satify (ii), $v''$ has inductive length $2$, so
$\mc{Y}_{\min}$ 
has a finite cusp on the $v'$-component by Corollary~\ref{Cstandardendpointguaranteed}.
 By Lemma~\ref{Lstandardendpointunique}, $e_{v'} > 1$, so by 
Proposition~\ref{Pbranchpointramindex} (with $a = 0$ in that
proposition), the finite cusp on the $v'$-component is geometrically ramified in $X \to \proj^1_K$.  By Lemma~\ref{Lramificationnotremovable}(ii), $v'$ is not
removable from $V_{\min}$, which is a contradiction.  So $(v, v'')$
satisfies (ii).  Lastly, if $(v, v'')$
does not satisfy (iii), then
Proposition~\ref{Pinftycrossingregular} shows that after contracting all
components of $\mc{X}_{\min}$ above the $v'$-component of
$\mc{Y}_{\min}$, the resulting model is no longer regular with normal
crossings above the intersection of the $v$ and $v''$-components.  We conclude that $v'$ is not removable, proving the proposition.
%\andrew{Complete the proof by
%  showing that $v'$ in fact has another neighbor --- the idea is that
%  $v \in V_3$ so there is a $d$-tail coming out from it.  Need to show
% that the $d$-tail is non-trivial.  Probably this part should be a
% lemma when $V_{\reg}$ is constructed.}
\end{proof}

Now we turn to the case where $S \supsetneq \{v_0\}$.  
%First we prove a lemma.

\begin{lemma}\label{Lplaceholder}
  Take $v \in V_{\reg}'$ with inductive length $1$, and let  
  $V$ be the set of all $w \in V_{\reg}'$ such that $w \succeq v$.
  Suppose that $v$ has a unique neighbor $w
  \succ v$ in $V$, that the inductive length of $w$ is $2$, and that
  $V$ is a regular normal crossings base for $X \to \proj^1_K$.  Then $v$ is removable from $V$ if and only if all the
  following hold:
  \begin{enumerate}[\upshape (i)]
  \item $e_v = 2$,
  \item $w \prec v_{f_i}^{\infty}$ for all $i$,
 \item $\gcd(d, e_ww(f)) = 2e_w\gcd(d,a)$.
 \end{enumerate}

Furthermore, if $v$ is removable from $V$, then $w$ is \emph{not}
removable from $V \setminus \{v\}$.
\end{lemma}

First, we prove a sublemma. 

\begin{lemma}\label{Lramificationlocations}
Let $v$, $w$, and $V$ be as in Lemma~\ref{Lplaceholder}.   Let $\mc{Y}$ be the
$V$-model of $\proj^1_K$, and let $\mc{X}$ be its normalization in
$K(X)$.
\begin{enumerate}[\upshape (i)]
  \item The $v$-component of $\mc{Y}$ contains
    both a finite cusp and the standard $\infty$-specialization.
  \item If $y$ is one of these two points, then the
geometric ramification index of $y$ in $\mc{X} \to \mc{Y}$ is divisible by
$e_v$, with the divisibility being strict if and only if some
$D_{f_i}$ meets $y$.
  \item If $y$ is as in part (ii) and no $D_{f_i}$ meets $y$, then the
    reduced special fiber of $\mc{X}$ is smooth above $y$.
    \end{enumerate}
\end{lemma}

\begin{proof}
  By Corollary~\ref{Cstandardendpointguaranteed}, 
   $\mc{Y}$ has a finite cusp on the $v$-component, which implies by
   Lemma~\ref{Lstandardendpointunique} that $e_v \geq 2$.   Since $v$
   is minimal in $V$, there is a standard $\infty$-specialization on
   the $v$-component by Corollary~\ref{CAllspecializations}(i).  This proves (i).
   
   By Corollary~\ref{CinftySNC}, the standard $\infty$-specialization is geometrically ramified in $\mc{X} \to \mc{Y}$ with index
  divisible by $e_v \geq 2$, and this divisibility is strict if and
  only if there exists
  any $i$ with $v \not \preceq v_{f_i}^{\infty}$.  For such an $i$,
  Proposition~\ref{PAllspecializations}\ref{Clowvalspecialization} implies that $D_{f_i}$ meets
  the standard $\infty$-specialization.  Furthermore, if there does
  not exist such an $i$, then conditions (i) and (ii) of
  Proposition~\ref{PinftySNC} hold, so
  the points above the standard $\infty$-specialization are not
  nodes.
  
%  Furthermore, since $w$ has inductive length $\geq 2$, the
%  $w$-component of $\mc{Y}$ does not meet the $v$-component at the
%  finite cusp of the $v$-component, see
%  Corollary~\ref{Cstandardendpointguaranteed}.
  Suppose
  $v = [v_0,\, v_1(\phi_1) = \lambda_1]$.  Consider the invertible change of
  variables $u = \pi_K^{\lceil \lambda_1 \rceil} / \phi_1$.  Under
  this change of variables, it is easy to check that $v$ becomes
  $[v_0, v_1(u) = \lceil \lambda_1 \rceil - \lambda_1]$, the
  finite cusp in terms of $t$ becomes the standard
  $\infty$-specialization in terms of $u$, and $e_v$ remains
  unchanged.  So just as in the previous paragraph, the geometric
  ramification index at the finite cusp (in terms of $t$) is
  divisible by $e_v \geq 2$, and that divisibility is strict if and
  only if some
  $D_{f_i}$ meets the finite cusp.\footnote{Morally, this
    should follow from Proposition~\ref{Pbranchpointramindex}, but
    we are not exactly in a situation where that proposition is valid,
    since we don't know that we can write $f = \phi_1^a h$ as in that
    proposition.}  Also as in the previous paragraph, there are no nodes above the finite
  cusp if no $D_{f_i}$ meets it. This proves (ii) and (iii).
\end{proof}

\begin{proof}[Proof of Lemma~\ref{Lplaceholder}]
Let $\mc{Y}$ be the $V$-model of $\proj^1_K$, and let $\mc{X}$ be
its normalization in $K(X)$.  By
Corollary~\ref{CAllspecializations}(i), the contraction morphism
$\mc{Y}_{\reg}' \to \mc{Y}$ is an isomorphism outside the preimage of
the standard $\infty$-specialization, which lies on the
$v$-component. In particular, no $D_{f_i}$ meets an intersection of
two components by Lemma~\ref{Lhorizspec}, and, outside of possibly the $\infty$-specialization, no
two $D_{f_i}$ meet each other by Lemma~\ref{LY0structure}(v).

By Lemma~\ref{Lramificationnotremovable}(iii) and Lemma~\ref{Lramificationlocations}, if $v$ is removable
from $V$, then $e_v = 2$ (so (i) holds) and no $D_{f_i}$ meets either the standard
$\infty$-specialization or the finite cusp.  Also, since by
\cite[Lemma 7.3(iii)]{ObusWewers}, all
other points of the $v$-component are regular in $\mc{Y}$, except possibly
where the $v$- and $w$-components meet,
Corollary~\ref{Cregularbranchindex} shows that if any $D_{f_i}$ meets
any of these points lying only on the $v$-component, then it is geometrically ramified in $\mc{X} \to
\mc{Y}$.  By Lemma~\ref{Lramificationnotremovable}(iv), this implies
that $v$ is not removable from $V$.  We have seen that no
$D_{f_i}$ specializes to the intersection point of the $v$- and
$w$-components of $\mc{Y}$, so if $v$ is removable from $V$, then no
$D_{f_i}$ specializes to the $v$-component at all, and this means that
$w \prec v_{f_i}^{\infty}$ for all $i$, that is, (ii) holds.

Now, assuming (i) and (ii) hold, we will show that $v$ being removable
from $V$
is equivalent to (iii) holding, and that furthermore, $w$ is not
removable from $V \setminus \{v\}$ in this case.  This will complete the
proof.  Let $\ol{Z}_v$ be the $v$-component of $\mc{Y}$ and let $\ol{W}_v$
be an irreducible component of the special fiber of $\mc{X}$ lying
above $\ol{Z}_v$. Now, $e_v$ is the multiplicity of $\ol{Z}_v$, and write
$\tilde{e}_v$ for the multiplicity of $\ol{W}_v$.  Similarly, let
$\tilde{e}_w$ be the multiplicity of any irreducible component of the special fiber of $\mc{X}$
above the $w$-component.  Since (i) and (ii) hold, combined with Lemma~\ref{Lramificationlocations}, we see that $\ol{W}_v \to \ol{Z}_v$ is
geometrically ramified above two points, each with geometric
ramification index $2$,
and not above any other point, except possibly
where $\ol{Z}_v$ meets the $w$-component.
%\padma{Did we rule out geometric ramification at these intersection
%points? $\ol{W}_{v}$ may not be a smooth curve, so Riemann-Hurwitz
%may not apply.}
Furthermore, since (ii)
holds, no $D_{f_i}$ meets the $v$-component by
Proposition~\ref{Pbranchspecialization} and hence by Lemma~\ref{Lramificationlocations}(iii),  the ramified points in $\ol{W}_v$ are smooth points of $\ol{W}_v$.
%\andrew{This is a consequence of the opening paragraphs, but will get a reference once we split it off.}. 

We claim that (iii) is equivalent to $\tilde{e}_v =
2\tilde{e}_w$.  Admitting the claim,
Lemma~\ref{Lramificationlocations}(ii), (iii) and Lemma~\ref{Lremovable}(ii) shows
that $v$ is removable from $V$ if and only (iii) holds.  Furthermore,
Lemma~\ref{Lremovable}(iii) shows that $w$ is not removable from $V
\setminus \{v\}$ in this case.  This completes the proof, so we need
only prove the claim.

%But $\ol{W}_v^{\red} \to
%\ol{Z}_v^{\red}$ is a cyclic cover, and a cyclic cover cannot be ramified at two
%points with index $2$ and a third point (since the monodromy
%generators would not multiply to the identity) \padma{should we worry
%  about nodal components again?}. \andrew{If we fix it in
%Lemma~\ref{Lremovable}, then we can cite that here and shorten the proof.} So $\ol{W}_v^{\red}$
%has genus $0$, and $\ol{W}_v$
%meets the rest of the special fiber of $\mc{X}$ at exactly $2$
%points.  Since these meetings are transverse by assumption, we have
%that $\ol{W}_v^{\red}$ is a $-1$-curve, and thus contractible, if and
%only if $\tilde{e}_v = 2 \tilde{e}_w$.  So we are reduced to proving
%that (iii) is equivalent to $\tilde{e}_v = 2\tilde{e}_w$.

Let us calculate $\tilde{e}_v$ and $\tilde{e}_w$.  Since no $D_{f_i}$ meets the standard $\infty$-specialization on
$\mc{Y}$, condition (ii) of Proposition~\ref{PinftySNC} holds.  So
locally near the $\infty$-specialization, the cover is given birationally by the
equation $z^d = \pi_K^a \phi_1^e$.  Since $\phi_1$ is linear and $d \mid
\deg (f)$, we have $d \mid e$, which means the cover is equivalently given
birationally by the equation $z^d = \pi_K^a$.  By Lemma~\ref{Linftyequivalence}, the complete local ring above the
$\infty$-specialization contains $\sqrt{\pi_K}$, so $d / \gcd(d, a)$
is even.  Since the generators of the value group of an extension of
$v$ to $K(X)$ can be taken to be $1/e_v$ and $v(z) = a/d = (a/\gcd(d, a))/(d/\gcd(d,a))$, one computes
\begin{equation}\label{Eev}
      \tilde{e}_v = \lcm\left(e_v, \frac{d}{\gcd(d, a)}\right) =
      \lcm\left(2, \frac{d}{\gcd(d, a)}\right) = \frac{d}{\gcd(d,a)}.
\end{equation}    
On the other hand, the ramification index of $\mc{X} \to \mc{Y}$ above
the valuation $w$ is $d/\gcd(d, e_ww(f))$, so
\begin{equation}\label{Eew}
\tilde{e}_w = e_w d/\gcd(d,
e_w w(f)).
\end{equation}  Equating (\ref{Eev}) to twice (\ref{Eew})
shows that (iii) is equivalent to $\tilde{e}_v =
2\tilde{e}_w$, completing the proof.
\end{proof}

\begin{prop}\label{Pminimalcontraction}
Suppose $S \supsetneq \{v_0\}$ as in Definition~\ref{DS}.  Let $v$ be a maximal
element of $S$ (by assumption, $v \neq v_0$).  Let $V_{\min}'$ be the set of all valuations $w \in
V_{\reg}'$ with $w \succeq v$.  If $v$ satisfies the hypotheses and conditions of
Lemma~\ref{Lplaceholder} relative to $V = V_{\min}'$, let $V_{\min} =
V_{\min}' \setminus \{v\}$.  If not, let $V_{\min} = V_{\min}'$.  Then
$V_{\min}$ is the minimal regular normal crossings base for $X \to \proj^1_K$.
\end{prop}

\begin{remark}
  The proposition shows, a posteriori, that $v$ is \emph{the}
  maximal element of $S$.
\end{remark}

We begin with two preparatory lemmas. 

\begin{lemma}\label{Lminneighborhood}
 In the context of Proposition~\ref{Pminimalcontraction}, let $\mc{Y}'_{\min}$ be the $V'_{\min}$ model. 
 % \begin{enumerate}[\upshape (i)]
%   The $v$-component of   $V_{\mathrm{min}}'$ contains the standard $\infty$-specialization; call it $y$. Then the canonical contraction $\mc{Y}_{\reg}' \to \mc{Y}_{\min}'$ is an isomorphism
%outside the preimage of $y$.
 %$\mc{Y}'_{\min} \setminus \{y\}$ is isomorphic to an open subscheme of $\mc{Y}'_{\reg}$.
If $V_{\min}'$ is a regular normal crossings base, the only removable valuation from $V_{\min}'$, if any, is $v$.                                                                           
\end{lemma}
\begin{proof}
By construction, $v$ is the minimal element of
$V_{\min}'$, so
Corollary~\ref{CAllspecializations}\ref{Cinftyspecialization} shows that the standard $\infty$-specialization $y$ lies only on the
$v$-component and that the canonical contraction $\mc{Y}_{\reg}' \to \mc{Y}_{\min}'$ is an isomorphism
outside the preimage of $y$. Thus a valuation (other than $v$) is
removable from $V_{\min}'$ if and only if it is removable from
$V_{\reg}'$. Since either $v_0 =
v$ or $v_0 \notin V_{\min}'$, and Proposition~\ref{Pv1max} shows that no valuation in
$V_{\reg}'$ is removable other than possibly $v_0$, we conclude the
only valuation that can possibly be removed from $V_{\min}'$ is $v$.
%, the contraction morphism $\mc{Y}_{\reg}' \to \mc{Y}_{\min}'$ is an isomorphism outside the preimage of $y$.
\end{proof}

\begin{lemma}\label{Lcontradictminimality}  
Let $v$ be as in Proposition~\ref{Pminimalcontraction}. If $e_v = 1$, then there exists some $f_i$ with $v \not \prec v_{f_i}^{\infty}$.
\end{lemma}
\begin{proof}
For a contradiction, assume $v \prec v_{f_i}^{\infty}$ for all $i$.
In this case $v = [v_0, v_1(x - a) =
\lambda]$ with $a \in \mc{O}_K$ and $\lambda \in \ints_{\geq 1}$,
since $v \neq v_0$.
If $v \prec v_{f_i}^{\infty}$ for all $f_i$, 
Lemma~\ref{Lgfollows} shows that all roots $\theta$ of $f(x)$ satisfy $v_K(\theta - a) \geq \lambda
\geq 1$.  But this contradicts the assumption on $f$ from \S\ref{Sreductions}.  
\end{proof}

\begin{proof}[Proof of Proposition \ref{Pminimalcontraction}] 
%   
%   
% {\sout{We first show that $V_{\min}'$ is a regular normal crossings base
%   for $X \to \proj^1_K$.  The proof is similar to that of the corresponding   assertion in Proposition~\ref{Pcontractiontoinftycrossing}, but we give it here   for completeness.  Let $\mc{Y}_{\min}'$ be the $V_{\min}'$-model.  
% By construction, $v$ is the minimal element of
% $V_{\min}'$, so the standard $\infty$-specialization $y$ lies only on the
% $v$-component by
% Proposition~\ref{PAllspecializations}\ref{Cinftyspecialization}.}}

Let $\mc{X}_{\min}'$ be the normalization of the $V_{\min}'$-model
$\mc{Y}_{\min}'$ in $K(X)$, and let $y$ be the standard
$\infty$-specialization, which lies only on the $v$-component by
Corollary~\ref{CAllspecializations}(i).  We first show that
$\mc{X}_{\min}'$ is a  regular normal crossings model. By
Corollary~\ref{CAllspecializations}(i), all points of $\mc{X}_{\min}'$ not above $y$ are regular and normal crossings since $\mc{Y}'_{\reg}$ is a regular normal crossings base by Proposition~\ref{Pv1max}. Futhermore, by Definition~\ref{DS} and Proposition~\ref{PinftySNC}, all points of $\mc{X}_{\min}'$ above $y$ are also regular with normal crossings.  This proves $V_{\min}'$ is a regular normal crossings base. 

% {\sout{Furthermore, by the ``only if'' direction of \cite[Lemma~3.7(ii)]{KW} applied to $v$, taking $y = \infty_v$ in that lemma, the contraction morphism $\mc{Y}_{\reg}' \to \mc{Y}_{\min}'$ is an isomorphism outside the preimage of $y$. So all points of $\mc{X}_{\min}'$ lying above $\mc{Y}_{\min}' \setminus \{y\}$ are regular with normal crossings.
% It follows that $\mc{X}_{\min}'$ is a regular normal crossings model. }}
% 
% {\sout{We now show that $V_{\min}'$ is the \emph{minimal} regular normal
% crossings base, except when $v$ satisfies the hypotheses and
% conditions of Lemma~\ref{Lplaceholder} for $V_{\min}'$, in which case
% $V_{\min}' \setminus \{v\}$ is.  Again, since $\mc{Y}_{\reg}' \to
% \mc{Y}_{\min'}$ is an isomorphism outside of the preimage of $y$ by Lemma~\ref{Lminneighborhood}, and
% since Proposition~\ref{Pv1max} shows that no valuation in $V_{\reg}'$ is removable
% other than possibly $v_0$, the only valuation that can
% possibly be removed from $V_{\min}'$ is the one whose corresponding
% component contains $y$, namely, $v$. }} 

By Lemma~\ref{Lminneighborhood}, $V_{\min}'$ has no removable valuations if $v$ is not removable from $V_{\min}'$. To prove $V_{\min}$ is the \emph{minimal} regular normal crossings base, it suffices to show that $v$ is removable from $V_{\min}'$ precisely when it satisfies the hypotheses and conditions of Lemma~\ref{Lplaceholder}, and in this case, $V_{\min}' \setminus \{v\}$ has no further removable valuations. If $v$ has three or more
neighbors in $V_{\min}'$, it is not removable from $V_{\min}'$ by
Lemma~\ref{Lramificationnotremovable}(i).  Suppose $v$ has two neighbors in
$V_{\min}'$. Lemma~\ref{Lminimalv0} shows
that $v$ can be removed from $V_{\min}'$ only if
$e_v = 1$.   In this case, by Lemma~\ref{Lcontradictminimality},
%$v = [v_0, v_1(x - a) =
%\lambda]$ with $a \in \mc{O}_K$ and $\lambda \in \ints_{\geq 1}$,
%since $v \neq v_0$.
%If $v \prec v_{f_i}^{\infty}$ for all $f_i$, 
%Lemma~\ref{Lgfollows} shows that all roots $\theta$ of $f(x)$ satisfy $v_K(\theta - a_i) \geq \lambda
%\geq 1$.  But this contradicts our running assumption on $f$ from
%\S\ref{Sreductions}.  So
there is some $f_i$ such that $v \not \prec v_{f_i}^{\infty}$.  
Proposition~\ref{PAllspecializations}\ref{Clowvalspecialization} shows that $D_{f_i}$ meets $y$,
and Corollary~\ref{Cregularbranchindex} in turn shows that $y$ is geometrically ramified in $\mc{X}_{\min}'
\to \mc{Y}_{\min}'$.   By Lemma~\ref{Lramificationnotremovable}(ii),
$v$ is not removable from $V_{\min}'$.

So assume $v$ has a single neighbor $w \succ v$.  Suppose first that
$w$ has inductive length $1$ and $v$ is removable
from $V_{\min}'$.  Then, after contracting the $v$-component of
$\mc{Y}_{\min}'$, the $\infty$-specialization lies on $w$.  By
Proposition~\ref{PinftySNC}, condition (i) and either condition (ii),
(iii), or (iv) of Proposition~\ref{PinftySNC} hold for
$w$.  But this contradicts the maximality of $v$ in $S$.   If, on the
other hand, $w$ has
inductive length $2$, then Lemma~\ref{Lplaceholder} shows that $v$ is removable from $V_{\min}'$ if and
only if the conditions of that lemma hold.  Furthermore, in this case
Lemma~\ref{Lplaceholder} shows that $w$ is not removable from
$V_{\min}' \setminus \{v\}$, so $V_{\min} = V_{\min}' \setminus \{v\}$
is the minimal regular normal crossings base.  This completes the proof.
\end{proof}

Combining Theorem~\ref{TX0regular}, Proposition~\ref{Pv1max},
Corollary~\ref{Cthreenotremovable}, Lemma~\ref{L01}, and
Propositions~\ref{Pcontractiontoinftycrossing} and
\ref{Pminimalcontraction}, we get the following theorem, which is
the main result of the paper.

\begin{theorem}\label{Tminnormalcrossingsbase}
Let $f \in \mc{O}_K[t]$ satisfy the assumptions from
\S\ref{Sreductions}.  The (unique) minimal normal crossings base $V_{\min}$ for the cover $X \to
\proj^1_K$ given by $z^d = f$ is constructed as follows:
\begin{enumerate}[\upshape (1)]
\item Construct $V_{\reg}$ as in Algorithm~\ref{AY0} (see Theorem~\ref{TX0regular}).
\item Construct $V_{\reg}' \subseteq V_{\reg}$ by removing all
 vertices satisfying the removability criterion of
  Definition~\ref{Dremovability} (see Proposition~\ref{Pv1max}).
\item Let $S \subseteq V_{\reg}'$ be the set constructed in
  Definition~\ref{DS}.  Let $n$ be the number of neighbors of $v_0$ in
  $V_{\reg}'$. 
  \begin{enumerate}[\upshape (i)]
    \item If $S = \{v_0\}$ and  $n \neq 2$, 
   then set
  $V_{\min} = V_{\reg}'$ (see Corollary~\ref{Cthreenotremovable} and Lemma~\ref{L01}).
\item If $S = \{v_0\}$ and $n=2$, 
then
  construct $V_{\min} \subseteq V_{\reg}'$ as in
  Proposition~\ref{Pcontractiontoinftycrossing}.
\item If $S \supsetneq \{v_0\}$,
then construct
  $V_{\min} \subseteq V_{\reg}'$ as in Proposition~\ref{Pminimalcontraction}.
\end{enumerate}
\end{enumerate}
  
\end{theorem}

\begin{comment}
\begin{corollary}\label{Cminnormalcrossingsbase}
Maintain the hypotheses of Theorem~\ref{Tminnormalcrossingsbase}. 
  \begin{enumerate}[\upshape (i)]
    \item If $f$ is monic (that is, if $a = 0$), then $V_{\min} = V_{\reg}'$.
    \item If, in addition to (i), $d$ is odd, then $V_{\min} =
      V_{\reg}$.
    \end{enumerate}    
\end{corollary}

\begin{proof}
To prove part (i), we may assume that we are in case (3)(ii) or
(3)(iii) of Theorem~\ref{Tminnormalcrossingsbase}, that is, that $v_0$
has $1$ or $2$ neighbors in $V'_{\reg}$.  First assume $v_0$ has $1$
neighbor.  Since $a = 0$, part (i) of Proposition~\ref{PinftySNC} is
satisfied for a valuation $v$ if and only if $e_v = 1$.

Now suppose $v_0$ has $2$ neighbors.  By
Theorem~\ref{Tminnormalcrossingsbase}(3)(ii), it suffices to show
that there do not exist valuations $v$, $v'$ in $V_{\reg}'$ satisfying
the criterion of Proposition~\ref{Pcontractiontoinftycrossing}.  Given
that two valuations $v$, $v'$ satisfy part (ii) of the criterion of
Proposition~\ref{Pcontractiontoinftycrossing},
Remark~\ref{Rneverwithmonic} shows that they do not satisfy part (iii)
of the criterion.  This shows that $V_{\min} = V_{\reg}'$, completing
the proof of part (i).

Part (ii) follows immediately from Remark~\ref{Rnooddcontract}.
\end{proof}
\end{comment}

\section{Examples}\label{Sexamples}
\begin{example}\label{E2}
  For the $\ints/5$-cover of $\proj^1_K$ given birationally by $z^5 = (t-1)^2(t^3 - \pi_K^2)$ in
  Example~\ref{Ebasic} with $\chara k \neq 5$, we verify that this
  paper's algorithm shows that $V_{\min} =
  V_{\reg}$ (note that $V_{\min} = V_{\reg}$ was already shown by
  other methods in Example~\ref{Ebasic}).  Recall
  that
  $$V_{\reg} = \{v_0, v_{5/8}, v_{2/3}, v_{7/10}, v_{4/5},
w_{10/3}, w_{5/2}, w_{20/9}, w_{25/12}\},$$ where $v_{\lambda} =
[v_0,\, v_1(t) = \lambda]$, and $w_{\lambda} = [v_0,\, v_1(t) = 2/3,\,
v_2(t^3 - \pi_K^2) = \lambda]$.  No maximal valuation in $V_{\reg}$
satisfies part (a) of the removability
criterion in Definition~\ref{Dremovability}, so by part (2) of Theorem~\ref{Tminnormalcrossingsbase},
$V_{\reg} = V_{\reg}'$. A valuation $v$ satisfies
condition (i) of Proposition~\ref{PinftySNC} if and only if
$e_v = 1$ (this is because $a = 0$ in that proposition).  The only
such valuation in $V_{\reg}'$ is $v_0$, and $v_0$ also satisfies condition
(ii) of Proposition~\ref{PinftySNC}, so $S = \{v_0\}$ in
Definition~\ref{DS}.  Since the only neighbor of $v_0$ in $V_{\reg}'$
is $v_{5/8}$, we are in case (3)(i) of
Theorem~\ref{Tminnormalcrossingsbase}.  In particular, $V_{\reg} = V_{\min}$.
\end{example}

\begin{example}\label{E1}
Consider the cover given by $z^2 = (t-1)(t-2)(t^2-\pi_K)$ with $\chara
k \neq
2$.  The normalization of the standard model
$\proj^1_{\mc{O}_K}$ of $\proj^1_K$ in the function field
corresponding to the cover gives a regular normal crossings model.
Indeed, the affine equation for such a model inside
$\aff^2_{\mc{O}_K}$ is simply $z^2 =
(t-1)(t-2)(t^2 - \pi_K)$, and it is easy to check that this gives a regular
scheme with normal crossings (the cover is \'{e}tale above $t =
\infty$, so there are no issues there).  In other words, $\{v_0\}$ is a
minimal regular normal crossings base.

We show how this results from our algorithm.  Write $f_1 = t-1$, $f_2 =
t-2$, and $f_3 = t^2 - \pi_K$.  Then
$$v_{f_1}^{\infty} = [v_0,\, v_1(t-1) = \infty],\, v_{f_2}^{\infty} =
[v_0,\, v_1(t-2) = \infty],\, v_{f_3}^{\infty} = [v_0,\, v_1(t) =
1/2,\, v_2(t^2 - \pi_K) = \infty].$$

So $V_1$ consists of the $v_{f_i}^{\infty}$ as well as their
predecessors $v_0$ and
$w := [v_0,\ v_1(t) = 1/2]$.  This set is already inf-closed, so $V_1
= V_2$.

The only adjacent pair of valuations in $V_2$ is $(v_0, w)$, so to form $V_2$, we replace this pair with the link
$L_{v_0,w}$.  In the language of Definition~\ref{Dlink}(i), we have $g = t^2 - \pi_K$ and $h = (t-1)(t-2)$, so $N = 1$, $e = 2$, $s =
0$, $d = 2$, $\widetilde{N} = 1$, and $r = 0$.  Thus we adjoin $v_\lambda
:= [v_0,\,
v_1(t) = \lambda]$, where $\lambda$ ranges over the shortest $1$-path between $0$ and
$1/2$.  Since $1/2 > 0$ is already a $1$-path, we see that $V_3 = V_2 = V_1$. 

To form $V_4$, observe that the only valuation in $V_3$ with a finite cusp is $v_{1/2}$.  So we replace this valuation with the
tail $T_w$.  In the language of Definition~\ref{Dtail}(i), for
this tail we have $g = 1$ and $h = f$, so $N
= 1$, $e = 0$, $s = 1$, $d = 2$, and $\widetilde{N} = 2$. 
By definition, $T_w = L_{w, w}$, which equals $\{w\}$.  So $V_4 = V_3
= V_2 = V_1$.

To form $V_5$, we append branch point tails $B_{V_4, f_i}$ for $i \in \{1, 2, 3\}$.  For $i=1$, we have (in
the language of Definition~\ref{Dtail}(ii)) that $g = t-1$ and $h =
(t-2)(t^2-\pi_K)$, so $N = 1$, $e = 1$, $d = 2$, and thus $\tilde{N} =
1$.  So $B_{V_4, f_1} = L_{v_0, v_0} = \{v_0\}$.  Similarly, $B_{V_4, f_2} =
\{v_0\}$.  For $B_{V_4, f_3}$, we have $g = t^2 - \pi_K$ and $h =
(t-1)(t-2)$, so $N = 2$, $e = 1$, $s = 0$, $d = 2$, $\widetilde{N} =
2$, and $r = 1$.  So $B_{V_4, f_3} = L_{w, w} = \{w\}$ (here we
interpret $w$ as $[v_0,\, v_1(t) = 1/2,\, v_2(t^2 - \pi_K) = 1]$).  So
$V_5 = V_4 = V_3 = V_2 = V_1$, and $V_{\reg} = \{v_0, w\}$.

Now, $w$ satisfies all the criteria of Definition~\ref{Dremovability}
(in the language of criterion (d), both sides equal $1$),
so by part (2) of Theorem~\ref{Tminnormalcrossingsbase}, $V_{\reg}' =
\{v_0\}$.  Thus we are in case (3)(i) of
Theorem~\ref{Tminnormalcrossingsbase}, and 
the same theorem shows that $V_{\min} = \{v_0\}$, as expected.
\end{example}

\begin{example}\label{E3}
We exhibit an example where $V_{\reg}'  \neq V_{\min}$.
Consider the $\ints/8$-cover $X \to \proj^1_K$ given birationally by $z^8 = f := \pi_K(t^2
- \pi_K)^4$, where $\chara k \neq 2$.  In this case, $f_1 = (t^2 -
\pi_K)$, and $v_{f_1}^{\infty} = [v_0,\, v_1(t)
= 1/2,\, v_2(t^2 - \pi_K) = \infty]$.   So $V_1$ consists of $v_{f_1}^{\infty}$ and its predecessors $v_0$ and
$v_{1/2} := [v_0,\, v_1(t) = 1/2]$.  This set is already inf-closed, so $$V_1 =
V_2 = \{v_{f_1}^{\infty}, v_0, v_{1/2}\}.$$

The only adjacent pair of valuations in $V_2$ is $(v_0, v_{1/2})$, so to form $V_3$, we replace this pair with the link
$L_{v_0,v_{1/2}}$, defined in Definition~\ref{Dlink}.  We have $g = (t^2 - \pi_K)^4$ and $h = \pi_K$, so $N = 1$, $e = 8$, $s =
1$, $d = 8$, $\widetilde{N} = 8$, and $r = 0$.  Thus we adjoin $v_\lambda
:= [v_0,\, v_1(t) = \lambda]$, where $\lambda$ ranges over the
shortest $8$-path from $1/2$ to $0$.  This $8$-path is $1/2 > 3/8 > 1/4
> 1/8 > 1/2$ so $V_3 = V_2 \cup
\{v_{1/8}, v_{1/4}, v_{3/8}\}$. That is, $$V_3 = \{v_{f_1}^{\infty}, v_0,
v_{1/8}, v_{1/4}, v_{3/8}, v_{1/2}\}.$$

To form $V_4$, observe that the only valuation in $V_3$ with a finite cusp is $v_{1/2}$.  So we replace this valuation with the
tail $T_{v_{1/2}}$ from Definition~\ref{Dtail}(i).  For this tail, we have $h = f$ and $g = 1$, so $N
= 1$, $e = 0$, $s = 5$, $d = 8$, and $\widetilde{N} = 8$.
By definition, $T_{v_{1/2}} = L_{v_{1/2}, v_{1/2}} = \{v_{1/2}\}$, so
$V_4 = V_3$.

To form $V_5$, observe that the valuation in $V_4$ that is maximal
among those bounded above by $v_{f_1}^{\infty}$ is $v_{1/2}$.  So we
replace this valuation with the branch point tail $B_{V_4, f_1}$ as in
Definition~\ref{Dtail}(ii).  For
this tail, we have $N = 2$ (since we think of $v_{1/2}$ as $[v_0,\,
v_1(t) = 1/2,\, v_2(t^2 - \pi_K) = 1]$), and $g = (t^2 - \pi_K)^4$ and $h = \pi_K$.  So $e = 4$,
$s = 2$, $d = 8$, $\widetilde{N} = 4$, and $r = 1$.  Then $B_{V_4, f_1}
= L_{v_{1/2} =: w_1,  w_{5/4}}$, where for $\lambda \in \rats$, we define
$w_{\lambda} := [v_0 =: w_0,\, w_1(t) = 1/2,\, w_2(t^2 - \pi_K) =
\lambda]$. Thus we adjoin $w_{\lambda}$ where $\lambda$ ranges over
those numbers such that $\lambda/2 + 1/8$ forms the shortest $4$-path from
$(5/4)/2 + 1/8 = 3/4$ to $1/2 + 1/8 = 5/8$.  This $4$-path is simply $3/4 > 5/8$, so
$$V_5 = \{v_{f_1}^{\infty}, v_0, v_{1/8}, v_{1/4}, v_{3/8}, v_{1/2} =
w_1, w_{5/4}\},$$ and $V_{\reg} = V_5 \setminus \{v_f^{\infty}\}$.

Now, the valuations in $V_{\reg}$ are totally ordered, and $w_{5/4}$
does not satisfy the removability criterion of
Definition~\ref{Dremovability}(a), so $V_{\reg}' = V_{\reg}$.  Since
$v_0$ has exactly $1$ neighbor in  $V_{\reg}'$, we are in case
(3)(iii) of Theorem~\ref{Tminnormalcrossingsbase}.   The set $S$
of Definition~\ref{DS} contains $v_{1/2}$, which
satisfies properties (i) and (ii) of Proposition~\ref{PinftySNC}.  So
we are in case (3)(iii) of Theorem~\ref{Tminnormalcrossingsbase}, and
Proposition~\ref{Pminimalcontraction} applies.  Now,
$V_{\min}'$ in Proposition~\ref{Pminimalcontraction} is $\{v_{1/2},
w_{5/4}\}$.  Furthermore, $v_{1/2}$ satisfies parts (i), (ii), and (iii) of
Lemma~\ref{Lplaceholder} (in the notation there, $d = 8$, $e_w = 4$,
$a = 1$, and $w(f) = 6$).  So by
Proposition~\ref{Pminimalcontraction}, $V_{\min} = \{w_{5/4}\}$.

In fact, one can calculate that the normalization of the $V_{\min}$-model of
$\proj^1_K$ in $K(X)$ is generically unramified above the special
fiber (since $w_{5/4}(f) = 6 \in 8\Gamma_{w_{5/4}}$), and its special
fiber consists of two irreducible components, meeting transversely
above the standard $\infty$-specialization.
\end{example}

\begin{example}\label{Ev0contraction}
Consider the $\ints/6$-cover of $\proj^1_K$ given birationally by $y^6 = \pi_K(t^3
- \pi_K)((t-1)^3 - \pi_K)$, where $6 \nmid \chara k$.  As in the
previous examples, one can show that $V_{\reg} = V_{\reg}' = \{v_0, v,
v'\}$, where $v = [v_0, v_1(t) = 1/3]$ and $v' = [v_0, v'_1(t-1) =
1/3]$.  Now, $v_0$ is the only valuation in $V_{\reg}'$ satisfying
condition (i) of Proposition~\ref{PinftySNC}, so
we are in case (3)(ii) of Theorem~\ref{Tminnormalcrossingsbase} and Proposition~\ref{Pcontractiontoinftycrossing} applies.  So we
check the condition of Proposition~\ref{Pinftycrossingregular} for $v$
and $v'$.  We have $d = 6$, $\delta = \delta' = 3$, $a = 1$, $r = 1$,
and $\mu = \mu' = 1/3$.  The condition of
Proposition~\ref{Pinftycrossingregular} is equivalent to $1/3 > 0$
being a $3$-path, which it is.  So by
Proposition~\ref{Pcontractiontoinftycrossing}, $v_0$ is removable from
$V_{\reg}'$, and $V_{\min} = \{v, v'\}$.
\end{example}

% Let's switch to bibtex starting with this paper!


\begin{bibdiv}
\begin{biblist}
% \bib{BW_Glasgow}{article}{
%   author={Bouw, Irene I.},
%   author={Wewers, Stefan},
%   title={Computing $L$-functions and semistable reduction of superelliptic
%   curves},
%   journal={Glasg. Math. J.},
%   volume={59},
%   date={2017},
%   number={1},
%   pages={77--108},
%   issn={0017-0895},
% %  review={\MR{3576328}},
% % doi={10.1017/S0017089516000057},
% }

\bib{AM}{book}{
   author={Atiyah, M. F.},
   author={Macdonald, I. G.},
   title={Introduction to commutative algebra},
   series={Addison-Wesley Series in Mathematics},
   edition={Student economy edition},
   note={For the 1969 original see [ MR0242802]},
   publisher={Westview Press, Boulder, CO},
   date={2016},
   pages={ix+128},
   isbn={978-0-8133-5018-9},
   isbn={0-201-00361-9},
   isbn={0-201-40751-5},
   review={\MR{3525784}},
}
	
	
\bib{Betts}{article}{
   author={Betts, L. Alexander},
   title={Weight filtrations on Selmer schemes and the effective
   Chabauty-Kim method},
   journal={Compos. Math.},
   volume={159},
   date={2023},
   number={7},
   pages={1531--1605},
   issn={0010-437X},
   review={\MR{4604872}},
   doi={10.1112/S0010437X2300725X},
}

\bib{BLR}{book}{
   author={Bosch, Siegfried},
   author={L\"{u}tkebohmert, Werner},
   author={Raynaud, Michel},
   title={N\'{e}ron models},
   series={Ergebnisse der Mathematik und ihrer Grenzgebiete (3) [Results in
   Mathematics and Related Areas (3)]},
   volume={21},
   publisher={Springer-Verlag, Berlin},
   date={1990},
   pages={x+325},
   isbn={3-540-50587-3},
   review={\MR{1045822}},
   doi={10.1007/978-3-642-51438-8},
}

\bib{BW}{article}{
   author={Bouw, Irene I.},
   author={Wewers, Stefan},
   title={Computing $L$-functions and semistable reduction of superelliptic
   curves},
   journal={Glasg. Math. J.},
   volume={59},
   date={2017},
   number={1},
   pages={77--108},
   issn={0017-0895},
   review={\MR{3576328}},
   doi={10.1017/S0017089516000057},
 }

 \bib{CES}{article}{
  AUTHOR = {Conrad, Brian},
  author = {Edixhoven, Bas},
  author = {Stein, William},
     TITLE = {{$J_1(p)$} has connected fibers},
   JOURNAL = {Doc. Math.},
  FJOURNAL = {Documenta Mathematica},
    VOLUME = {8},
      YEAR = {2003},
     PAGES = {331--408},
      ISSN = {1431-0635},
   MRCLASS = {11G18 (11F11 14H40)},
%  MRNUMBER = {2029169},
MRREVIEWER = {Alessandra Bertapelle},
}

\bib{DM}{article}{
   author={Deligne, P.},
   author={Mumford, D.},
   title={The irreducibility of the space of curves of given genus},
   journal={Inst. Hautes \'Etudes Sci. Publ. Math.},
   number={36},
   date={1969},
   pages={75--109},
   issn={0073-8301},
   review={\MR{0262240}},
}

\bib{Dokchitser}{article}{
   author={Dokchitser, Tim},
   title={Models of curves over discrete valuation rings},
   journal={Duke Math. J.},
   volume={170},
   date={2021},
   number={11},
   pages={2519--2574},
   issn={0012-7094},
   review={\MR{4302549}},
   doi={10.1215/00127094-2020-0079},
 }

\bib{DDMM}{article}{
   author={Dokchitser, Tim},
   author={Dokchitser, Vladimir},
   author={Maistret, C\'eline},
   author={Morgan, Adam},
   title={Arithmetic of hyperelliptic curves over local fields},
   journal={Math. Ann.},
   volume={385},
   date={2023},
   number={3-4},
   pages={1213--1322},
   issn={0025-5831},
   review={\MR{4566695}},
   doi={10.1007/s00208-021-02319-y},
 }
 
\bib{FGMN}{article}{
  AUTHOR = {Fern\'{a}ndez, Julio},
  AUTHOR = {Gu\`ardia, Jordi},
  AUTHOR = {Montes, Jes\'{u}s},
  AUTHOR = {Nart, Enric},
     TITLE = {Residual ideals of {M}ac{L}ane valuations},
   JOURNAL = {J. Algebra},
    VOLUME = {427},
      YEAR = {2015},
     PAGES = {30--75},
      ISSN = {0021-8693},
   MRCLASS = {13A18 (11S05 12J10)},
  MRNUMBER = {3312294},
}

\bib{FN}{article}{
   author={Faraggi, Omri},
   author={Nowell, Sarah},
   title={Models of hyperelliptic curves with tame potentially semistable
   reduction},
   journal={Trans. London Math. Soc.},
   volume={7},
   date={2020},
   number={1},
   pages={49--95},
   review={\MR{4201122}},
   doi={10.1112/tlm3.12023},
 }

 \bib{GMP}{article}{
   author={Green, B.},
   author={Matignon, M.},
   author={Pop, F.},
   title={On valued function fields. III. Reductions of algebraic curves},
   note={With an appendix by E. Kani},
   journal={J. Reine Angew. Math.},
   volume={432},
   date={1992},
   pages={117--133},
   issn={0075-4102},
   review={\MR{1184762}},
   doi={10.1515/crll.1992.432.117},
 }
 
\bib{SGA1}{collection}{
   author={Grothendieck, Alexander},
   title={Rev\^etements \'etales et groupe fondamental. Fasc. I: Expos\'es 1
   \`a{} 5},
   note={Troisi\`eme \'edition, corrig\'ee;
   S\'eminaire de G\'eom\'etrie Alg\'ebrique, 1960/61},
   publisher={Institut des Hautes \'Etudes Scientifiques, Paris},
   date={1963},
   pages={iv+143 pp. (not consecutively paged) (loose errata)},
   review={\MR{0217087}},
}
% \bib{Kollar}{book}{
%    author={Koll\'{a}r, J\'{a}nos},
%    title={Lectures on resolution of singularities},
%    series={Annals of Mathematics Studies},
%    volume={166},
%    publisher={Princeton University Press, Princeton, NJ},
%    date={2007},
%    pages={vi+208},
%    isbn={978-0-691-12923-5},
%    isbn={0-691-12923-1},
%    review={\MR{2289519}},
% }

\bib{HM}{article}{
   author={Harris, Joe},
   author={Mumford, David},
   title={On the Kodaira dimension of the moduli space of curves},
   note={With an appendix by William Fulton},
   journal={Invent. Math.},
   volume={67},
   date={1982},
   number={1},
   pages={23--88},
   issn={0020-9910},
   review={\MR{0664324}},
   doi={10.1007/BF01393371},
}
\bib{KW}{article}{
  author={Kunzweiler, Sabrina},
   author = {Wewers, Stefan},
   title={Integral differential forms for superelliptic curves},
   
  eprint = {arxiv:2003.12357}, 
  year =         {2020}, 
}

	
	
% 	
% \bib{Li:cd}{article}{
%     AUTHOR = {Liu, Qing},
%      TITLE = {Conducteur et discriminant minimal de courbes de genre {$2$}},
%    JOURNAL = {Compositio Math.},
%   FJOURNAL = {Compositio Mathematica},
%     VOLUME = {94},
%       YEAR = {1994},
%     NUMBER = {1},
%      PAGES = {51--79},
%       ISSN = {0010-437X},
%    MRCLASS = {14H45 (11G20 14H25)},
% %  MRNUMBER = {1302311},
% MRREVIEWER = {Zhi Jie Chen},
% %       URL = {http://www.numdam.org.proxy01.its.virginia.edu/item?id=CM_1994__94_1_51_0},
% }

\bib{LiuBook}{book}{
    AUTHOR = {Liu, Qing},
     TITLE = {Algebraic geometry and arithmetic curves},
    SERIES = {Oxford Graduate Texts in Mathematics},
    VOLUME = {6},
      NOTE = {Translated from the French by Reinie Ern\'{e},
              Oxford Science Publications},
 PUBLISHER = {Oxford University Press, Oxford},
      YEAR = {2002},
     PAGES = {xvi+576},
      ISBN = {0-19-850284-2},
   MRCLASS = {14-01 (11G30 14A05 14A15 14Gxx 14Hxx)},
%  MRNUMBER = {1917232},
MRREVIEWER = {C\'{\i}cero Carvalho},
}


\bib{LL}{article}{
  AUTHOR = {Liu, Qing},
  AUTHOR = {Lorenzini, Dino},
     TITLE = {Models of curves and finite covers},
   JOURNAL = {Compositio Math.},
  FJOURNAL = {Compositio Mathematica},
    VOLUME = {118},
      YEAR = {1999},
    NUMBER = {1},
     PAGES = {61--102},
      ISSN = {0010-437X},
   MRCLASS = {14G20 (11G20 14H25 14H30)},
  MRNUMBER = {1705977},
MRREVIEWER = {Carlo Gasbarri},
       DOI = {10.1023/A:1001141725199},
       URL = {https://doi-org.proxy01.its.virginia.edu/10.1023/A:1001141725199},
}


\bib{MacLane}{article}{
    AUTHOR = {MacLane, Saunders},
     TITLE = {A construction for absolute values in polynomial rings},
   JOURNAL = {Trans. Amer. Math. Soc.},
  FJOURNAL = {Transactions of the American Mathematical Society},
    VOLUME = {40},
      YEAR = {1936},
    NUMBER = {3},
     PAGES = {363--395},
      ISSN = {0002-9947},
   MRCLASS = {13A18 (13F20)},
 % MRNUMBER = {1501879},
   %    DOI = {10.2307/1989629},
       URL = {https://doi-org.proxy01.its.virginia.edu/10.2307/1989629},
}

\bib{Muselli1}{article}{
   author={Muselli, Simone},
   title={Models and integral differentials of hyperelliptic curves},
   journal={Glasg. Math. J.},
   volume={66},
   date={2024},
   number={2},
   pages={382--439},
   issn={0017-0895},
   review={\MR{4818175}},
   doi={10.1017/S001708952400003X},
 }

 \bib{Muselli2}{article}{
   author={Muselli, Simone},
   title={Regular models of hyperelliptic curves},
   journal={Indag. Math. (N.S.)},
   volume={35},
   date={2024},
   number={4},
   pages={646--697},
   issn={0019-3577},
   review={\MR{4770973}},
   doi={10.1016/j.indag.2023.12.001},
}
\bib{OS1old}{article}{
   author={Obus, Andrew},
   author = {Srinivasan, Padmavathi},
   title={Conductor-discriminant inequality for hyperelliptic curves
     in odd residue characteristic},
   date={2019},
   eprint={arxiv:1910:02589v1},
}


\bib{OS1}{article}{
   author={Obus, Andrew},
   author={Srinivasan, Padmavathi},
   title={Conductor-discriminant inequality for hyperelliptic curves in odd
   residue characteristic},
   journal={Int. Math. Res. Not. IMRN},
   date={2024},
   number={9},
   pages={7343--7359},
   issn={1073-7928},
   review={\MR{4742824}},
   doi={10.1093/imrn/rnad173},
}

\bib{OShoriz}{article}{
   author={Obus, Andrew},
   author={Srinivasan, Padmavathi},
   title={Explicit minimal embedded resolutions of divisors on models of the
   projective line},
   journal={Res. Number Theory},
   volume={8},
   date={2022},
   number={2},
   pages={Paper No. 27, 27},
   issn={2522-0160},
   review={\MR{4409862}},
   doi={10.1007/s40993-022-00323-y},
 }
 
\bib{ObusWewers}{article}{
   author={Obus, Andrew},
   author = {Wewers, Stefan},
   title={Explicit resolution of weak wild arithmetic surface singularities},
   date={2018},
   eprint={arxiv:1805.09709v3},
}

\bib{Ruth}{article}{
  author = 	 {R{\"u}th, Julian},
  title = 	 {Models of curves and valuations},
  date =         {2014}, 
  note = 	 {Ph.D. Thesis, Universit\"{a}t Ulm}, 
  eprint = {https://oparu.uni-ulm.de/xmlui/handle/123456789/3302},
 doi = {10.18725/OPARU-3275}
}

% \bib{Sage}{manual}{
%       author={Developers, The~Sage},
%        title={{S}agemath, the {S}age {M}athematics {S}oftware {S}ystem
%   ({V}ersion 8.0)},
%         date={YYYY},
%         note={{\tt https://doc.sagemath.org/html/en/reference/valuations/index.html}},
% }

\bib{RuthSage}{article}{
  author = 	 {R{\"u}th, Julian},
  title = 	 {A framework for discrete valuations in Sage},
  %date =         {2014}, 
  %note = 	 {Ph.D. Thesis, Universit\"{a}t Ulm}, 
  eprint = {https://trac.sagemath.org/ticket/21869}
}

% 
% 
% \bib{saito2}{article}{
%    author={Saito, Takeshi},
%    title={Conductor, discriminant, and the Noether formula of arithmetic
%    surfaces},
%    journal={Duke Math. J.},
%    volume={57},
%    date={1988},
%    number={1},
%    pages={151--173},
%    issn={0012-7094},
% %   review={\MR{952229 (89f:14024)}},
% %   doi={10.1215/S0012-7094-88-05706-7},
% }
% 
% \bib{PadmaRational}{article}{
%   author = 	 {Srinivasan, Padmavathi},
%   title = 	 {Conductors and minimal discriminants of
%     hyperelliptic curves with rational Weierstrass points},
%   eprint = {arxiv:1508.05172v1}, 
%   year =         {2015}, 
% }

\bib{PadmaTame}{article}{
  author = 	 {Srinivasan, Padmavathi},
  title = 	 {Conductors and minimal discriminants of
    hyperelliptic curves: a comparison in the tame case},
  eprint = {arxiv:1910.08228v1}, 
  year =         {2019}, 
}

\bib{StacksProject}{article}{
author = {The Stacks Project Authors},
title = {The Stacks Project},
eprint = {https://stacks.math.columbia.edu}
}

@misc{stacks-project,
  author       = {The {Stacks project authors}},
  title        = {The Stacks project},
  howpublished = {\url{https://stacks.math.columbia.edu}},
  year         = {2023},
}

\end{biblist}
\end{bibdiv}
\end{document}